\documentclass{article}
\usepackage{amsmath,amsfonts,amsthm,amssymb,amscd,color,xcolor,mathrsfs,enumitem}
\usepackage{hyperref}

\usepackage{soul}

\colorlet{darkblue}{blue!50!black}

\hypersetup{
    colorlinks,%
    citecolor=darkblue,%
    filecolor=red,%
    linkcolor=darkblue,%
    urlcolor=magenta,%
    pdfnewwindow=true,%
    pdfstartview={FitH}
    
}

\renewcommand{\Re}{\mathop{\rm Re}\nolimits}

\newcommand{\p}{\partial}
\newcommand{\e}{\varepsilon}

\newcommand{\C}{{\mathbb C}}

\newcommand{\R}{{\mathbb R}}
\newcommand{\Z}{{\mathbb Z}}
\newcommand{\IP}{{\mathbb P}}
\newcommand{\pP}{{\mathbb P}}

\newcommand{\E}{{\mathbb E}}

\newcommand{\N}{{\mathbb N}}

\newcommand{\om}{\omega}
\newcommand{\eps}{\varepsilon}

\newcommand{\MMM}{\mathfrak{M}}

\newcommand{\XXXX}{{\mathfrak X}}

\newcommand{\bKK}{{\boldsymbol{\KK}}}

\newcommand{\eeta}{{\boldsymbol{{\eta}}}}
\newcommand{\mmu}{{\boldsymbol{{\mu}}}}
\newcommand{\zzeta}{{\boldsymbol{{\zeta}}}}

\newcommand{\GGamma}{{\boldsymbol\Gamma}}
\newcommand{\oomega}{{\boldsymbol\omega}}
\newcommand{\ssigma}{{\boldsymbol\sigma}}
\newcommand{\xxi}{{\boldsymbol\xi}}

\newcommand{\BB}{{\cal B}}

\newcommand{\DD}{{\cal D}}
\newcommand{\EE}{{\cal E}}
\newcommand{\FF}{{\cal F}}

\newcommand{\HH}{{\cal H}}
\newcommand{\KK}{{\cal K}}
\newcommand{\LL}{{\cal L}}

\newcommand{\OO}{{\cal O}}
\newcommand{\PPP}{{\mathscr P}}
\newcommand{\PP}{{\mathcal P}}

\newcommand{\RR}{{\cal R}}

\newcommand{\dd}{{\textup d}}

\newcommand{\PPPP}{{\mathfrak P}}

\newcommand{\HHHH}{{\mathfrak H}}

\newcommand{\SSS}{{\mathscr S}}

\newcommand{\nnn}{{\boldsymbol{\mathit n}}}

\newcommand{\yyy}{{\boldsymbol{\mathit y}}}

\newcommand{\be}{\begin{equation}}
\newcommand{\ee}{\end{equation}}

\newcommand{\Id}[1]{\mathop{\rm Id}_{#1}\nolimits} 
\newcommand{\lspan}{\mathop{\rm span}\nolimits}

\newcommand{\supp}{\mathop{\rm supp}\nolimits}
\newcommand{\diver}{\mathop{\rm div}\nolimits}
\newcommand{\Lip}{\mathop{\rm Lip}\nolimits}

 \newcommand{\strela}{\rightharpoonup}

\theoremstyle{plain}
\newtheorem*{mt}{Main Theorem}

\newtheorem{theorem}{Theorem}[section]

\newtheorem{lemma}[theorem]{Lemma}
\newtheorem{proposition}[theorem]{Proposition}
\newtheorem{corollary}[theorem]{Corollary}
\theoremstyle{definition}
\newtheorem{definition}[theorem]{Definition}

\theoremstyle{remark}
\newtheorem{remark}[theorem]{Remark}
\newtheorem{example}[theorem]{Example}

\numberwithin{equation}{section}

\begin{document}
\title{Mixing for dynamical systems driven by stationary noises}
\author{Sergei Kuksin\footnote{Universit\'e Paris Cit\'e and Sorbonne Universit\'e, CNRS, IMJ-PRG, 75013 Paris, France \& Peoples Friendship University of Russia  \& Steklov Mathematical Institute of Russian Academy of Sciences \& National Research University Higher School of Economics, 
 Moscow, Russia; Email: \href{mailto:Sergei.Kuksin@imj-prg.fr}{Sergei.Kuksin@imj-prg.fr}} \and Armen Shirikyan\footnote{Department of Mathematics, CY Cergy Paris University, CNRS UMR 8088, 2 avenue Adolphe Chauvin, 95302 Cergy--Pontoise, France; Email: \href{mailto:Armen.Shirikyan@cyu.fr}{Armen.Shirikyan@cyu.fr}}}
\date{\today}
\maketitle

\begin{abstract}
The paper deals with the problem of long-time asymptotic behaviour of solutions for  classes of ODEs and PDEs, perturbed by  stationary noises. The latter are not assumed to be $\delta$-correlated in time, therefore the evolution in question is not necessarily Markovian. We first prove an abstract result which implies the mixing for random dynamical systems satisfying appropriate dissipativity and controllability conditions. It is applicable to a large class of evolution equations, and we illustrate this  on the examples of a chain of anharmonic oscillators coupled to heat reservoirs, the 2d Navier--Stokes system, and a complex Ginzburg--Landau equation. Our results also apply to the general theory of random processes on the 1d lattice and allow one to get for them results related to Dobrushin's theorems on reconstructing processes via their conditional distributions. The proof is based on an iterative construction with Newton's quadratic approximation. It uses the method of Kantorovich functional, introduced  earlier by the authors in the context of randomly forced PDEs, and some ideas used by them in the Markovian case to prove mixing with the help of controllability properties of an associated system. 


\end{abstract}

\newpage
\tableofcontents

\section{Introduction}
\label{s0} 

In a Hilbert space~$H$ of finite or infinite dimension, we consider smooth discrete-time and continuous-time random dynamical systems, stirred by bounded stationary stochastic processes. Our goal is to study the long-time asymptotic behaviour of their trajectories. That is, we examine random dynamical systems (RDS) of the form 
\begin{align}
u_k &= S(u_{k-1}, \eta_k),& \quad &k\ge1;& \qquad u_0&=v;& \qquad u_k &\in H;
\label{01}\\[6pt]
\dot u(t) &= F(u(t), \eta(t)),& \quad &t\ge0;& \qquad  u(0)&=v;
&\qquad u(t) &\in H.
\label{02}
\end{align}
Here $\{\eta_k^\om, k\in\Z\}$ and $\{\eta^\om(t), t\in\R\}$ are bounded stationary processes in a Banach space~$E$, and the mappings $S: H\times E \to H$ and $F: D\times E \to H$ are $C^2$-smooth (in the sense of Fr\'echet), where~$D\subset H$ is a suitable dense subspace. We begin with reducing Equation~\eqref{02} to a discrete-time system of the form~\eqref{01}. To do it, we generalise the setting a bit and assume that the random input~$\eta^\om(t)$ entering~\eqref{02} is a bounded process in~$E$ that is stochastically $T$-periodic for some $T>0$; that is, 
\begin{equation*}\label{03}
\DD\bigl(\eta(\cdot)\bigr) = \DD\bigl(\eta(\cdot +T)\bigr)	
\end{equation*}
(so $T$ is any positive number if the process $\eta$ is stationary). We also assume that, for any $v\in H$ and $\eta \in L_{\mathrm{loc}}^2(\R_+;E)$, Equation~\eqref{02} has a unique solution~$u(t)$ in an appropriate functional class, and that the mapping 
$$
S_T: H\times L_2(0,T;E) \to H, \;\; (v, \eta_{(0,T)} )\mapsto u(T)
$$
is $C^2$-smooth, and its first- and second-order derivatives are bounded on bounded subsets. Let us define intervals $J_k = [(k-1)T, kT)$ for $k\in \Z$ and denote 
\begin{equation}\label{path}
\eta_k:=\eta\bigr|_{J_k}\in L_2(J_k;E). 
\end{equation}
We naturally identify intervals~$J_k$ with $J_T:=[0,T)$ and spaces $L_2(J_k;E)$ with $L_2(J_T;E)=:\hat E$. Then $\{\eta_k, k\in \Z\}$ becomes a stationary process in~$\hat E$, and if we set $u_k=u(kT)$ for $k\ge0$, then the sequence $\{u_k, k\ge0\}$ satisfies Equation~\eqref{01} with $S=S_T$ and~$E$ replaced by~$\hat E$. Thus, we have reduced the problem of long-time behaviour of trajectories of~\eqref{02} to that for some system~\eqref{01}. So below in Introduction, we shall talk about systems~\eqref{01}, and only at the end shall discuss some applications to nonlinear ODEs and  PDEs which can be written as systems~\eqref{02}.

\subsubsection*{System \eqref{01}}
We assume that $\KK := \supp\DD(\eta_k)$ is a compact subset of~$E$ and that some compact set $X \subset H$ is invariant for system~\eqref{01}; that is, $S(X \times \KK) \subset X$. In practice the compactness of~$\KK$ is a  mild restriction after we assumed that the process~$\eta$ is bounded, while an invariant compact set $X\subset H$ exists if~\eqref{01} is a dissipative finite-dimensional system, or if it comes from a system~\eqref{02} which is a well-posed nonlinear parabolic PDE; see the applications discussed at the end of Introduction. Our goal is to prove that system~\eqref{01} is exponentially mixing in the dual-Lipschitz distance (also known as the Kantorovich--Rubinstein distance).\footnote{\label{dln}Recall that the distance between two (Borel) probability measures~$\mu$ and~$\nu$ equals the dual-Lipschitz norm of $\mu-\nu$, defined as $\|\mu-\nu\|_L^*=\sup|\langle f,\mu\rangle-\langle f,\nu\rangle|$, where the supremum is taken over all Lipschitz functions on~$X$ such that $|f| \le1$ and $\Lip(f)\le1$. This distance metrises the weak convergence of measures; in particular, relation~\eqref{04} implies that $\DD (u_k) \strela \mu$ as $k\to\infty$; e.g., see \cite[Section~1.2.3]{KS-book}.} That is, there exists a probability  measure~$\mu$ on~$X$ such that, for any initial state $v\in X$, the trajectory~$\{u_k\}$ of~\eqref{01} satisfies 
\be\label{04}
\| \DD(u_k) -\mu\|_L^* \le C e^{-\gamma k}, \quad k\ge0, 
\ee
for some positive constants~$C$ and~$\gamma$, independent of $v$.\footnote{See below the first item of Remark~\ref{r_intro} for a discussion of relation of this notion  with the traditional mixing for stationary random processes.} If the random variables~$\{\eta_k\}$ are i.i.d., then system~\eqref{01} defines a Markov process in~$H$. The mixing in Markov systems has a long history and goes back to the thirties of the last century. Staring from the pioneering works of Kolmogorov~\cite{kolmogorov-1937} and Doeblin~\cite{doeblin-1938, doeblin-1940} various methods were developed to prove it in one or another form, similar to~\eqref{04}; see the books \cite{hasminskii2012,MT1993,borovkov1998,KS-book}. But if the random variables~$\{\eta_k\}$ and~$\{\eta(t)\}$ are not independent, then the dynamics in systems~\eqref{01} and~\eqref{02} are not Markovian, and their long time behaviour remains essentially an open problem, which was studied only for some specific models. When system~\eqref{02} is a chain of anharmonic oscillators, coupled to heat baths, Eckmann--Pillet--Rey-Bellet~\cite{EPR-1999} used a Markovian reduction to prove, in some special cases, the convergence in law of its solutions to a unique stationary measure (see also Jak\v si\'c--Pillet~\cite{JP-1997,JP-1998} for general Hamiltonian systems coupled to one heat bath). A similar idea was used by Hairer~\cite{hairer-2005-fBM} and Hairer--Ohashi~\cite{HO-2007} for some class of 
dissipative SDEs driven by a fractional Brownian motion. On the other hand, for certain 
 non-Markovian systems that arise in applied probability and are rather different from~\eqref{01} and~\eqref{02}, the mixing is known, probably, starting with Sevastianov's work~\cite{sevastjanov-1957}; see also the paper~\cite{veretennikov-2017} and the references therein. 

The goal of this work is to prove the mixing~\eqref{04} for a large class of systems~\eqref{01} in a phase space~$H$ of finite or infinite dimension (which includes the systems to which Equations~\eqref{02} may be reduced). To do that, apart from the compactness assumptions on the support~$\KK$ of the distribution of~$\eta_k$ and on the invariant set~$X$, below we  impose two more restrictions on the process~$\{\eta_k\}$ and two restrictions on the  mapping~$S$. They are simplified versions of Hypotheses~\hyperlink{GD}{(GD)}, \hyperlink{ALC}{(ALC)}, \hyperlink{DLP'}{(DLP${}'$)} and \hyperlink{SRZ}{(SRZ)}, stated in Section~\ref{s-formulation-MR}. Our Main Theorem, given below, is the {\it first general result\/} that establishes the mixing for a large class of dynamical systems of finite or infinite dimension, stirred by random forces that are not delta-correlated in time.

\subsubsection*{Assumptions on the process $\{ \eta_k\}$}

We recall that $\eeta:=\{\eta_k\}_{k\in\Z}$ is a stationary process in the Banach space~$E$ such that $\supp\DD(\eta_k)=\KK$ is a compact subset of~$E$. For $l\in \Z$, we denote by $\eeta_l=(\dots,\eta_{l-1},\eta_l)$ the {\it $l$-past\/} of the process, so that~$\eeta_l$ is a random element of~$\KK^{\Z_-}$, where  $\Z_- = \Z \setminus \N$. We provide  the product space~$E^{\Z_-}$ with the Tikhonov topology, which can be metrised on its compact subset $\bKK:=\KK^{\Z_-}$ with the help of the distance
\begin{equation}\label{distance-EE}
\dd(\xxi,\xxi')=\sum_{k=-\infty}^0 \iota^{k}\|\xi_{k}-\xi_{k}'\|_E, \quad 
\xxi=(\xi_k,k\in\Z_-),\quad \xxi'=(\xi_k',k\in\Z_-), 
\end{equation}
where $\iota>1$ is an arbitrary fixed number. For any $\xxi\in\bKK$, consider the conditional distribution $\IP\{\eta_1\in\cdot\,|\, \eeta_0 = \xxi\} = : Q(\xxi;\cdot)$ of~$\eta_1$ under the condition that the $0$-past is fixed: $\eeta_0 = \xxi$. This is a probability measure on~$E$, depending on $\xxi \in \bKK$ and supported by~$\KK$ (e.g., see \cite[Section~10.2]{dudley2002}). For any integer~$m\ge1$, the measures $Q(\xxi;\cdot)$ allow one to calculate the conditional distribution $Q_m(\xxi;\Gamma)=\IP\bigl\{(\eta_1,\dots,\eta_m)\in\Gamma\,|\,\eeta_0 =\xxi\bigr\}$, where $\Gamma\in\BB(E^m)$ (see relation~\eqref{conditional-law} in Section~\ref{s-formulation-MR}). For $\xxi\in\bKK$, let $\{ \eta_j(\xxi), j=1, 2, \dots\}$ be a process in~$\KK$ such that, for any $m\ge1$,
$$
\DD\bigl(\eta_1(\xxi),\dots,\eta_m(\xxi)\bigr)=Q_m(\xxi;\cdot). 
$$
We impose on the process $\eta$ the following two restrictions:

\begin{itemize}\sl 
	\item[\hypertarget{eta1}{$\boldsymbol{(\eta1)}$}] If $\dim E<\infty$, then the measures $\{Q(\xxi;\cdot), \xxi \in \bKK\}$ have densities, so that 
\begin{equation}\label{Q=pxi}
	Q(\xxi;\dd x) = p_\xxi(x)\,\dd x,
\end{equation}
	where the function $p:\bKK\times E \to\R$ is Lipschitz-continuous. If~$E$ is a Hilbert space with an orthonormal basis basis $\{f_l\}$, then~$\eta_k$ is representable in the form 
\begin{equation}\label{decomp}
\eta_k = \sum_{l=1}^\infty r_l \eta_k^l f_l,	
\end{equation}
where the numbers $r_l>0$ are such that $\sum |r_l| < \infty$, and $\{\eta_k^l, k\in\Z\}$  are independent real-valued stationary processes $($indexed by the integer $l\ge1)$ such that $|\eta_k^l|\le1$, and the projections $Q^l(\xxi;\cdot)$ of the conditional measure~$Q(\xxi;\cdot)$ to vector span of~$f_l$ are representable in the form~\eqref{Q=pxi} with some Lipschitz-continuous functions~$p_\xxi^l(x)$. Moreover, 
$$
\lim_{\xxi'\to\xxi}\sum_{l=1}^\infty\|p_{\xxi'}^l-p_{\xxi}^l\|_{L^1}=0. 
$$
\end{itemize}
For a general form of this condition when~$E$ is a Banach space, see Hypothesis~\hyperlink{DLP'}{(DLP$'$)} in Section~\ref{s-formulation-MR}. 

The second assumption on the process~$\eta$ requires that it stays close to zero with positive probability, for every past:

\begin{itemize}\sl 
	\item[\hypertarget{eta2}{$\boldsymbol{(\eta2)}$}] 
For any $n\in\N$ and $\delta>0$ there exist $s\in\N$ and $\eps>0$ such that
$$
\pP\{ \| \eta_j(\xxi)\| < \delta, \;j=s+1, \dots, s+n\} \ge \eps
\quad\mbox{for every $\xxi\in \bKK$}. 
$$
\end{itemize}

Crucial assumption \hyperlink{eta1}{$(\eta1)$} means not only  regularity of the conditional distributions of the process~$\eta$,  but also that  it forgets the past exponentially fast. Indeed, definition~\eqref{distance-EE} of the distance on~$\bKK$ implies that~$p_\xxi(x)$ regarded as a function of $\xi_{-l}$ is Lipschitz-continuous with a constant of order~$\iota^{-l}$. Moreover, in Section~\ref{convergence-condlaw}, we show that the results of our work imply that every process satisfying~\hyperlink{eta1}{$(\eta1)$} and~\hyperlink{eta2}{$(\eta2)$} is mixing in some sense that is weaker than it is customary in ergodic theory, and which we call {\it feeble mixing\/}. 

Condition \hyperlink{eta1}{$(\eta1)$} is close to those used by Dobrushin in his work~\cite{dobrushin-1968, dobrushin-1970} on reconstructing the distribution of a lattice random field from its conditional distributions. The results of the above papers immediately imply that every family of measures $\{ Q(\xxi;\cdot), \xxi\in \KK\}$ as in \hyperlink{eta1}{$(\eta1)$} is a system of conditional distributions of a discrete-time stationary process~$\eta$. However, they do not ensure the uniqueness of the law of~$\eta$.\footnote{Dobrushin addressed the uniqueness problem in his later works devoted to Gibbs systems (e.g., see~\cite{dobrushin-1974}). To the best of our knowledge, those results were not extended to random processes.} In this regard, let us note that, in Theorem~\ref{p-stationary-convergence}, we derive from our results that, if in addition to~\hyperlink{eta1}{$(\eta1)$} a family of measures~$\{Q(\xxi;\cdot)\}$ satisfies~\hyperlink{eta2}{$(\eta2)$}, then the law of the process~$\eta$ is defined uniquely. 

Let us note that Hypotheses~\hyperlink{eta1}{$(\eta1)$} and~\hyperlink{eta2}{$(\eta2)$} are independent from the restrictions imposed below on the mapping~$S$. They specify a natural class of bounded stationary processes which suits well to establish the mixing for various systems of the form~\eqref{01}. 

\subsubsection*{Assumptions on the mapping $S$}  
We recall that the map $S:H\times E \to H$ is $C^2$-smooth and bounded on bounded 
sets, together with its derivatives of the first and second order. We assume in addition that it possesses the following two properties: 

\begin{itemize}\sl 
	\item[\hypertarget{S1}{\bf(S1)}] 
There is $k\ge1$ and a number $a\in (0,1)$ such that the trajectory of the free 
system \eqref{01} with $\eta_1=\eta_2=\dots=0$ satisfies the inequality
$$
\| u_k\| \le a \|v\| \quad \mbox{for all $v\in H$}. 
$$
\end{itemize}

Assumption~\hyperlink{S1}{$(S1)$} implies, in particular, that the origin is a stable equilibrium for the free system. We state the second assumption in a weaker and simplified form, sufficient for the validity of the main theorem, referring the reader to~\hyperlink{ALC}{(ALC)} in Section~\ref{s-formulation-MR} for its complete version. 

\begin{itemize}\sl 
	\item[\hypertarget{S2}{\bf(S2)}] 
	There exist a Hilbert space~$V$, compactly and densely embedded into~$H$, and an open set $O \subset H\times E$ containing $X\times \KK$ such that~$S$ defines a $C^1$-smooth mapping from~$H\times E$ to~$V$, and for any $(u,\eta) \in O$ the linear operator $D_\eta S(u,\eta): E\to H$ has a dense image.
\end{itemize} 

Note that if $\dim H<\infty$, then Assumption~\hyperlink{S2}{(S2)} means that, for  $(u,\eta) \in O$, the mapping $D_\eta S(u,\eta): E\to H$ is surjective. Both in  finite- and infinite-dimensional cases, this type of properties of system~\eqref{01} often appears in the control theory and may be verified by well-known  arguments; see examples in Section~\ref{s-applications}. 

\subsubsection*{The main result}  
The statement below is a simplified version of the main result of this paper, proved in Sections~\ref{s-RDS} and~\ref{s-MR}:
 
\begin{mt}
Let us assume that Hypotheses~\hyperlink{eta1}{$(\eta1)$}, \hyperlink{eta2}{$(\eta2)$}, \hyperlink{S1}{\rm(S1)}, and~\hyperlink{S2}{\rm(S2)} are fulfilled. Then there is a Borel probability measure~$\mu$ on~$H$ such that $\supp\mu\subset X$ and, for any initial condition~$v$ which is a random variable in~$X$ independent from $\{\eta_k, k\ge1\}$, inequality~\eqref{04} holds. 
\end{mt}

In Section~\ref{s-remarks},  we show that~$\mu$ is a stationary measure for system~\eqref{01} in the sense that there is an $X$-valued stationary process $\{\hat u_k, k\ge0\}$, which is a {\it weak solution\/} of the equation in~\eqref{01}\,\footnote{See Section~\ref{s-RDS} for a definition of this notion.} such that $\DD(\hat u_k) \equiv \mu$. We also prove that the attraction property for~$\mu$  established in the Main Theorem extends to the convergence in distribution of the whole trajectories: if $v$ is as in the theorem and $\{u_k\}$ is the corresponding trajectory of~\eqref{01}, then for any $s\in\N$ 
\be\label{06}
 \DD(u_k, \dots, u_{k+s}) \strela 
 \DD(\hat u_0, \dots, \hat u_{s}) =: \mu^{s+1}\quad \text{as} \quad k\to\infty. 
\ee
 
\begin{remark}\label{r_intro}%
\begin{enumerate}[leftmargin=*]
	\item Let $\{\hat u_k\}$ be a weak stationary solution of~\eqref{01} and let~$\FF_0$ be the $\sigma$-algebra generated by the initial state~$\hat u_0$. Take any $s\in\N$ and any bounded continuous function $g:X^{s+1}\to\R$. Then, by~\eqref{06}, we have 
 $$
 \E \bigl\{g(\hat u_k, \dots, \hat u_{k+s})\,\big|\,\FF_0 \bigr\}
 =\int_{E^{s+1}} g(\vec w) \mu^{s+1}(\dd\vec{w})+o(1)\quad\mbox{as $k\to\infty$}.
 $$
It follows that if $f:X\to\R$ is another bounded continuous function, then 
\begin{equation*}\label{07}
 \E\bigl(f(\hat u_0) g(\hat u_k, \dots, \hat u_{k+s}) \bigr)
 =\E f(\hat u_0)\,\E g(\hat u_k, \dots, \hat u_{k+s})+o(1)\quad\mbox{as $k\to\infty$}.
\end{equation*}
 If this relation was valid for the indicator functions~$f$ and~$g$ of all Borel sets in~$X$ and~$X^{s+1}$, then the stationary process $\{ \hat u_k\}$ would have been mixing in the sense of classical ergodic theory. Thus, the convergence relations~\eqref{04} and~\eqref{06} can be characterised as a {\it mixing with a smaller set of observables\/}. This type of feeble mixing has been actively studied over the last thirty years; e.g., see~\cite{DDLLLP2007}. It is known, for example, that similarly to  the classical mixing, it implies both SLLN and CLT for the process $\{ \hat u_k\}$; see~\cite{DDLLLP2007} and Section~4.1 in~\cite{KS-book}. 
 
 \item
In fact, our proofs do not use the stationarity of the process~$\{\eta_k\}$ entering~\eqref{01}. They also apply in the case when the conditional distribution $Q(\xxi;\cdot)$ of~$\eta_{l+1}$, given the past $\eeta_l=\xxi$, depends on~$l$, but the bounds on characteristics of decompositions~\eqref{Q=pxi} and~\eqref{decomp} in~\hyperlink{eta1}{$(\eta1)$} and the choice of constants~$s$ and~$\e$ in~\hyperlink{eta2}{$(\eta2)$} may be made uniform in~$l$. 
 
	\item If $\{\eta_k\}$ are independent random variables, then the Main Theorem is close to the result of paper~\cite{KZ-spa2020} (which is a simplified version of the work~\cite{KNS-gafa2020} in which a stronger statement is proved). 

	\item In the case when the phase space~$H$ is finite-dimensional, convergence~\eqref{04} can be strengthen to a similar inequality for the total variation metric. This result can be derived from the Main Theorem and is presented in the subsequent article~\cite{KS-TV-2025}.

	\item When applied to systems~\eqref{02}, our results allow one to treat nonlinear PDEs which depend nonlinearly on stationary processes. This is in contrast with the theory of stochastic PDEs, where the nonlinear PDEs are driven by random forces that as functions of time are white noises and therefore  should enter the equation linearly (since nonlinear functions of such noises are not well defined).

	\item
In the interesting and important case of stochastic PDEs with conservative polynomial nonlinearities (and white in time random forces) existing techniques permit one to establish exponential mixing only for equations with quadratic nonlinearities. They also may apply to an equation with a cubic conservative nonlinearity if it allows for very strong a priori estimates. In the latter case, the existing techniques give only mixing with an algebraic rate. The reason for this restriction comes from the Foia\c s--Prodi inequality, used to prove the mixing in all such systems.\footnote{with the exception of the stochastic Burgers equation, which possesses some very special extra properties.} Consider, for example,  paper~\cite{KN-2013}. It establishes the mixing for the cubic CGL equation~\eqref{CGL} with $\mu=0$ and any~$d$, when~$\eta(t,x)$ is a random field, smooth in~$x$ and white in~$t$.  There, a crucial step of the proof is a Foia\c s--Prodi type estimate at the bottom of p.~415, which involves quadratic exponential moments of the norms of solutions. If the nonlinearity in equation was of degree $r>3$ rather than cubic, then the estimate would have involved exponential moments of solutions of degree $r-1>2$. They are ``hopelessly unbounded'' since it is impossible to believe that norms of solutions for an SPDE with conservative nonlinearity allow for better estimates then those for norms of the Wiener process whose time-derivative drives the equation. At the same time, the Main Theorem above applies to stochastic perturbations of many well posed PDEs with arbitrary growth of their nonlinearities; cf.\ below discussion of Equation~\eqref{CGL}.

\item
In our work we assume that the stationary process $\{\eta_k\}$ in Equation \eqref{01} is exponentially mixing (see Section~\ref{s-SR-mixing}), which implies that  solutions $\{u_k\}$ also are exponentially mixing. If~$\{\eta_k\}$ is algebraically mixing with a suitable rate of mixing, then the proof still applies and shows that the solutions are mixing processes. Finding a minimal rate of mixing for the process~$\{\eta_k\}$ which allows to derive this conclusion  is a natural question which will be addressed in a subsequent work.

\item 
It was drawn to our attention by an anonymous referee that the approach of this work to establish the mixing for Equations~\eqref{01} with bounded stationary and mixing processes~$\{\eta_k\}$ has a similarity with the celebrated 
construction of mixing invariant measures for an Anosov diffeomorphism  $T:M\to M$ of a compact manifold $M$; see~\cite{sinai-1972}. Those measures are obtained as certain Gibbs measures, written in terms of a Markov partition of~$M$. Indeed, a first step in the  approach of~\cite{sinai-1972} is a representation of  an isomorphism~$T$ as a shift operator in the space~$\Omega_\Pi$ of  infinite words in the alphabet, made of the elements of the Markov partition above (the words forming~$\Omega_\Pi$ agree with the map~$T$). On~$\Omega_\Pi$, there is a natural measure~$\bar\mu$ with respect to which  the shift operator  defines   a stationary Markov process. Next, one constructs a class of Gibbs measures, defined by their potentials and the measure~$\bar\mu$, and shows that with respect  to all of them the shift operator in~$\Omega_\Pi$   is a stationary and  mixing process. 
 Exactly one of these Gibbs measures corresponds to the SRB measure for~$T$; see \cite{sinai-1972,bow}.  Further development of Gibbs measures as a tool for the theory of dynamical
systems was closely related to the progress in the problem of reconstructing random processes via their conditional 
distributions, achieved about the same time, starting the works  \cite{dobrushin-1968, dobrushin-1970}.

On the other hand, in our work we consider the space of half-infinite words $\bKK=\KK^{\Z_-}$ in the alphabet $\KK = $supp$\,\DD(\eta_k)$ (which is a compact subset of $E$).  Next, using conditional distributions  of the process $\{\eta_k\}$, 
we construct in the  product space $H\times \bKK$ a random dynamical system. The system is 
supplemented with a random initial  condition at $k=0$, depending on~$v$. Then we show that the 
 measure  in $(H\times \bKK)^{Z_+}$, defined by a random trajectory of  this system, is mapped by the natural projection ${H\times \bKK}\to H$ to the distribution of a solution for~\eqref{01}.\footnote{But this does not mean that, for each random parameter~$\omega$, a trajectory of~\eqref{01} is being lifted to a trajectory of the system in $H\times \bKK$. } The dynamics of the  system in  $H\times \bKK$  turns out to be Markovian and mixing. Moreover, it is Gibbsian in some natural sense, see~\cite{KS-cmp2000}. So the dynamics \eqref{01}, which  is a projection to~$H$ of the one 
 above,  is mixing as well. The process in $H\times \bKK$ has a stationary trajectory, which is unique in the sense of distribution. 
 The projection of  its law to~$H$ is a distribution of a unique stationary solution for the equation in~\eqref{01}.  We note that, in difference with the well understood 
SRB measures for Anosov's diffeomorphisms, our knowledge of the stationary measure for system~\eqref{01} is very poor.
\end{enumerate}
\end{remark}

\subsubsection*{About the proof}
It is well known that any random process becomes Markovian if we add to the process its past. In our case this thesis implies the following (cf.~\cite{BF-1992,EPR-1999,hairer-2005-fBM,HO-2007} and~\cite[Section~13]{borovkov1998}). Recalling that $\eeta_k = (\eta_l,l\le k) \in E^{\Z_-}$, for a trajectory $\{u_k, k\ge0\}$ of~\eqref{01} we denote 
 $$
 U_k = (u_k, \eeta_k) \in H\times E^{\Z_-}=: \HHHH.
 $$
 Then $\{U_k, k\ge0\}$ is a trajectory of the following lifting of system \eqref{01} in $H$ to a system in $\HHHH$:
\begin{align}
U_k &= \SSS(U_{k-1}, \eta_k), \quad k\ge1,\label{Sys}\\
U_0 &=(v,\eeta_0),\label{Syss}
\end{align} 
where
$$
\SSS: \HHHH\times E \to \HHHH, \qquad \bigl((u,\eeta), \xi\bigr) \mapsto \bigl( S(u, \xi), (\eeta, \xi)\bigr).
$$
Since $\eta_k\in \KK$ a.s., Equation~\eqref{Sys} defines a stochastic system in $\mathfrak X = X\times \bKK$, and trajectories of~\eqref{01} in~$X$ define trajectories of \eqref{Sys}, \eqref{Syss} in~$\mathfrak X$. The natural projection $\Pi_X: \mathfrak X\to X$ sends these trajectories of~\eqref{Sys}, \eqref{Syss} back to trajectories of~\eqref{01}. It is not hard to see how to define properly trajectories of~\eqref{Sys}  with any initial data $U_0 \in \HHHH$, and that the corresponding  dynamics in~$\HHHH$ is Markovian; see Section~\ref{s-description}.  So it remains to show that the obtained Markov process in~$\HHHH$ is exponentially mixing with some stationary measure~$\mmu$ since then \eqref{04} holds with $\mu = (\Pi_X)_*\mmu$. But the Markov process, defined by~\eqref{Sys}, is rather complicated because the phase space~$\HHHH$ is ``much bigger'' than~$H$, and the resulting process is not strong Feller. Accordingly, existing approaches do not imply the mixing property for system~\eqref{Sys}. Below, to establish the mixing, we use coupling, combined with the method of Kantorovich functional suggested in \cite{KPS-cmp2002, kuksin-ams2002} (also see in \cite{kuksin2006,KS-book}), and enrich it with some ideas from our works \cite{shirikyan-asens2015, KNS-gafa2020, KZ-spa2020,shirikyan-jems2021}.
	 
The proof goes in two steps. Firstly, in Section~\ref{s-RDS}, we prove Theorem~\ref{t-mixing-kantorovich}, which establishes the mixing for a class of Markov systems that includes~\eqref{Sys} as a particular case. The proof uses Newton's method of quadratic convergence. Due to the complexity of the situation, the method provides only exponential (rather than super-exponential) rate of convergence to a  limit:  the power of the method is used to cope with the fact that available estimates are rather weak. Secondly, in Section~\ref{s-MR}, we verify in Theorem~\ref{t-mixing-dissipation} that if Hypotheses~\hyperlink{eta1}{$(\eta1)$}, \hyperlink{eta2}{$(\eta2)$}, \hyperlink{S1}{\rm(S1)}, and~\hyperlink{S2}{\rm(S2)} hold for system \eqref{01}, then the extended system~\eqref{Sys} meets the conditions of  Theorem~\ref{t-mixing-kantorovich}. This proves the Main Theorem stated above.

\subsubsection*{Applications} 
  The Main Theorem applies to Equations~\eqref{01} obtained by the discrete-time reduction from various ODEs and PDEs depending on bounded stationary processes. To show this, we first construct in Section~\ref{s-examplesprocesses} a large class of $T$-periodic random processes $\eta(t)$, taking values in a Hilbert space~$E$ of finite or infinite dimension and having locally square-integrable trajectories, such that the corresponding path-valued stationary process~$\{\eta_l\}$ in $L^2(0,T;E)$ defined by~\eqref{path} satisfies~\hyperlink{eta1}{$(\eta1)$} and~\hyperlink{eta2}{$(\eta2)$}. Next, in Section~\ref{s-applications}, we give  various applications of  our results to equations driven by additive stationary processes possessing the above properties.

\smallskip
In Section~\ref{s-ode}, we examine a chain of~$n$ anharmonic 1d oscillators with coordinates $q_j$, momenta $p_j$ $(1\le j\le n)$, and  a smooth Hamiltonian 
\begin{equation}\label{hamiltonian}
H(p,q) = \sum_{j=1}^n \frac{p_j^2}{2}+V(q), 
\quad (p,q)=(p_1,\dots,p_n,q_1,\dots,q_n)\in\R^{2n},
\end{equation}
where $V$ is a $C^2$-smooth function. At sites~$j=1$ and $j=n$, the chain is coupled with heat reservoirs which affect the $1^{\text{st}}$ and $n^{\text{th}}$ particles. This is done by adding the terms $-\gamma_1 p_1 +\zeta_1(t)$ and $-\gamma_n p_n +\zeta_n(t)$ to the corresponding equations. Here $\gamma_1, \gamma_n>0$ are damping coefficients and $\zeta_1, \zeta_n$ are statistically $1$-periodic scalar random processes. Our choice of damping terms and stochastic perturbations is inspired by the non-Markovian Langevin equation studied in~\cite{JP-1997} and~\cite{EPR-1999}; see Equations~(1) and~(2.5) in those papers. As was pointed out by an anonymous referee, a physically justified choice of those terms requires that  the noise should be Gaussian and that its correlation and the dissipation kernel satisfy a fluctuation-dissipation relation.  These properties cannot hold in our case, since the dissipation has no memory whereas the noise is bounded and correlated. Nevertheless, we believe that our technique developed for the problem described above provides a relevant mathematical framework that will allow one to deal with a physically justified model of the chain of anharmonic oscillators coupled to {\it several\/} heat reservoirs.

In what follows, we assume that 
\begin{equation}\label{as-boundedness}
	|\zeta_1(t)|, |\zeta_2(t)| \le c_*\quad
	\mbox{for almost every $t\ge0$},
\end{equation} 
where $c_*>0$ is a number. Let us set $\zeta =(\zeta_1, \zeta_n)$ and denote by~$\KK$ the support of the process $\{\zeta(t), 0\le t\le1\}$ in the space  $L_2(0,1; \R^2)$. Let $v(p,q)$ be the vector field of the chain of oscillators in the case when $\zeta_1=\zeta_2=0$. Roughly speaking, we impose the two hypotheses below:   
 
\begin{description}\sl 
	\item [Global stability.]
The vector field~$v(p,q)$ has a globally asymptotically stable equilibrium at $(0,0)$ and admits a coercive\footnote{This means that $L(p,q)\to+\infty$ as $|(p,q)|\to+\infty$.} Lyapunov function $L(p,q)$ such that 
\begin{equation}\label{lyapunov-qualified}
\bigl\langle \nabla L(p,q),v(p,q)\bigr\rangle 
+ 2c_*|\nabla L(p,q)|\le -\delta L(p,q)	
\end{equation}
outside a compact set in~$\R^{2n}$, where $\langle \cdot,\cdot\rangle$ stands for the standard inner product, and~$\delta>0$ is a number. 
\end{description}
This condition implies that there is a compact set $X \subset \R^{2n}$ which, for every $\omega$, is invariant and absorbing for the inhomogeneous dynamics under the study. 

\begin{description}\sl 
	\item [Linearised controllability.]
For every initial condition at $t=0$ in $X$ and every curve  $\zeta (\cdot) \in \KK$, the system linearised at the corresponding solution is linearly controllable at $t=1$. 
\end{description}
In Section~\ref{s-ode}, we prove Theorem~\ref{t-CO}, which shows that the above properties are sufficient for the system in question to be mixing. We also provide an example when those conditions are satisfied. 

It should be noted that chains and networks of anharmonic oscillators and rotators have been the focus of intensive research for the last thirty years; see the papers~\cite{EPR-1999,EH-2000,CE-2016,CEHR-2018,raquepas-2019,NR-2021} and the references therein. The problem in question consists of a system of ODEs driven by Brownian or Poissonian noises, which are {\it unbounded\/} processes in time. In those works, a sufficient condition for mixing is the existence of a {\it Lyapunov function in the stochastic sense\/} and the {\it H\"ormander bracket condition\/}. In our setting, we deal with a bounded noise and require the existence of a {\it Lyapunov function in the deterministic sense\/} (which may be easier to construct; see~\cite{DLS-2024}) and {\it controllability of a family of linearised equation\/}, which is different from, but related to the H\"ormander condition. 

\smallskip
Next, in Section~\ref{s-nse}, we deal with the 2d Navier--Stokes system, which is considered  in a smooth bounded domain $D\subset\R^2$ and is subject to an external random force:
\begin{equation}\label{NS}
	\p_tu+\langle u,\nabla\rangle u-\nu\Delta u+\nabla p=\eta(t,x), \quad\diver u=0, \quad x\in D. 
\end{equation}
Here $u=(u_1,u_2)$ and~$p$ are unknown velocity field and pressure, $\nu>0$ is the viscosity, and~$\eta$ is a random force, described below. Equations~\eqref{NS} are supplemented with the no-slip boundary condition 
\begin{equation}\label{non-slip}
u\bigr|_{\p D}=0	
\end{equation}
and with an initial condition for the velocity
\begin{equation}\label{NS-IC}
u(0,x)=u_0(x).	
\end{equation}
Let us denote by~$H$ the standard space of square-integrable divergence-free vector field on~$D$ whose normal component vanishes on the boundary~$\p D$ (see~\eqref{space-H}). Then, under mild hypotheses on~$\eta$, for any $u_0\in H$, problem~\eqref{NS}, \eqref{NS-IC} has a unique $H$-valued solution $u(t)$, and our goal is to investigate its long-time asymptotics.  

Let $\{e_j, j\ge1\}$ be an orthonormal basis in~$H$, formed by the eigenfunctions of the corresponding Stokes operator. We assume that~$\eta(t)$ is a random process in~$H$ of the form
 $$
 \eta(t)=\sum_{j=1}^\infty b_j\eta^j(t)e_j, 
 $$
 where the sequence $\{ b_j>0\}$ converges to zero sufficiently fast and $\{\eta^j\}$ are i.i.d.\ statistically $1$-periodic real process such that the corresponding stationary process in $L^2(0,1;\R)$ satisfies~\hyperlink{eta1}{$(\eta1)$} and~\hyperlink{eta2}{$(\eta2)$}; see Section~\ref{ss-CTN}. In Theorem~\ref{t-NS}, we show that the Main Theorem applies to~\eqref{NS} and therefore convergence~\eqref{04} holds for the laws of its solutions. 

The Main Theorem also applies if $\eta(t)$ in~\eqref{NS} is a kick force of the form 
\begin{equation}\label{kick-force}
\eta(t,x) = \sum_{k=1}^\infty \eta_k(x) \delta(t-k), \quad \eta_k(x) = \sum_{j=1}^\infty b_j \eta_k^j e_j(x),
\end{equation}
 where the sequence of positive numbers $\{b_j\}$ converges to zero sufficiently fast, and for $j=1,2, \dots$ the real-valued random variables $\{\eta_k^j, k=1,2,\dots\}$ form i.i.d.\ bounded stationary processes. If they  satisfy Hypotheses~\hyperlink{eta1}{$(\eta1)$} and~\hyperlink{eta2}{$(\eta2)$}, then the Main Theorem applies in this situation as well. 

 \smallskip
Finally, in Section~\ref{s-cgl}, we consider the following complex Ginzburg--Landau equation in a smooth  bounded domain $D\subset\R^d$, with the Dirichlet boundary condition:
\begin{equation}\label{CGL}
 \p_tu-(\nu+i\mu)\Delta u+i|u|^{2s}u=\eta(t,x), \quad x\in D.	
\end{equation}
Here $\nu$ and~$\mu$ are positive numbers and $s\ge1$ is an integer such that $s\le 2/(d-2)$ if $d\ge3$. If the external force~$\eta$ is a statistically $1$-periodic random process in a suitable Sobolev space of functions of~$x$, then as we show in Theorem~\ref{t-cgl}, the Main Theorem also applies to the solutions of this equation.

\subsection*{Acknowledgments} 
The authors thank V.\,I.~Bogachev for pertinent remarks and an anonymous referee for useful comments and suggestions. The research of AS was supported by the \textit{CY Initiative of Excellence\/} through the grant {\it Investissements d'Avenir\/} ANR-16-IDEX-0008 and by the {\it ANR project DYNACQUS\/} through the grant ANR-24-CE40-5714-01.

\subsection*{Notation}
We write $\Z$ ($\Z_+$, $\Z_-$) for the set of (non-negative, non-positive) integers,
 denote by~$B_E(a,r)$ an open $r$-ball in a Banach space~$E$, centred at~$a$, and write~$B_E(r)$ if $a=0$. For an integer $m\ge1$, a vector $\vec{\xi}_m=(\xi_1,\dots,\xi_m)\in E^m$, and  a number $\delta>0$, we denote by $\OO_\delta(\vec{\xi}_m)$ the set $B_E(\xi_1,\delta)\times\cdots\times B_E(\xi_m,\delta)$. When~$E$ is finite dimensional, we write~$\ell$ for the Lebesgue measure on~$E$. By~$\DD(\eta)$ we denote the law of a random variable~$\eta$, and $\strela$ signifies the weak convergence of measures. 

 If~$X$ is a Polish space, then  we write $\BB(X)$ for its Borel $\sigma$-algebra  and~$\PP(X)$ for the set of probability measures on~$(X,\BB(X))$. Unless otherwise stated, spaces~$\PP(X)$ are provided with the weak topology. For $\mu_1,\mu_2\in\PP(X)$, we denote by $\|\mu_1-\mu_2\|_{\mathrm{var}}$ the total variation distance between~$\mu_1$ and~$\mu_2$. Borel-measurable maps from~$X$ to another measurable space are often called just~{\it measurable\/}, and a random field $\{\zeta^\omega_x\in X, x\in X\}$ is said to be measurable if it define a measurable map $(\omega,x)\to\zeta_x^\omega$ from~$\Omega\times X$ to~$X$. Given a subset~$Q\subset X$, we denote by~$Q^c=X\setminus Q$ its complement, by~${\mathbf1}_Q$ its indicator function, and by~$\overline Q$ its closure.
 
\section{A class of random dynamical systems}
\label{s-RDS}

\subsection{Description of the model}
\label{s-description}
Let~$\XXXX$ be a compact metric space with a distance~$\dd_\XXXX$, $E$ be a separable Banach space, and $(\Omega,\FF,\IP)$ be a complete probability space. Consider a transition probability $\{\PPP(U;\cdot)\}_{U\in\XXXX}$ from~$\XXXX$ to~$E$.\footnote{Thus, for any $U\in\XXXX$, we have a Borel probability measure $\PPP(U;\cdot)$ on~$E$ such that, for any $\Gamma\in\BB(E)$, the function $U\mapsto \PPP(U;\Gamma)$ is measurable from~$\XXXX$ to~$\R$; cf.\ \cite[Chapter~5]{parthasarathy2005}.} In what follows, $\PPP(U;\cdot)$ plays the role of the law of a random input (or noise) acting on a dynamical system in the space~$\XXXX$, and we shall always assume that it satisfies the following hypothesis:
\begin{description}
	\item [\hypertarget{DN}{(DN) Driving noise}.]\sl 
The mapping $U\mapsto \PPP(U;\cdot)$ is continuous from~$\XXXX$ to the space~$\PP(E)$ endowed with the weak topology. Moreover, there is a number $C>0$ and a compact set~$\KK\subset E$ such that $\supp \PPP(U;\cdot)\subset\KK$ for any $U\in\XXXX$, and 
\begin{equation}\label{lipcshitz-transition}
\|\PPP(U;\cdot)-\PPP(U';\cdot)\|_{\mathrm{var}}\le C\dd_\XXXX(U,U')\quad\mbox{for $U,U'\in\XXXX$}. 
\end{equation}
\end{description}
Let $\{\zeta_k^U, U\in\XXXX\}_{k\ge1}$ be a sequence of independent~$\KK$-valued measurable random fields on~$\XXXX$ defined on a complete probability space $(\widehat\Omega,\widehat\FF,\widehat\IP)$. Thus, the map $(U,\widehat\omega)\mapsto \zeta_k^U(\widehat\omega)$ is measurable from~$\XXXX\times\widehat\Omega$ to~$\KK$ for any $k\ge1$. We shall always assume that
\begin{equation}\label{(a)}
	\DD(\zeta_k^U)=\PPP(U;\cdot)\quad\mbox{\sl for any $U\in\XXXX$ and $k\ge1$}. 
\end{equation}
Let us fix a Lipschitz-continuous map $\SSS:\XXXX\times \KK\to\XXXX$ and consider the following random dynamical system (RDS) in~$\XXXX$:
\begin{equation}\label{RDS}
	U_k=\SSS(U_{k-1},\zeta_k^{U_{k-1}}), \quad k\ge1.
\end{equation}
Equation~\eqref{RDS} is supplemented with the initial condition
\begin{equation}\label{IC}
	U_0=V,
\end{equation}
where $V$ is an $\XXXX$-valued random variable that is defined on the same probability space and is independent of the family $\{\zeta_k^U, U\in\XXXX, k\ge1\}$. 

In what follows, we always assume that the underlying probability space $(\widehat\Omega,\widehat\FF,\widehat\IP)$ for~\eqref{RDS} is the {\it completion of the tensor product of countably many copies of a fixed complete probability space $(\Omega,\FF,\IP)$\/}. Moreover, writing $\widehat\omega=(\omega_k,k\ge0)$ for points of~$\widehat\Omega$, we shall assume that
\begin{equation}\label{dependence}
\mbox{\sl $V$ depends only on~$\omega_0$ and $\zeta_k^U$ depends only on~$\omega_k$ for  $k\ge1$, $U\in\XXXX$}.
\end{equation}
It is well known that, given a  transition probability~$\{\PPP(U;\cdot)\}_{U\in\XXXX}$ on~$E$ that satisfies Hypothesis~\hyperlink{(DN)}{(DN)}, there is a measurable random field $\{\zeta^U,U\in\XXXX\}$ (where for the probability space one may take  the interval $[0,1]$ endowed with the Lebesgue measure) such that $\DD(\zeta^U)=\PPP(U;\cdot)$ for any $U\in\XXXX$; see~\cite[Theorem~11.7.5]{dudley2002} and~\cite[Theorem~1.2.28]{KS-book}. Hence, setting $\zeta_k^U(\widehat\omega)=\zeta^U(\omega_k)$ for $U\in\XXXX$ and $k\ge1$, we obtain a sequence of independent random fields satisfying~\eqref{(a)}. In what follows, we write  $\{U_k(V),k\ge0\}$ for a trajectory of~\eqref{RDS}, \eqref{IC} and call it a {\it solution\/} of~\eqref{RDS}, \eqref{IC} corresponding to the probability space $(\widehat\Omega,\widehat\FF,\widehat\IP)$  and the random fields $\{\zeta_k^\cdot\}$. Very often, we shall drop a specification of the probability space and random fields. 

Let~$\{\FF_k\}_{k\ge0}$ be the natural filtration of~$(\widehat\Omega,\widehat\FF,\widehat\IP)$; that is, $\FF_k$ consists of those elements of~$\widehat\FF$ that depend only on $\oomega_k:=(\omega_0,\dots,\omega_k)$. In what follows, we call a space~$\widehat\Omega$ as above with the filtration~$\{\FF_k\}$ a {\it suitable (filtered) probability space\/} constructed from the space $(\Omega,\FF,\IP)$, always assuming~\eqref{dependence}. The following result is a consequence of~\eqref{(a)}, \eqref{dependence}, and Fubini's theorem. 
 
\begin{proposition}\label{p-markov-chain}
	The family of  trajectories  $\{ U_k(V)\}$ of~\eqref{RDS}, corresponding to all possible deterministic initial conditions~$V\in\XXXX$, form a Markov process in~$\XXXX$, corresponding to the filtration~$\{\FF_k\}$, with the continuous time-$1$ transition probability 
	\begin{equation}\label{TF}
		P_1(U;\cdot)=\SSS_*\bigl(U,\PPP(U;\cdot)\bigr), \quad U\in\XXXX.
	\end{equation}
\end{proposition}

Let us note that the right-hand side of~\eqref{TF} is a measure in~$\PP(\XXXX)$ which is the image of the measure $\PPP(U;\cdot)$ under the mapping from~$\KK$ to~$\XXXX$ taking~$\eta$ to~$\SSS(U,\eta)$. 
Due to~\eqref{TF}, the Markov semigroups defined by the process do not depend on the specific choice of a suitable probability space~$(\widehat\Omega,\widehat\FF,\widehat\IP)$ and random fields~$\{\zeta_k^\cdot\}$. Therefore, the laws of solutions~$\{U_k(V)\}$ are also independent of that choice. 

\begin{proof}[Proof of Proposition~\ref{p-markov-chain}]
By construction, the random variable~$U_k(V)$, $k\ge0$, depends only on~$\oomega_k$, so it is $\FF_k$-measurable. We now prove that, for any integers $k,m\ge0$ and any bounded measurable function $f:\XXXX\to\R$, 
\begin{equation}\label{MP-general}
	\E\{f(U_{k+m}(V))\,|\,\FF_k\}=\hat f_m(U_k(V)) \quad \mbox{$\IP$-almost surely}, 
\end{equation}
where $\hat f_m:\XXXX\to\R$ is a bounded measurable function depending only on~$f$ and~$m$. The relation is obvious for $m=0$, and a simple induction argument shows that it suffices to establish it for $m=1$. But for $m=1$ and any $\FF_k$-measurable function function $\varphi(\oomega_k)$, we have 
\begin{align*}
	\E\bigl(f(U_{k+1}(V))\varphi(\oomega_k)\bigr)&=
	\E^{\oomega_k}\Bigl[\E^{\omega_{k+1}} f\bigl(\SSS(U_k^{\oomega_k},\zeta_{k+1}^{U_k^{\oomega_k}}(\omega_{k+1})\bigl)\bigr)\varphi(\oomega_k)\Bigr]\\
	&=\E^{\oomega_k}\Bigl[\varphi(\oomega_k)\,\E f\bigl(\SSS(U_k^{\oomega_k},\zeta_1^{U_k^{\oomega_k}}\bigl)\bigr)\Bigr],
\end{align*}
since the law of~$\zeta_k^U$ does not depend on~$k$. Thus, 
$$
\E\{f(U_{k+1}(V))\,|\,\FF_k\}=\E f\bigl(\SSS(x,\zeta_{1}^x)\bigr)\bigr|_{x=U_k}.
$$
It remains to note that the map $x\mapsto \E f(\SSS(x,\zeta_{1}^x))$  defines a bounded measurable function~$\hat f_1(x)$, depending only on~$f$, so~\eqref{MP-general} holds for $m=1$. Relation~\eqref{TF} follows obviously from the last equality. 
\end{proof}

In what follows, we denote by $\{P_k(U;\cdot), k\in\Z_+, U\in\XXXX\}$ the transition probability for the Markov process associated with the RDS~\eqref{RDS} (it is given by~\eqref{TF} for $k=1$). Often, when studying the long-time asymptotics of trajectories for~\eqref{RDS}, we need to replace the underlying probability space~$\widehat\Omega$ with another {\it suitable\/} probability space. This will not change the Markovian character of the dynamics and the laws of solutions of~\eqref{RDS}. The next subsection describes an important example of a dynamical system driven by a stationary process that can be reduced to an RDS of the form~\eqref{RDS}.

\subsection{Dynamical systems with  stationary noises}
\label{ss-RDS}

Consider an RDS of the form~\eqref{01}:
\begin{equation}\label{stationary-RDS}
	u_k=S(u_{k-1},\eta_k), \quad k\ge1, \qquad u_0 =v.
\end{equation}
Here $u_k\in H$, where~$H$ is a Hilbert space, $\{\eta_k\}_{k\in\Z}$ is a stationary process that  takes values in a Banach space~$E$ and is defined on a complete probability space~$(\Omega,\FF,\IP)$, and $S:H\times E\to H$ is a continuous map. We assume that $\supp\DD(\eta_k)=:\KK$ is a compact subset of~$E$ (so $\{\eta_k\}$ also may be regarded as a process in~$\KK$), and suppose that there is a compact subset $X\subset H$ such that $S(X\times\KK)\subset X$. Then~\eqref{stationary-RDS} defines an RDS in~$X$. We are interested in the long-time behaviour of trajectories of this system. 

In what follows, we distinguish between {\it weak\/} and {\it strong\/} solutions of~\eqref{stationary-RDS}: strong solutions are those that are constructed by relations~\eqref{stationary-RDS}, and will be denoted $\{u_k(v),k=0,1,\dots\}$, whereas weak solutions are, as usual, trajectories~$\{\tilde u_k\}$ of system~\eqref{stationary-RDS}, in which~$\{\eta_k\}$ and~$v$ are replaced with another process~$\{\tilde\eta_k\}$ and a random variable~$\tilde v$ having the same joint distribution and possibly defined on a different probability space. All weak solutions have the same distribution. This is in difference with system~\eqref{RDS}, where the random fields~$\{\zeta_k^\cdot\}$ must be defined on some suitable probability space, and~$\zeta_k^\cdot$ must be $\FF_k$-measurable. Accordingly, a weak solution of the equation in~\eqref{stationary-RDS} is that of problem~\eqref{stationary-RDS} with some~$v$. Weak solutions of~\eqref{stationary-RDS} will be denoted as~$(\{\tilde u_k\},\{\tilde\eta_k\})$ or, for short, as~$\{\tilde u_k\}$.

As before, we provide the space~$E^{\Z_-}$ with the Tikhonov topology whose restriction to the compact subset $\bKK:=\KK^{\Z_-}$ is metrised by distance~\eqref{distance-EE}. We define the measure 
\begin{equation}\label{sigma-etak}
\sigma:=\DD(\eeta_-), \quad \eeta_-=\{\eta_k,k\in\Z_-\},	
\end{equation}
which is an element of $\PP(\bKK)$, and denote by $\EE\subset \bKK$ its support. 


For any $\xxi\in\bKK$, denote by $Q(\xxi;\cdot)$ the conditional law $\IP\{\eta_1\in\cdot\,|\,\eeta_-=\xxi\}$ of~$\eta_1$ given the past $\{\eta_k=\xi_k\}_{k\in\Z_-}$; see~\cite[Theorem~10.2.1]{dudley2002}. It is uniquely defined up to~$\xxi$'s from a negligible set  with respect to the measure~$\sigma$. From now on we assume in addition that the process~$\{\eta_k\}$ in Eq.~\eqref{stationary-RDS}satisfies the following continuity condition, similar to the strong Feller property for Markov processes.\footnote{We emphasise that this ``strong Feller'' holds for the noise, rather than the process defined by~\eqref{stationary-RDS}.} 
\begin{itemize}
	\item [\hypertarget{Fel}{\bf(SF)}] \sl The conditional measure $Q(\xxi;\cdot)$ can be chosen to be a continuous mapping from~$\EE$ to~$\PP(\KK)\subset\PP(E)$. Moreover, there is $C>0$ such that 
\begin{equation*}\label{LP-measures}
\|Q(\xxi;\cdot)-Q(\xxi';\cdot)\|_{\mathrm{var}}\le C\dd(\xxi,\xxi')\quad
	\mbox{for any $\xxi,\xxi'\in\EE$}. 
\end{equation*}
\end{itemize}
This condition allows us to (uniquely) define $Q(\xxi;\cdot)$ for $\xxi\in\EE$ and  regard it as a transition probability, continuous in~$\xxi$ (see~\cite[Section~5.1]{parthasarathy2005}).

We now fix a specific realisation of the process~$\{\eta_k\}$ in~\eqref{stationary-RDS}, convenient for the purposes of this work. We construct it as a process defined on some suitable probability space $(\widehat\Omega,\widehat\FF,\widehat\IP)$, where $\widehat\Omega=\{(\omega_0,\omega_1,\dots)\}$ with $\omega_j\in\Omega$ (see Section~\ref{s-description}). For $k\le0$, the random variables~$\eta_k$ depends on~$\omega_0$ as for the original process~$\eta_k$ in~\eqref{stationary-RDS}, while for $k\ge1$ we take $\eta_k=\eta_k(\oomega_k)=\eta_k(\oomega_{k-1},\omega_k)$ with $\oomega_k:=(\omega_0,\dots,\omega_k)$, where for any $\oomega_{k-1}\in \Omega^k$ the mapping $\omega\mapsto\eta_k(\oomega_{k-1},\omega)$ is a random variable $\omega\mapsto \zeta_k^{\eeta_{k-1}}(\omega)$, with law equal to $Q(\eeta_{k-1};\cdot)$, such that $\zeta_k^\eeta(\omega)$ is a measurable random  field, for each $k\ge1$. Here $\bKK\ni\eeta_k:=(\eta_l,l\le k)$. An induction argument in~$k\ge0$ readily shows that this modification of the process~$\{\eta_k\}$ is distributed as the original one. 

Let us introduce the product space $\XXXX=X\times\EE$ endowed with the metric 
\begin{equation}\label{metric-X}
\dd_\XXXX(U,U')=L\,\|v-v'\|_H+\dd(\xxi,\xxi'),
\end{equation}
where $L\ge1$ is a parameter specified later, and denote the natural projections to the two components by 
\begin{equation}\label{nat-proj}
	\Pi_X:\XXXX\to X, \quad \Pi_\EE:\XXXX\to\EE.
\end{equation}
We consider a map $\SSS:\XXXX\times \KK\to\XXXX$ defined by the relation
\begin{equation}\label{S-stationary}
\SSS(U,\eta)=\bigl(S(v,\eta),(\xxi,\eta)\bigr), 	
\end{equation}
where $U=(v,\xxi)\in\XXXX$ and $\eta\in \KK$. Finally, we introduce an RDS in~$\XXXX$ by 
\begin{equation}\label{RDS-example}
	U_k=\SSS(U_{k-1},\zeta_k^{\xxi_{k-1}}), \quad k\ge1,
\end{equation}
where $U_k=(v_k,\xxi_k)$. The system~\eqref{RDS-example} is supplemented with the initial condition 
\begin{equation*}\label{in_cond}
	U_0=V:=(v,\xxi), 
\end{equation*}
where $(v,\xxi)$ is a random variable depending only on~$\omega_0$. We thus obtain a special case of system~\eqref{RDS}. Condition~\hyperlink{Fel}{(SF)} implies \hyperlink{DN}{(DN)} since $\DD(\zeta_k^{\xxi}) = Q(\xxi; \cdot)$,
so by  Proposition~\ref{p-markov-chain} system \eqref{RDS-example} defines a Markov process in the extended phase space~$\XXXX$. 

Let $\{u_k=u_k(v)\}$ be a trajectory of the RDS~\eqref{stationary-RDS}. For $k\ge0$, we denote $\widehat U_k=(u_k,\eeta_k)$, where $\eeta_k=(\eta_l,l\le k)$. By the construction of the version of the process~$\{\eta_k\}$ described above, for $k\ge0$ we have $\eta_k=\eta_k(\oomega_k)=\eta_k(\oomega_{k-1},\omega_k)$ (so also $\eeta_k=\eeta_k(\oomega_k)$), where $\eta_k(\oomega_{k-1},\cdot)$ is the  random variable $\omega_k\mapsto \zeta_k^{\eeta_{k-1}(\oomega_{k-1})}(\omega_k)$ with $\DD(\zeta_k^\xxi)=Q(\xxi;\cdot)$. It follows that 
$$
\SSS\bigl(\widehat U_k,\zeta_k^{\eeta_{k-1}}\bigr)
=\bigl(S(u_{k-1},\eta_k),(\eeta_{k-1},\eta_k)\bigr)=\widehat U_k.
$$
So $\{\widehat U_k,k\ge0\}$ is a trajectory of~\eqref{RDS-example} with $\widehat U_0=(v_0,\eeta_0)$. We have thus proved the following result. 

\begin{lemma}\label{l-reduction}
	Let an $\FF_0$-measurable random initial condition $V=(v,\xxi)$ be such that $v\in H$ and~$\xxi$ is a $\bKK$-valued random variable whose law is equal to~$\sigma$ (see \eqref{sigma-etak}). 
	Let $\{U_k=(v_k,\xi_k),k\ge0\}$ be a solution of~\eqref{RDS-example}, \eqref{IC}, where for $U=(v,\xxi)\in\XXXX$ we have $\DD(\zeta_k^U)=Q(\xxi;\cdot)$ {\rm(}so $\PPP((v,\xxi),\cdot)=Q(\xxi;\cdot)${\rm)}. Then, for any $k\ge0$, the distribution of the vector $[U_0,\dots,U_k]$, $U_j=(v_j, \xxi_j)$, 
	 is such that 
	\begin{gather}\label{new-relation}
		\DD([v_0,\dots,v_k])=\DD\bigl([u_0(v),\dots,u_k(v)]\bigr), \quad \DD(\xxi_k)=\sigma.
	\end{gather}
\end{lemma}
Since all weak solutions have the same law, relation~\eqref{new-relation} remains true if we replace there the strong solution~$\{u_k(v)\}$ with any weak solution of~\eqref{stationary-RDS}. 

In conclusion, let us note that both spaces~$H$ and~$E$ can be assumed to be separable. Indeed, denoting by~$\widetilde H$ and~$\widetilde E$ the closures of the vector spans of~$X$ and~$\KK$, respectively, we define the map 
$$
\widetilde S:\widetilde H\times\widetilde E\to \widetilde H, 
\quad (u,\eta)\mapsto{\mathsf P}_{\widetilde H}S(u,\eta),
$$
where ${\mathsf P}_{\widetilde H}:H\to H$ stands for the orthogonal projection to~$\widetilde H$. Since the image of~$X\times\KK$ under~$S$ is contained in~$X$, we see that~$\widetilde S$ and~$S$ coincide on~$X\times\KK$. Thus, replacement of~$S$ with~$\widetilde S$ in~\eqref{stationary-RDS} will not change the trajectories, and we obtain an RDS of the same form with separable spaces.

\subsection{A criterion for mixing}
\label{s-criterion}
We now prove a result providing a sufficient condition for the property of exponential mixing for the RDS~\eqref{RDS}. To this end, we impose  the two hypotheses below on the map~$\SSS$ and the transition probabilities~$\PPP(U;\cdot)$
\begin{description}
\item [\hypertarget{GCP}{(GCP)} Global controllability to points.]\sl For any $\e>0$, there is an integer $m\ge1$ and a point $\widehat U\in\XXXX$ {\rm(}both depending on~$\e)$ for which the following property holds: for any $U\in\XXXX$ there are $\xi_1,\dots,\xi_m\in \KK$ such that the points $U_0,\dots,U_m\in\XXXX$ defined by $U_0=U$ and $U_{k}=\SSS(U_{k-1}, \xi_k)$ for $1\le k\le m$ satisfy the relations
\begin{gather}\label{xi_k-support}
	\xi_k\in\supp \PPP (U_{k-1}, \cdot)\quad\mbox{for $1\le k\le m$}, \qquad \dd_\XXXX(U_m,\widehat U)<\e.
\end{gather}
\end{description}

Let us set
\begin{align}\label{D}
D_\delta&=\{(U,U')\in\XXXX\times\XXXX:\dd_\XXXX(U,U') \le \delta\}
\end{align}
and define the Lipschitz seminorm of a map~$\varPhi:\KK\to E$  by the relation
$$
\Lip_\xi(\varPhi):=\sup_{\xi_1,\xi_2}\frac{\|\varPhi(\xi_1)-\varPhi(\xi_2)\|_E}{\|\xi_1-\xi_2\|_E},
$$
where the supremum is taken over all $\xi_1,\xi_2\in\KK$ such that $0<\|\xi_1-\xi_2\|_E\le1$.

\begin{description}
	\item [\hypertarget{LAC}{(LAC) Local approximate controllability}.]\sl 
There are positive numbers $C_*$, $\delta$, and $q<1$, a finite-dimensional subspace $F\subset E$, and a continuous map 
$$
\varPhi:D_\delta\times \KK\to F 
$$
such that, for any $(U,U')\in D_\delta$, 
\begin{align}
	\sup_{\xi\in \KK}\|\varPhi(U,U',\xi)\|_E+\Lip_\xi\bigl(\varPhi(U,U',\cdot)\bigr)&\le C_*\dd_\XXXX(U,U'), \label{Phi-bound}\\
	\sup_{\xi\in \KK}\dd_\XXXX\bigl(\SSS(U,\xi),\SSS(U',\xi+\varPhi(U,U',\xi))\bigr) &\le q\,\dd_\XXXX(U,U'). \label{Phi-squeezing}
\end{align}
\end{description}
We shall also need a decomposability and regularity hypothesis on the driving noise in~\eqref{RDS}. We recall that a closed subspace $F\subset E$ is said to be {\it complemented\/} if there is another closed subspace~$F^\dagger\subset E$ such that $E=F\dotplus F^\dagger$. Obviously, the existence of a {\it complementary subspace\/} is equivalent to the existence of a continuous projection $\mathsf P_F:E\to E$ onto~$F$. A simple consequence of the Hahn--Banach theorem is that any finite-dimensional subspace is complemented. In what follows, if~$E$ is represented as the direct sum of closed subspaces~$F$ and~$F^\dagger$, then we denote by~$\mathsf P_F$ and~$\mathsf P_{F^\dagger}$ the associated projections. 

\begin{description}
	\item [\hypertarget{DLP}{(DLP) Decomposability and Lipschitz property}.]\sl The  subspace~$F$ entering~\hyperlink{LAC}{\rm(LAC)} possesses a complementary subspace $F^\dagger\subset E$ such that the measure $\PPP(U;\cdot)$ admits a disintegration of the form 
\begin{equation}\label{decomposition-E-E}
	\PPP (U;\dd\xi_F,\dd \xi_{F^\dagger})=\PPP_{F^\dagger}(U;\dd\xi_{F^\dagger})\PPP_F(U,\xi_{F^\dagger};\dd\xi_F), \quad U\in\XXXX,
\end{equation}
where $\PPP_{F^\dagger}(U;\cdot)$ denotes the projection of~$\PPP(U;\cdot)$ to the subspace~$F^\dagger$, and $\PPP_F(U,\xi_{F^\dagger};\cdot)\in \PP(F)$ 
stands for the regular conditional probability of~$\PPP(U;\cdot)$ when the projection to~$F^\dagger$ is fixed and is equal to~$\xi_{F^\dagger}$. Moreover, there is a function $p_F:\XXXX\times E\to\R_+$, Lipschitz continuous in both arguments, i.e.
\begin{equation}\label{Lip}
	|p_F(U,\xi)-p_F(U',\xi')|\le M\bigl(\dd_\XXXX(U,U')+\|\xi-\xi'\|_E\bigr)\mbox{  for $U, U'\in\XXXX$, $\xi,\xi'\in E$}, 
\end{equation}
such that, for any $U\in\XXXX$ and $\xi_{F^\dagger}\in F^\dagger$, we have
\begin{equation}\label{Q-F-density}
\PPP_F(U,\xi_{F^\dagger};\dd\xi_F)=p_F(U,\xi_{F^\dagger},\xi_F)\,\ell_F(\dd\xi_F),
\end{equation} 
where $\ell_F$ stands for the Lebesgue measure on~$F$. 
\end{description}

In view of the well-known formula for the total variation distance (e.g., see Proposition~1.2.7 in~\cite{KS-book}), it follows from~\eqref{Lip} and~\eqref{Q-F-density} that, for any $U,U'\in\XXXX$ and $\xi_{F^\dagger},\xi_{F^\dagger}'\in F^\dagger$, we have 
\begin{equation*}\label{Lip-U}
\|\PPP_F(U,\xi_{F^\dagger};\cdot)-\PPP_F(U',\xi_{F^\dagger}';\cdot)\|_{\mathrm{var}}\le M_1\,\bigl(\dd_\XXXX(U,U')+\|\xi_{F^\dagger}-\xi_{F^\dagger}'\|_E\bigr),
\end{equation*}
where $M_1>0$ does not depend on~$U,U'$ and~$\xi_{F^\dagger},\xi_{F^\dagger}'$. 

Recall that, by Proposition~\ref{p-markov-chain}, the trajectories of~\eqref{RDS} form a discrete-time Markov process. The following theorem is the main result of this section.

\begin{theorem}\label{t-mixing-kantorovich}
Suppose that Hypotheses~\hyperlink{DN}{\rm (DN)}, \hyperlink{GCP}{\rm (GCP)}, \hyperlink{LAC}{\rm (LAC)}, and~\hyperlink{DLP}{\rm(DLP)} are fulfilled. Then the Markov process associated with the RDS~\eqref{RDS} has a unique stationary measure $\mmu\in\PP(\XXXX)$, and there are positive numbers~$C$ and~$\gamma$ such that\,\footnote{Recall that the norm~$\|\cdot\|_L^*$ was defined in footnote~\ref{dln}.} 
	\begin{equation}\label{mixing-RDS}
		\|P_k(U;\cdot)-\mmu\|_L^*\le Ce^{-\gamma k}\quad\mbox{for any $U\in\XXXX$, $k\ge0$}. 
	\end{equation}
	Moreover, if $\{V_k,k\ge0\}$ is a stationary trajectory of this Markov process, $\DD(V_k)\equiv\mmu$, and $\mmu_s=\DD(V_1,\dots,V_s)$, then 
	\begin{equation}\label{convergence-V-s}
		\|\DD(U_k,\dots,U_{k+s})-\mmu_{s+1}\|_L^*\le C_se^{-\gamma k}
		\quad\mbox{for any $U\in\XXXX$, $k\ge0$}, 
	\end{equation}
where $\{U_k\}$ stands for the trajectory of~\eqref{RDS} issued from~$U$. 
\end{theorem}

A proof of this result is presented in the next subsection. It uses Theorem~3.1.1 of~\cite{KS-book} and some ideas from~\cite{shirikyan-asens2015,KNS-gafa2020,KZ-spa2020,shirikyan-jems2021}.

\subsection{Proof of Theorem~\ref{t-mixing-kantorovich}}
\label{s-proof-of criterion}
We begin with a description of the scheme of the proof. It is well known that, to prove~\eqref{mixing-RDS} under some minimal compactness assumption (which trivially holds now since the space~$\XXXX$ is compact), it suffices to establish that, for any probability measures $\lambda, \lambda'$ on~$\XXXX$, we have
\begin{equation}\label{mmix}
\|\PPPP_k^* \lambda - \PPPP_k^* \lambda'\|_L^* \le C e^{-\gamma k} \quad \forall\, k\ge0, 
\end{equation}
where $\PPPP_k^*:\PP(\XXXX)\to \PP(\XXXX)$ denote the Markov operators, corresponding to the kernels~$P_k(U;\Gamma)$. To establish~\eqref{mmix}, we use the method of Kantorovich functional (suggested in \cite{KPS-cmp2002,kuksin-ams2002,kuksin2006} in the context of randomly forced PDEs), following its presentation in \cite[Section~3.1.1]{KS-book}. Let $F:\XXXX\times\XXXX\to\R_+$ be a measurable symmetric function minorised by~$\dd_\XXXX$. We define the {\it Kantorovich functional associated with~$F$\/} by the formula
$$
K_F:\PP(\XXXX)\times \PP(\XXXX)\to\R_+, \quad 
K_F(\mu,\mu')=\inf_\MMM\int_{\XXXX\times\XXXX}F(V,V')\,\MMM(\dd V,\dd V'), 
$$
where the infimum is taken over all  couplings $\MMM $  of $\mu$ and $\mu'$, i.e., over all measures $\MMM \in\PP(\XXXX\times\XXXX)$  whose marginals coincide with~$\mu$ and~$\mu'$. In view of Theorem~3.1.1 of~\cite{KS-book}, relation~\eqref{mmix} (and so also~\eqref{mixing-RDS}) will be established if we find a function~$F$ satisfying the above properties, a   number $\varkappa\in(0,1)$, and an integer $s\ge1$ such that 
\begin{equation}\label{contraction}
	K_F(\PPPP_s^*\lambda,\PPPP_s^*\lambda')\le \varkappa K_F(\lambda,\lambda')
	\quad\mbox{for any $\lambda,\lambda'\in\PP(\XXXX)$}.
\end{equation}

To prove~\eqref{contraction}, it suffices to find a coupling $(\varPhi_s^\lambda,\varPhi_s^{\lambda'})$ for the pair of  measures $(\PPPP_s^*\lambda,\PPPP_s^*\lambda')$ such that 
\begin{equation}\label{new}
	\E\,F(\varPhi_s^\lambda,\varPhi_s^{\lambda'})\le \varkappa K_F(\lambda,\lambda'),
\end{equation}
since the left-hand side of~\eqref{contraction} is bounded by $\E\,F(\varPhi_s^\lambda,\varPhi_s^{\lambda'})$. The following construction of a coupling satisfying this property was suggested in the references mentioned above. 

Suppose  that, for some complete probability space $(\Omega,\FF,\IP)$, we have constructed measurable maps $\RR,\RR':\XXXX\times\XXXX\times\Omega\to\XXXX$ such that $\RR=\RR(U,U';\omega)$ and $\RR'=\RR^\prime(U,U';\omega)$ form a coupling for $(P_1(U;\cdot),P_1(U';\cdot))$ for any $U,U'\in\XXXX$, i.e.,
\begin{equation}\label{couple}
	\DD(\RR)=P_1(U;\cdot), \quad 
	\DD(\RR')=P_1(U';\cdot).
\end{equation}
Then, denoting by~$(\widehat\Omega,\widehat\FF,\widehat\IP,\FF_k)$ a suitable probability space as in Section~\ref{s-description}, constructed from the space $(\Omega,\FF,\IP)$ above and formed of points $\widehat\omega=(\omega^0,\omega^1,\dots)$ with $\omega^j\in\Omega$, and setting $\widehat\omega^k=(\omega^0,\dots,\omega^k)$, we define iterations of~$(\RR,\RR')$ by the formulas\footnote{When we write a random variable $\eta(\widehat\omega)$ as $\eta(\widehat\omega^k)$, it means that~$\eta$ depends only on the components~$\widehat\omega^k$ of a point $\widehat\omega\in\widehat\Omega$.}
\begin{equation}\label{Phi-Phi'}
\left.
\begin{aligned}
\varPhi_0(U,U';\omega^0)&=U, \quad \varPhi_0'(U,U';\omega^0)=U', \\
\varPhi_k(U,U';\widehat\omega^k)&=\RR\bigl(\varPhi_{k-1}(U,U';\widehat\omega^{k-1}),\varPhi_{k-1}'(U,U';\widehat\omega^{k-1}),\omega^k\bigr),\\
\varPhi_k'(U,U';\widehat\omega^k)&=\RR'\bigl(\varPhi_{k-1}(U,U';\widehat\omega^{k-1}),\varPhi_{k-1}'(U,U';\widehat\omega^{k-1}),\omega^k\bigr). 
\end{aligned}
\right\}	
\end{equation}
The construction readily implies that 
\begin{equation}\label{laws}
\begin{aligned}
\DD\bigl(\{\varPhi_k(U,U')\}_{k\ge0}\bigr) &= \DD\bigl(\{U_k(U)\}_{k\ge0}\bigr),\\
\DD\bigl(\{\varPhi_k'(U,U')\}_{k\ge0}\bigr) &= \DD\bigl(\{U_k(U')\}_{k\ge0}\bigr);
\end{aligned}
\end{equation}
we recall that $\{U_k(V),k\ge0\}$ is a solution of~\eqref{RDS}, \eqref{IC}. Besides, it is easy to see that 
\begin{equation}\label{Mark}
\begin{gathered}
\text{\sl the family $\bigl\{\bigl(\varPhi_k(U,U')), \varPhi'_k(U,U')\bigr), k\ge0 \bigr\}$ is a Markov}\\ 
\text{\sl  process in~$\XXXX\times \XXXX$ with respect to the filtration~$\{\FF_k\}$}.
\end{gathered}
\end{equation}
It is shown in~\cite{KS-book} that if the random variables~$\RR$ and~$\RR'$ satisfy certain properties, then for~$\lambda=\delta_U$ and $\lambda'=\delta_{U'}$ relation~\eqref{new} with $\varPhi_s^\lambda=\varPhi_s(U,U')$, $\varPhi_s^{\lambda'}=\varPhi_s'(U,U')$ holds for some $\varkappa<1$, and that this implies the validity of~\eqref{new} for any measures~$\lambda$ and~$\lambda'$, with suitable $\varPhi_s^\lambda$ and $\varPhi_s^{\lambda'}$ constructed from~$\varPhi_s$ and~$\varPhi_s'$. 

In accordance with this programme, we now construct a pair of maps~$\RR$, $\RR'$ and show that they possess the required properties. Recall that the sets~$D_\delta$ were defined by~\eqref{D}. In the following two propositions, we assume that the hypotheses of Theorem~\ref{t-mixing-kantorovich} are fulfilled. 

\begin{proposition}\label{p-constructionRR}
	There is a complete probability space $(\Omega,\FF,\IP)$, measurable maps $\RR,\RR':\XXXX\times\XXXX\times\Omega\to\XXXX$, and positive numbers~$\theta$, $N$, and $q<1$ satisfying the inequality
\begin{equation}\label{const}
N\theta<1-q,	
\end{equation}
such that~\eqref{couple} holds for any $(U,U')\in \XXXX\times\XXXX$, and the following properties are fulfilled.
\begin{description}	
	\item [\hypertarget{independence}{Independence}.] For $(U,U')\in  D_\theta^c$, the random variables $\RR(U,U')$ and $\RR'(U,U')$ are independent.
	\item [\hypertarget{squeezing}{Squeezing}.] 	
For $(U,U')\in  D_\theta$, we have 
\begin{equation}\label{P-RR}
	\IP\bigl\{\dd_\XXXX\bigl((\RR(U,U'),\RR'(U,U')\bigr)\le q\,\dd_\XXXX(U,U')\bigr\}
	\ge 1-N\dd_\XXXX(U,U').
\end{equation}
\end{description}
\end{proposition}

As was mentioned above, iterations of~$\RR$ and~$\RR'$ allows one to construct a Markov process~$(\varPhi_k,\varPhi_k')$ on a suitable filtered probability space $(\widehat\Omega,\widehat\FF,\widehat\IP,\FF_k)$. The following result establishes two key properties of that process. 

\begin{proposition}\label{p-key-estimates}
There is a number $\theta>0$ satisfying~\eqref{const} such that, for any $(U,U')\in D_\theta$, the following inequality holds for the Markov process~\eqref{Mark}{\rm:}
\begin{equation}\label{infinite-coupling}
	\widehat\IP\bigl\{\dd_\XXXX\bigl(\varPhi_k(U,U'),\varPhi_k'(U,U')\bigr)
	\le q^k\dd_\XXXX(U,U')\mbox{ for all  $k\ge0$}\bigr)\bigr\}
	\ge 1-N_1\,\dd_\XXXX(U,U'),
\end{equation}
where $N_1=N(1-q)^{-1}$. Besides, for any $\delta>0$, there exist a number $p>0$ and an integer $l\ge1$ such that, for any $(U,U')\in\XXXX\times\XXXX$, 
\begin{equation}\label{to-coupling}
	\widehat\IP\bigl\{\dd_\XXXX\bigl(\varPhi_l(U,U'),\varPhi_l'(U,U')\bigr)
	\le\delta\bigr\}\ge p.
\end{equation}
\end{proposition}

Once these two propositions are established, direct application of the results in~\cite[Section~3.2.3]{KS-book} will imply the validity of~\eqref{new} and prove Theorem~\ref{t-mixing-kantorovich}. 

\begin{proof}[Proof of Proposition~\ref{p-constructionRR}] 
In what follows, we denote  $d:= \dd_\XXXX(U,U')$. The proof is divided into three steps.

\smallskip
{\it Step~1: Estimate for a total variation distance\/}. Let~$\varPhi$ be the map entering Hypothesis~\hyperlink{LAC}{(LAC)} and let 
$$
\varPsi(U,U',\xi)=\xi+\varPhi(U,U',\xi).
$$ 
We claim that there is $\theta\in(0,\delta)$ such that, for any $(U,U')\in D_\theta$, 
\begin{equation}\label{TV-estimate}
\Delta(U,U'):=
	\bigl\|\PPP(U;\cdot)-\varPsi_*(U,U',\PPP(U;\cdot))\bigr\|_{\mathrm{var}}
	\le C_1d. 
\end{equation}
Here and below  $C_i$, $i=1,2,\dots$, stand for some positive numbers not depending on~$U$ and~$U'$. 

To prove~\eqref{TV-estimate}, note that, for any $(U,U')\in D_\delta$, the map $\varPhi(U,U',\cdot):\KK\to E$ satisfies the hypotheses mentioned in the beginning of Section~\ref{s-transformation}. In particular, in view of~\eqref{Phi-bound} and~\eqref{Phi-squeezing}, inequalities~\eqref{lipschitz-property} hold with $\varkappa=C_*d$. Applying Proposition~\ref{p-image-measure}, we see that $\varPhi(U,U',\cdot)$ must satisfy inequality~\eqref{estimate-image}, which coincides with~\eqref{TV-estimate}. 

\smallskip
{\it Step~2: Construction of~$\RR$ and~$\RR'$}. Let us fix any $(U,U')\in D_\theta$ and consider  two pairs of random variables in $E$, 
$\bigl(\zeta^U,\varPsi(U,U',\zeta^U)\bigr)$ and $\bigl(\varPsi(U,U',\zeta^U),\zeta^{U'}\bigr)$, where $\DD(\zeta^U)=\PP (U, \cdot)$ and  $\DD(\zeta^{U'}) = \PP (U', \cdot)$. We denote by~$\lambda(U,U')$ the law of the first pair and define~$\lambda'(U,U')$ to be the law of a maximal coupling for the second pair, i.e., for the pair
$$
\bigl(\,\DD(\varPsi(U,U',\zeta^U))=\varPsi_*(U,U',\PPP(U;\cdot)), \ \DD(\zeta^{U'})=\PPP(U';\cdot)\,\bigr).
$$ 
In view of Theorem~1.2.28 in~\cite{KS-book}, we can assume that~$\lambda$ and~$\lambda'$ are transition probabilities from~$D_\theta$ to $E\times E$. Then $\Pi_{2*}\lambda(U,U')=\Pi_{1*}\lambda'(U,U')$ for any $(U,U')\in D_\theta$, where $\Pi_i:E\times E\to E$, $i=1,2$ denotes the natural projection from the direct product to its $i^\text{\rm{th}}$ component. In view of the measurable version of the gluing lemma (see~\cite[Chapter~1]{villani2009} and~\cite[Corollary~3.7]{BM-2020}), there exists a triplet of $E$-valued random variables $\tilde\zeta(U,U')$, $\xi(U,U')$, $\tilde\zeta'(U,U')$, defined on the same probability space, such that the maps $\tilde\zeta,\xi,\tilde\zeta':D_\theta\times\Omega\to E$ are measurable, and for any $(U,U')\in D_\theta$, we have 
\begin{equation*}\label{noise-coupling}
	\DD\bigl(\tilde\zeta(U,U'),\xi(U,U')\bigr)=\lambda(U,U'), \quad 
	\DD\bigl(\xi(U,U'),\tilde\zeta'(U,U')\bigr)=\lambda'(U,U'). 
\end{equation*}
Enlarging,  if necessary, the probability space, we extend $\tilde\zeta(U,U')$ and $\tilde\zeta'(U,U')$ from $(U,U')\in D_\theta$  to random fields on $\XXXX\times \XXXX$ in such a way that 
\begin{equation}\label{exten}
\text{$\tilde\zeta(U,U')$ and $\tilde\zeta'(U,U') $ are independent if $(U,U')\in D_\theta^c$},	
\end{equation}
and their laws are $\PPP(U; \cdot)$ and $\PPP(U';\cdot)$,  respectively. Finally, for any  $U,U'\in\XXXX$, we set 
\begin{equation}\label{RR-definition}
	\RR(U,U',\omega)=\SSS\bigl(U,\tilde\zeta(U,U')\bigr), \quad
	\RR'(U,U',\omega)=\SSS\bigl(U',\tilde\zeta'(U,U')\bigr).
\end{equation}
This implies, in particular, that $(\RR,\RR')$ is a coupling for the measures~$P_1(U;\cdot)$ and~$P_1(U';\cdot)$.

\smallskip
{\it Step~3: Properties of~$\RR$ and~$\RR'$\/}. We claim that the maps defined by~\eqref{RR-definition} satisfy the properties of \hyperlink{independence}{Independence} and \hyperlink{squeezing}{Squeezing}. Indeed, the independence of  $\RR$ and~$\RR'$  for $(U,U')\in D_\theta^c$ follows from~\eqref{exten}, so we only need to establish~\eqref{P-RR} for $(U,U')\in D_\theta$.

Since~$\DD(\xi,\tilde\zeta')=\lambda'(U,U')$ is the law of a maximal coupling for $\DD(\varPsi(U,U',\zeta^U))$ and~$\DD(\zeta^{U'})$, it follows from~\eqref{TV-estimate} and~\eqref{lipcshitz-transition} that 
\begin{align}
&\IP\{\xi(U,U')\ne\tilde\zeta'(U,U')\}
= \bigl\|\DD(\varPsi(U,U',\zeta^U))-\DD(\zeta^{U'})\bigr\|_{\mathrm{var}}\notag\\
&\qquad\le \bigl\|\DD(\varPsi(U,U',\zeta^U))-\DD(\zeta^U)\bigr\|_{\mathrm{var}}
+\bigl\|\DD(\zeta^U)-\DD(\zeta^{U'})\bigr\|_{\mathrm{var}}\le C_1d.\label{maximal-coupling}
\end{align}
On the other hand, since the law of $(\tilde\zeta,\xi)$ coincides with that of $(\zeta^U,\varPsi(U,U',\zeta^U))$, the definition of~$\varPsi$ and inequality~\eqref{Phi-squeezing} imply that, with probability~$1$, 
$$
\dd_\XXXX\bigl(\SSS(U,\tilde\zeta(U,U')),\SSS(U',\xi(U,U'))\bigr) \le q\,{d}. 
$$
Therefore, if $\omega\in\Omega$ is such that  $\xi(U,U')=\tilde\zeta'(U,U')$, then 
$$
\dd_\XXXX\bigl(\SSS(U,\tilde\zeta(U,U')),\SSS(U',\tilde\zeta'(U,U'))\bigr) \le q\,{d}.
$$
Combining this with~\eqref{maximal-coupling} and~\eqref{RR-definition}, we arrive at~\eqref{P-RR} with $N=C_1$. 
\end{proof}

\begin{proof}[Proof of Proposition~\ref{p-key-estimates}]
{\it Step~1: Proof of~\eqref{infinite-coupling}\/}. Given an integer $k\ge1$, we define the event 
$$
G_k=\Bigl\{\dd_\XXXX\bigl(\varPhi_k(U,U'),\varPhi_k'(U,U')\bigr)\le q\,\dd_\XXXX\bigl(\varPhi_{k-1}(U,U'),\varPhi_{k-1}'(U,U')\bigr)\Bigr\}.
$$
We claim that, for any $(U,U')\in D_\theta$, 
\begin{equation}\label{PGk}
	\widehat\IP\bigl(\,\widetilde  G_k\bigr)
	\ge 1-N{d}\sum_{l=0}^{k-1}q^l=:\delta_k, \quad \text{where}\quad 
	 \widetilde  G_k = \bigcap_{l=1}^k G_l.
\end{equation}
If this inequality is proved, then the probability of the intersection~$\widetilde G$ of all $G_k$ can be minorised by 
$$
1-Nd\sum_{l=0}^\infty q^l=1-N_1 d,
$$ 
and the required inequality~\eqref{infinite-coupling} will follow from the observation that~$\widetilde G$ is a subset of  the event in the left-hand side of~\eqref{infinite-coupling}. 

To prove~\eqref{PGk}, we  argue by induction. For $k=1$, inequality~\eqref{PGk} coincides with~\eqref{P-RR}. Assuming that inequality~\eqref{PGk} holds for $k=m$, we now prove it for $k=m+1$. 

 By  \eqref{Mark} and  the Markov property, we have 
$$
\widehat\IP\bigl\{G_{m+1}\,|\,\FF_m\bigr\}=\widehat\IP\Bigl\{\dd_\XXXX\bigl(\RR(V,V'),\RR'(V,V')\bigr
)\le q\,\dd_\XXXX(V,V')\Bigr\},
$$
where $V=\varPhi_m(U,U')$ and $V'=\varPhi_m'(U,U')$. Now note that if $\widehat\omega\in \widetilde G_m$, then  $\dd_\XXXX(\varPhi_m(U,U'),\varPhi_m'(U,U'))\le q^m{d}\le \theta$, so by \eqref{P-RR}  
 the probability on the right-hand side is minorised by $1-Nq^m{d}$. Using the induction hypothesis, we thus obtain 
\begin{align*}
	\widehat\IP\bigl(\widetilde G_{m+1}\bigr)&=\widehat\E\Bigl({\mathbf1}_{\widetilde G_m}\,\widehat\IP\bigl\{G_{m+1}\,|\,\FF_m\bigr\}\Bigr)
	\ge \bigl(1-Nq^m{d}\bigr)\,\widehat\IP\bigl(\widetilde G_m\bigr)\\
	&\ge \bigl(1-Nq^m{d}\bigr)\,\delta_m>\delta_{m+1}. 
\end{align*}
This completes the induction step and the proof of~\eqref{infinite-coupling}. 

\smallskip
{\it Step~2: Transition to a ball of small radius\/}. Let us note that all our constructions and conclusions above remain true if we replace the number~$\theta$, defining the set~$D_\theta^c$ (cf.~\eqref{const}), with any smaller positive constant, depending only on the system~\eqref{RDS}. Thus, without loss of generality, we shall assume in what follows that 
\begin{equation}\label{modif}
	\theta\le\frac{1}{2N_1},
\end{equation}
where $N_1>0$ is the number in~\eqref{infinite-coupling}. In this case, the right-hand side of~\eqref{infinite-coupling} is minorised by~$1/2$ if $d=\dd_\XXXX(U,U')\le \theta$.

As a first step in the proof of~\eqref{to-coupling}, we establish a simpler result concerning single trajectories $\{U_k,k\ge0\}$. Namely, we claim that, for any $\delta>0$, there is $U_\delta\in\XXXX$, a number $p_\delta>0$, and an integer $m_\delta\ge1$ such that 
\begin{equation}\label{transition-U-Udelta}
	\widehat\IP\bigl\{U_{m_\delta}\in B_\XXXX(U_\delta,\delta)\bigr\}\ge p_\delta\quad\mbox{for any initial point $U\in\XXXX$}.
\end{equation}
Indeed, in view of Hypothesis~\hyperlink{GCP}{(GCP)}, in which we take $\e=\delta/2$, there is a point $\widehat U\in\XXXX$ and an integer $m\ge1$ such that, for any $U\in\XXXX$ and some appropriately chosen vectors $\xi_k\in E$, $1\le k\le m$ relations~\eqref{xi_k-support} hold with $\e=\delta/2$. By the continuity of~$\mathscr S$, there is $\nu>0$ such that, for any $\tilde\xi_k\in B_E(\xi_k,\nu)$, $1\le k\le m$, the corresponding trajectory of the RDS~\eqref{RDS} belongs to the ball $B_\XXXX(\widehat U,\delta)$ at time $k=m$. Since $\xi_k\in \supp\DD(\zeta_k^{U_{k-1}})$ for $1\le k\le m$, it follows that 
 $$
 \widehat\IP\bigl\{\zeta_k^{U_{k-1}}\in B_E(\xi_k,\nu)\mbox{ for }1\le k\le m \bigr\}>0. 
 $$
 We have thus proved that the probability of transition from any point $U\in\XXXX$ to the $\delta$-neighbourhood of~$\widehat U$ at time $k=m$ is positive. By the Feller property and the well-known characterisation of the weak convergence of measures, the function $U\mapsto P_m(U,B_\XXXX(\widehat U,\delta))$ is lower semicontinuous. Since it is positive everywhere, we conclude that 
 $$
 \inf_{U\in\XXXX}P_m\bigl(U,B_\XXXX(\widehat U,\delta)\bigr)>0.
 $$
 This shows that~\eqref{transition-U-Udelta} holds with $m_\delta=m$ and $U_\delta=\widehat U$. 
 
\smallskip
{\it Step~3: Proof of~\eqref{to-coupling}\/}. 
In view of~\eqref{laws}, if the processes~$\varPhi_k(U,U')$ and~$\varPhi_k'(U,U')$ were independent, then the required result would follow immediately from~\eqref{transition-U-Udelta}. However, they are not, and we have to proceed differently. In what follows, we can assume without loss of generality that~$\delta<\theta$, where $\theta>0$ is satisfies~\eqref{modif}. Let $k_*>0$ be the smallest integer such that $\theta q^{k_*} \le \delta$. We claim that~\eqref{to-coupling} holds with $l=m_{\delta/2}+k_*$, where $m_{\delta/2}$ is the integer entering~\eqref{transition-U-Udelta} in which the radius~$\delta$ of the ball is replaced with~$\delta/2$. 

To prove this, we abbreviate $\Phi_k(U, U') = \Phi_k$, $\Phi'_k(U, U') = \Phi'_k$, $m=m_{\delta/2}$. We set $\zeta_l=(\varPhi_l,\varPhi_l')\in\XXXX\times\XXXX$, $d(\zeta_l)=\dd_\XXXX(\varPhi_l,\varPhi_l')$ and define the stopping times 
$$
\tau=\tau(U,U')=\min\bigl\{k\ge0:d(\zeta_k)\le\theta\bigr\}, \quad \tau^m=\tau\wedge(m+1).
$$
Without loss of generality, we may assume that, in construction~\eqref{Phi-Phi'}, 
when $(U,U')\in D_\theta^c$ we take random variables $\RR(U,U',\omega_k)$, $\RR'(U,U',\omega_k)$, $k\ge1$, from some fixed collection of pairs of independent random fields 
$$
(\RR(\cdot,\omega_k), \RR'(\cdot,\omega_k), k=1,2,\dots).
$$ 
Then we define processes $\{\hat\varPhi_k, k\ge1\}$ and $\{\hat\varPhi_k', k\ge1\}$ by relation~\eqref{Phi-Phi'}, where always the random fields~$\RR$ and~$\RR'$ are independent and are taken from the collection above. We set $\hat\zeta_l=(\hat\varPhi_l,\hat\varPhi_l')$ and define $d(\hat\zeta_l)$, $\hat\tau$, $\hat\tau^m$ similarly to $d(\zeta_l)$, $\tau$, $\tau^m$. Note that 
$$
\tau^m=\hat\tau^m\quad\mbox{and}\quad \zeta_l=\hat\zeta_l\quad\mbox{for $l\le \tau^m$}. 
$$
In view of~\eqref{transition-U-Udelta}, 
\begin{equation}\label{Z1}
	\widehat\IP\{\hat\tau^m=m+1\}\le \widehat\IP\{d(\hat\zeta_m)>\theta\}
	\le 1-(p_{\theta/2})^2.
\end{equation}
Let us set $A_m=\{\dd_\XXXX(\varPhi_k(\zeta_{\tau^m}),\varPhi_k'(\zeta_{\tau^m}))\le q^k\theta\mbox{ for }k\ge0\}$. By the strong Markov property, applied to the process~$\{\zeta_l\}$, we have 
\begin{align}
\widehat\IP\{d(\zeta_{\tau^m+k})\le q^k\theta\mbox{ for }k\ge0\}
&=\widehat\IP(A_m)\ge \widehat\IP\bigl(\{\tau^m\le m\}\cap A_m\bigr)\notag\\
&=\E\bigl({\mathbf1}_{\{\tau^m\le m\}}\,\IP\{A_m\,|\,\FF_{\tau^m}\}\bigr).
\label{lower-bound-dz}
\end{align}
But for every~$\omega\in \{\tau^m\le m\}$, we have $\tau^m=\tau$ and $d(\zeta_{\tau^m})\le\theta$. So for every such~$\omega$ it holds that $\IP\{A_m\,|\,\FF_{\tau^m}\}\ge\frac12$ in view of~\eqref{infinite-coupling} and~\eqref{modif}. Since $\tau^m=\hat\tau^m$, then by~\eqref{Z1} the right-most term in~\eqref{lower-bound-dz} is no less than~$\frac12(p_{\theta/2})^2$. As the event in~\eqref{to-coupling} contains the event~$\{d(\zeta_{\tau^m+k})\le q^k\theta\mbox{ for }k\ge0\}$ if $l=m+1$, then~\eqref{to-coupling} is proved with $l=m_{\delta/2}+1$ and $p=\frac12(p_{\theta/2})^2$. 

\smallskip
We have thus established~\eqref{mixing-RDS}. To establish~\eqref{convergence-V-s}, we first note that it suffices to find $C_s>0$ such that
\begin{equation}\label{U-k-Lipschitz}
	\|\DD(U_0,\dots,U_s)-\DD(U_0',\dots,U_s')\|_L^*\le C_s\dd_\XXXX(U,U')
	\quad\mbox{for $U,U'\in\XXXX$},
\end{equation}
where $\{U_k\}$ and $\{U_k'\}$ stand for the trajectories of~\eqref{RDS} issued from~$U$ and~$U'$, respectively. Once this inequality is proved, \eqref{convergence-V-s} will easily follow due to the Markov property. 

To prove~\eqref{U-k-Lipschitz}, we first consider the case $s=1$. In view of~\eqref{RDS} and~\eqref{TF}, for any continuous function $f:\XXXX\times\XXXX\to\R$, we have 
$$
\bigl\langle f,\DD(U_0,U_1)\bigr\rangle
=\int_\KK f\bigl(U,\SSS(U,y)\bigr)\PPP(U;\dd y).
$$
Assuming that $f$ is $1$-Lipschitz and satisfies the inequality $\|f\|_\infty\le1$, and using the Lipschitz property of~$\SSS$ and inequality~\eqref{lipcshitz-transition}, it is easy to show that
$$
\bigl|\bigl\langle f,\DD(U_0,U_1)\bigr\rangle-\bigl\langle f,\DD(U_0',U_1')\bigr\rangle\bigr|\le C_1\,\dd_\XXXX(U,U'),
$$
where $C_1>0$ does not depend on~$f$. Taking the supremum with respect to~$f$, we arrive at~\eqref{U-k-Lipschitz} for $s=1$. A similar argument combined with induction allows one to establish~\eqref{U-k-Lipschitz} for any integer $s\ge2$. The proof of Theorem~\ref{t-mixing-kantorovich} is complete. 
\end{proof}

\section{Main results}
\label{s-MR}

\subsection{Formulation of the main theorem}
\label{s-formulation-MR}
Let us recall the setting of Subsection~\ref{ss-RDS}. As was mentioned there after Lemma~\ref{l-reduction}, we assume without loss of generality that~$H$ is a separable Hilbert space, $E$ is a separable Banach space, $S:H\times E\to H$ is a continuous map, and $X\subset H$ and $\KK\subset E$ are compact subsets such that $S(X\times\KK)\subset X$. We consider the stochastic system in~\eqref{stationary-RDS}, 
$$
u_k=S(u_{k-1}, \eta_k), \quad k\ge1,
$$
supplemented with an initial condition 
\begin{equation}\label{IC-stationary}
	u_0=u,
\end{equation}
where $u\in X$ is a given point, and~$\{\eta_k\}$ is a stationary sequence satisfying the properties mentioned in Subsection~\ref{ss-RDS}. In particular, the law of~$\eta_k$ has the compact support $\KK\subset E$, and the conditional law~$Q(\xxi;\cdot)$ satisfies~\hyperlink{Fel}{(SF)}; see Section~\ref{ss-RDS}.

\begin{definition}\label{d-mixing}
	We shall say that stochastic system~\eqref{stationary-RDS}  is {\it exponentially mixing in the dual-Lipschitz norm\/}\,\footnote{See the first footnote.}  if there is a measure $\mu\in\PP(X)$ and positive numbers~$C$ and~$\gamma$ such that, for any initial state $u\in X$, the corresponding trajectory $\{u_k\}_{k\in\Z_+}$ of~\eqref{stationary-RDS}, \eqref{IC-stationary} satisfies the inequality
	\begin{equation}\label{expoconv}
		\|\DD(u_k)-\mu\|_L^*\le C\,e^{-\gamma k}, \quad k\ge0.
	\end{equation}
\end{definition}

In what follows, we drop for short {\it in the dual-Lipschitz norm\/} and call stochastic systems as in Definition~\ref{d-mixing} just {\it exponentially mixing\/}. Let us note that if the initial state~$u_0$ is an $X$-valued random variable independent of~$\{\eta_k\}_{k\in\Z}$, then~\eqref{expoconv} remains valid. To see this, it suffices to remark that if $\lambda=\DD(u_0)$, then
\begin{equation*}\label{law-randomIC}
	\DD(u_k)=\int_X P_k(v,\cdot)\lambda(\dd v),
\end{equation*}
where $P_k(v,\cdot)$ stands for the law at time~$k$ of the trajectory of~\eqref{stationary-RDS}, \eqref{IC-stationary} with $u=v$. We recall that the laws~$\DD(u_k)$, $k\ge1$, depend not on the process~$\{\eta_k\}$, but only on its distribution. 

Below, we shall prove that, under certain restrictions, system~\eqref{stationary-RDS} is mixing, and at the end of Section~\ref{s-proofMT}, shall show that the measure~$\mu$ in~\eqref{expoconv} is stationary in the following sense: there exists a random variable~$v$, measurable with respect to~$\FF_0$, such that $\DD(v)=\mu$,  and a weak solution $\{u_k,k\ge0\}$ for~\eqref{stationary-RDS} such that $\DD(u_k)\equiv\mu$. It is obvious from~\eqref{expoconv} that~$\mu$ is the only measure with this property. 

To ensure the property of exponential mixing for~\eqref{stationary-RDS}, we assume that the map $S:H\times E\to H$ is {\it twice continuously differentiable and is  bounded on  bounded subsets together with its derivatives up to the second order\/}, and impose the following four hypotheses on the dynamics and driving noise. For an integer $k\ge1$, we denote $\mathbf{0}_k=(0,\dots,0)$, where $0\in E$ is repeated~$k$ times, and for a vector $\vec{\xi}_n=(\xi_1,\dots,\xi_n)\in\supp\DD(\eta_1,\dots,\eta_n)\subset\KK^n$, we denote
\begin{equation}\label{notat}
\mbox{$S_n(v;\vec{\xi}_n)=u_n$, where $\{u_1,\dots,u_n\}$ is a trajectory of~\eqref{stationary-RDS} with $\eta_k=\xi_k$}. 
\end{equation}

\begin{description}\sl 
	\item [\hypertarget{GD}{(GD) Global dissipation}.] There is an integer $k\ge1$ and a number $a\in(0,1)$ such that
\begin{equation*}\label{global-dissipation}
	\|S_k(u;\mathbf{0}_k)\|_H\le a\,\|u\|_H
	\quad\mbox{for any $u\in H$}. 
\end{equation*} 
\end{description}

\begin{remark}\label{r-GS}
This hypothesis implies that the unperturbed dynamics has one globally stable equilibrium. In a subsequent paper, we shall relax this condition to include stochastic perturbations of more general systems. The  result will be applicable, for instance, to a reaction-diffusion equation possessing finitely many hyperbolic equilibria and a Lyapunov function. 
\end{remark}

We shall say that a closed subspace $G\subset H$ is {\it determining\/} if there is a  number $\varkappa\in(0,1)$ such that
\begin{equation}\label{determining}
	\bigl\|(I-{\mathsf P}_G)D_uS(u,\eta)\bigr\|_{\LL(H)}
	\le \varkappa\quad\mbox{for any $u\in X$, $\eta\in \KK$},
	\end{equation} 
	where ${\mathsf P}_G:H\to H$ denote the orthogonal projection to~$G$. This definition is related to the usual concept of {\it determining modes\/} going back to Foia\c s--Prodi~\cite{FP-1967}. Indeed, assuming that $X$ is convex and applying the mean value theorem, we see that 
$$
\bigl\|(I-{\mathsf P}_G)\bigl(S(u,\eta)-S(u',\eta)\bigr)\bigr\|_H\le \varkappa\,\|u-u'\|_H.
$$
In particular, if $\{u_k\}$ and $\{u_k'\}$ are two trajectories of~\eqref{stationary-RDS} with the same driving noise such that ${\mathsf P}_Gu_k={\mathsf P}_Gu_k'$ for $k\ge1$, then the difference $u_k-u_k'$ goes to zero exponentially fast. 

\begin{remark}\label{r-determining}
If~$H$ is a finite-dimensional space, then~\eqref{determining} trivially holds with $G=H$. It also holds if the map~$S$ is smoothing in the sense that for some Hilbert space~$V$, compactly and densely embedded in~$H$, the map~$S$ is continuously differentiable from $H\times E$ to~$V$. Indeed, in this case~$H$ admits a Hilbert basis $\{e_j, j\ge1\}$ which also is an orthogonal basis of~$V$ such that the norms $\|e_j\|_V$ go to infinity with~$j$ (e.g., see~\cite[Section~2.1]{LM1972}). Then~\eqref{determining} is fulfilled  if we take for~$G$ the vector span of $\{e_j, 1\le j\le N\}$ with a sufficiently large~$N$.
\end{remark}

\begin{description}\sl 
	\item [\hypertarget{ALC}{(ALC) Approximate linearised controllability}.] There exists an open set $O\supset X\times\KK$ and a finite-dimensional determining subspace $G\subset H$ such that the closure of the image of the operator $D_\eta S(u,\eta):E\to H$ contains~$G$ for any $(u,\eta)\in O$. 
\end{description}

Recall that we denote by~$Q(\xxi; \cdot)$ the conditional law of~$\eta_1$ given the past $\{\eta_k,k\in\Z_-\}=\xxi\in\EE$ and that Assumption~\hyperlink{Fel}{(SF)} holds for it. The next hypothesis is an analogue of~\hyperlink{DLP}{(DLP)} in the current setting.  

\begin{description}\sl 
	\item [\hypertarget{DLP'}{(DLP$'$) Decomposability and Lipschitz property}.] 
	There is an increasing sequence of finite-dimensional subspaces $F_n\subset E$, with complementary subspaces~$F_n^\dagger\subset E$, and projections $\mathsf P_n:E\to F_n$, such that the union~$\cup_nF_n$ is dense in~$E$, the operator norms of~$\mathsf P_n$ are bounded uniformly in~$n\ge1$, and the following property holds for any $n\ge1$: for each $\xxi\in\EE$, the measure $Q(\xxi;\cdot)$ admits a disintegration
\begin{equation}\label{decomposition-Q}
	Q(\xxi;\dd y)=Q_n^\dagger(\xxi;\dd y_n^\dagger)Q_n(\xxi,y_n^\dagger;\dd y_n),
\end{equation}
where the decomposition $y=(y_n,y_n^\dagger)$ is associated with the direct sum $E=F_n\dotplus F_n^\dagger$, and we write $Q_n^\dagger(\xxi;\cdot)\in\PP(F_n^\dagger)$ for the image of~$Q(\xxi;\cdot)$ under the projection~$I-\mathsf P_n$ onto~$F_n^\dagger$. Moreover, there is a Lipschitz continuous function $\rho_n:\EE\times E\to\R_+$ supported by~$\EE\times\KK$ such that 
\begin{equation*}\label{Q-density}
Q_n(\xxi,y_n^\dagger;\dd y_n)=\rho_n(\xxi,y_n^\dagger,y_n)\ell_n(\dd y_n), 
\end{equation*}
where $\ell_n$ stands for the Lebesgue measure on~$F_n$.
\end{description}
This hypothesis is a version of a condition often imposed in Dobrushin's work on reconstructing random fields on lattices from their conditional distributions (e.g., Assumption~2 in~\cite[Theorem~1]{dobrushin-1970} or~\cite[Theorems~5 and~6]{dobrushin-1968}). If $\dim E <\infty$, we take $F_n=E $ for all $n\ge1$, so that $F_n^\dagger=\{0\}$, and condition~\hyperlink{DLP'}{(DLP$'$)} simplifies to the assumption that $Q(\xxi;\dd y)=\rho(\xxi,y)\ell(\dd y)$, where~$\rho$ is a Lipschitz function. In this case, \hyperlink{Fel}{(SF)} follows from~\hyperlink{DLP'}{(DLP$'$)}. We also note that these two assumptions are consequences of property~\hyperlink{eta1}{($\eta1$)} in the Introduction for both finite and infinite-dimensional spaces~$E$. Indeed, the claim follows easily from the observation that the projection of $Q(\xxi;\cdot)$ to the vector span of $f_1,\dots,f_n$ has the form 
$$
Q^1(\xxi,\dd x_1)\otimes\cdots\otimes Q^n(\xxi,\dd x_n)=\sum_{l=1}^np_\xxi^1(x_1)\dots p_\xxi^n(x_n)\,\dd x_1\dots\dd x_n. 
$$

For any $\xxi\in\EE$ and any integer $k\ge1$, the conditional law of the vector~$(\eta_1,\dots,\eta_k)$ given that the past~$\eeta_-:=\{\eta_l,l\le0\}$ equals~$\xxi$ is the measure 
\begin{equation}\label{conditional-law}
	{\boldsymbol Q}_k(\xxi;\dd y_1,\dots,\dd y_k)
	=Q(\xxi;\dd y_1) Q(\xxi_1;\dd y_2)\cdots Q(\xxi_{k-1};\dd y_k),
\end{equation}
where $\xxi_l=(\xxi, y_1,\dots, y_l)$ is the sequence obtained by concatenation of~$\xxi$ and $(y_1,\dots, y_l)$. Given integers $k\ge0$ and $m\ge1$, we denote by $Q_k^m(\xxi;\cdot)\in\PP(\KK^m)$ the measure 
\begin{equation}\label{Qkm}
	Q_k^m(\xxi;\cdot)=\langle\mbox{projection of ${\boldsymbol Q}_{k+m}(\xxi;\cdot)$ to the last~$m$ components}\rangle. 
\end{equation}
Thus, $Q_k^m(\xxi;\cdot)$ is the conditional law of the vector $(\eta_{k+1},\dots,\eta_{k+m})$ given the past $\eeta_-=\xxi$. The following hypothesis is similar to the property of {\it strong recurrence\/} in the context of Markov processes. 

\begin{description}\sl 
	\item [\hypertarget{SRZ}{(SRZ) Strong recurrence to zero}.] \sl For any $n\in\N$   and $\delta>0$, there is an integer $s\ge0$ such that
	\begin{equation}\label{to-zero}
	\inf_{\xxi\in\EE}Q_s^n\bigl(\xxi;\OO_\delta(\mathbf{0}_n)\bigr)>0. 
	\end{equation}
\end{description}
The following theorem is the main result of this paper. Its proof is based on Theorem~\ref{t-mixing-kantorovich} and is presented in the next subsection.

\begin{theorem}\label{t-mixing-dissipation}
	Suppose that Hypotheses~{\rm \hyperlink{Fel}{(SF)}, \hyperlink{GD}{(GD)}, \hyperlink{ALC}{(ALC)}, \hyperlink{DLP'}{(DLP$'$)}}, and~\hyperlink{SRZ}{\rm(SRZ)} are fulfilled. Then the RDS~\eqref{stationary-RDS} is exponentially mixing. 
\end{theorem}

In conclusion of this subsection, let us mention that the question of existence of processes for which Hypotheses~\hyperlink{DLP'}{(DLP$'$)} and~\hyperlink{SRZ}{(SRZ)} are satisfied will be discussed in Section~\ref{s-mixingnoise}. Here we only point out that if an $E$-valued stationary process $\eeta:=\{\eta_k\}_{k\in\Z}$ is such that~\hyperlink{DLP'}{(DLP$'$)} holds for its conditional probabilities~$Q(\xxi;\cdot)$, then we have the following implications:\,
\begin{equation}\label{mixing-recurrence}
	\mbox{strong feeble mixing for~$\eeta$\ $\Longrightarrow$\ strong recurrence\ $\Longrightarrow$\ feeble mixing for~$\eeta$};
\end{equation}
see Subsection~\ref{s-SR-mixing} for exact statements. Thus, our class of admissible random forces~$\eeta$ can be informally characterised as {\it  mixing stationary processes with Lipschitz continuous densities for  conditional laws\/}. 

\subsection{Proof of Theorem~\ref{t-mixing-dissipation}}
\label{s-proofMT}
 Let us recall that the law of a trajectory~$u_k$ for~\eqref{stationary-RDS} issued from~$u\in X$ depends not on the concrete realisation of the process $\{\eta_k,k\ge1\}$, but only on its law. It coincides with that of the first component of the trajectory $U_k=(v_k,\xxi_k)$ for~\eqref{RDS-example}, with the phase space $\XXXX=X\times\EE$, where~$\SSS$ is defined by~\eqref{S-stationary}, and $\DD(U_0)=\delta_u\otimes\sigma$; see~\eqref{sigma-etak}. Below we call the corresponding Markov process in~$\XXXX$ the {\it $U$-process\/}. If we prove that the $U$-process satisfies~\eqref{mixing-RDS}, then taking its projection~$\Pi_X$ to the $X$-component, we can conclude that~$\DD(u_k)$ satisfies~\eqref{expoconv} with 
 \begin{equation}\label{proj-mu}
 \mu=(\Pi_X)_*\mmu\,;
 \end{equation}
see~\eqref{nat-proj}. Moreover, applying~\eqref{mixing-RDS} to a solution of~\eqref{RDS} with initial condition~$U_0$ as above and using~\eqref{new-relation}, we get 
\begin{equation}\label{grr}
(\Pi_\EE)_*\mmu=\sigma. 	
\end{equation}
 
Thus, to prove Theorem~\ref{t-mixing-dissipation}, it suffices to check that the three hypotheses of Theorem~\ref{t-mixing-kantorovich} are fulfilled. The proof of this fact is divided into three steps. To simplify notation, we denote by~$\ssigma\in\PP(E^\Z)$ the law of $\eeta=\{\eta_k,k\in\Z\}$ and write~$\ssigma_{\!m}$ for the law of $(\eta_1,\dots,\eta_m)$, so that~$\ssigma_{\!m}$ is  the projection of~$\ssigma$ to~$E^m$. 

\smallskip
{\it Step~1: Transition probability\/}. We first express the law of the noise in the RDS~\eqref{RDS-example} and the transition probability of the $U$-process in terms of the objects entering~\eqref{stationary-RDS}. To this end, note that since  the law of~$\zeta_1^\xxi$  equals~$Q(\xxi;\cdot)$ for any $\xxi\in\EE$, we have
\begin{equation}\label{Q-law-stat}
	\PPP(U;\Gamma)=Q(\xxi;\Gamma)
	\quad\mbox{for $\Gamma\in\BB(E)$,  $U=(v,\xxi)\in\XXXX$}.
\end{equation}
This relation and Hypothesis~\hyperlink{DLP'}{(DLP$'$)} imply that $\PPP(U;\cdot)$ satisfies Hypothesis~\hyperlink{DLP}{(DLP)}, in which one can take $F=F_n$ with an arbitrary $n\ge1$. 

We now describe the transition probability of the $U$-process. Recalling~\eqref{S-stationary} and~\eqref{TF}, we see that the time-$1$ transition probability can be written as 
\begin{equation}\label{transition-1}
P_1\bigl(U,B\times(\GGamma\times\Gamma)\bigr)=\delta_\xxi(\GGamma)\int_{\Gamma}{\mathbf1}_B\bigl(S(v,y)\bigr)Q(\xxi;\dd y), 
\end{equation}
where $U$ and $\Gamma$ are the same as in~\eqref{Q-law-stat}, $\GGamma\in\BB(\EE)$, and $B\in\BB(X)$. In other words, for $U=(v,\xxi)\in\XXXX$, denoting $\Psi_U:\KK\to X\times\EE$, $\eta\mapsto \SSS(U,\eta)$, we can write $P_1(U;\cdot)=(\Psi_U)_*Q(\xxi;\cdot)$. Then, using~\eqref{transition-1} and arguing by induction, it is straightforward to see that the time-$k$ transition probability has the form 
\begin{equation*}\label{transition-k}
	P_k\bigl(U,B\times(\GGamma\times \Gamma_k)\bigr)
	=\delta_\xxi(\GGamma)\int_{\Gamma_k}{\mathbf1}_B\bigl(S_k(v;y_1,\dots,y_k)\bigr){\boldsymbol Q}_k(\xxi;\dd y_1,\dots,\dd y_k),
\end{equation*}
where $U=(v,\xxi)\in\XXXX$, $\GGamma\in\BB(\EE)$, $\Gamma_k\in\BB(E^k)$, $B\in\BB(X)$,  and we recall~\eqref{notat} and~\eqref{conditional-law}. 
\smallskip

{\it Step~2: Global controllability to  points~\hyperlink{GCP}{\rm(GCP)}\/}. We claim that the required property holds with $\widehat U=(0,{\mathbf0})$, where ${\mathbf0}\in\EE$ is the sequence whose elements are all zero. Indeed, let us fix any $\e>0$ and choose an integer $l\ge1$ and a number~$\delta>0$ so that, for any element $\boldsymbol{\hat\xi}=(\hat\xi_j,j\le0)\in\EE$ satisfying the inequality $\|\hat\xi_j\|_E\le\delta$ for $1-l\le j\le0$, we have $\dd(\boldsymbol{\hat\xi},{\mathbf0})<\e/3$. Let us take any $U=(u,\xxi)\in\XXXX$. By~\hyperlink{GD}{(GD)}, there is an integer $n\ge l$ not depending on~$U$ such that $\|S_n\bigl(u;{\mathbf0}_n)\|_H<\frac{\e}{3L}$. Since~$S$ is continuous, we can find $\delta>0$ such that   
\begin{equation}\label{e/2}
	\bigl\|S_n\bigl(u;\xi_1',\dots,\xi_n'\bigr)\bigr\|_H<\frac{2\e}{3L}\quad\mbox{for $u\in X$, $(\xi_1',\dots,\xi_n')\in \OO_\delta({\mathbf0}_n)$}. 
\end{equation}
By~\hyperlink{SRZ}{(SRZ)}, there is $s\in\N$ such that~\eqref{to-zero} holds. It follows that there exist vectors $\xi_1,\dots,\xi_{s+n}\in\KK$ such that 
\begin{gather}
(\xi_{s+1},\dots,\xi_{s+n})\in \OO_\delta({\mathbf0}_n),\label{xi-in-delta}\\
\xi_{j+1}\in \supp Q\bigl((\xxi;\xi_1,\dots,\xi_j);\cdot\,\bigr)\quad\mbox{for $0\le j<s+n$}.\label{xi-j-supp}
\end{gather}
We now set $v_k=S(v_{k-1},\xi_k)$, where $1\le k\le s+n$ and $v_0=u$. In view of~\eqref{e/2} and~\eqref{xi-in-delta}, we have $\|v_{s+n}\|_H<\frac{2\e}{3L}$. Furthermore, inequality~\eqref{xi-in-delta} and the choice of~$l$ imply that $\dd(\xxi_{s+n},{\mathbf0})<\frac{\e}{3}$, where $\xxi_k=(\xxi,\xi_1,\dots,\xi_k)$. Recalling that the $U$-process is given by $U_k=(v_k,\xxi_k)$ and using~\eqref{metric-X}, we see that 
$$
\dd_\XXXX(U_{s+n},\widehat U)=L\|v_{s+n}\|_H+\dd(\xxi_{s+n},{\mathbf0})<\e. 
$$
Finally, it follows from~\eqref{xi-j-supp} and~\eqref{Q-law-stat} that $\xi_k\in\supp\PPP(U_{k-1};\cdot)$ for $1\le k\le s+n$. This completes the verification of~\eqref{xi_k-support} with $m=s+n$.

\smallskip
{\it Step~3: Local approximate controllability~\hyperlink{LAC}{\rm(LAC)}\/}. To prove~\hyperlink{LAC}{(LAC)}, we define the set 
$$
\widetilde D_\delta=\{(v,v')\in X\times X:\|v-v'\|_H\le\delta\}
$$
and suppose that, for a small $\delta>0$, we have constructed a continuous map 
$$
\tilde\varPhi:\widetilde D_{\delta}\times \KK\to E, \quad (v,v',\xi)\mapsto\xi',
$$
such that, for any $(v,v')\in \widetilde D_{\delta}$, the map $\tilde\varPhi(v,v',\cdot):\KK\to E$ is continuously differentiable and satisfies the inequalities (cf.~\eqref{Phi-bound} and~\eqref{Phi-squeezing})
 \begin{align}
\sup_{\xi\in \KK}\Bigl(\|\tilde\varPhi(v,v',\xi)\|_E+\Lip_\xi\bigl(\tilde\varPhi(v,v',\xi)\bigr)\Bigr)&\le C'\|v-v'\|_H, \label{tPhi-bound}\\
\sup_{\xi\in \KK}\,\bigl\|S(v,\xi)-S\bigl(v',\xi+\tilde\varPhi(v,v',\xi)\bigr)\bigr\|_H &\le q'\|v-v'\|_H, \label{tPhi-squeezing}
\end{align}
where $C'$ and $q'<1$ are some positive numbers. In this case, recalling the constant $L\ge1$ in the distance~\eqref{metric-X} and defining $\varPhi:D_\delta\times \KK\to E$ by the relation
$$
\varPhi(U,U',\xi)=\tilde\varPhi(v,v',\xi), \quad U=(v,\xxi), \quad U'=(v',\xxi'), 
$$ 
we see that~\eqref{Phi-bound} is trivially satisfied with $C_*=C'L^{-1}$. To prove~\eqref{Phi-squeezing}, we use~\eqref{tPhi-bound} and~\eqref{tPhi-squeezing} to write
\begin{align*}
	\dd_\XXXX\bigl(\SSS(U,\xi),\SSS(U',\xi+\varPhi(U,U',\xi)\bigr) 
	&\le \iota^{-1}\,\dd(\xxi,\xxi')+\|\varPhi(U,U',\xi)\|_E\\
	&\qquad+ L\, \|S(v,\xi)-S(v',\xi+\tilde\varPhi(v,v',\xi)\|_H\\
	&\le \iota^{-1}\,\dd(\xxi,\xxi')+L\,(L^{-1}C'+q')\|v-v'\|_H.
\end{align*} 
Up to this point, the number~$L\ge1$ was arbitrary. We now choose it so large that $L^{-1}C'+q'<1$. Then the above estimate implies inequality~\eqref{Phi-squeezing} with $q=\max\{\iota^{-1}$, $L^{-1}C'+q'\}<1$. Thus, it remains to construct~$\tilde\varPhi$ satisfying~\eqref{tPhi-bound} and~\eqref{tPhi-squeezing}. 

To this end, let us denote by $\Delta(v,v',\xi)$ the expression under the norm sign in the left-hand side of~\eqref{tPhi-squeezing}, where $\tilde\varPhi$ is a map to be defined. Using the Taylor formula and the $C^2$-smoothness of~$S$, we write 
	\begin{equation}\label{taylor}
	\Delta(v,v',\xi)=(D_vS)(v,\xi)(v-v')-(D_\xi S)(v,\xi)\tilde\varPhi(v,v',\xi)+r(v,v',\xi),
\end{equation}
where the remainder term~$r$ satisfies the inequality
\begin{equation}\label{estimate-r}
	\|r(v,v',\xi)\|_H\le C_1\bigl(\|v'-v\|_H^2+\|\tilde\varPhi(v,v',\xi)\|_E^2\bigr)
\end{equation}
with some constant~$C_1$ not depending on~$v$, $v'$, and~$\xi$. We now need the following lemma whose proof is given in Section~\ref{s-proof-lemma} and is based on  the {\it Moore--Penrose formula\/} for pseudo-inverse; cf.~\cite[Theorem~2.8]{KNS-gafa2020}.  

\begin{lemma}\label{l-right-inverse}
	Let~$H$ and~$E$ be separable Hilbert and Banach spaces, $G\subset H$ be a finite-dimensional vector subspace, $O$ be an open subset of a Banach space~$\HH$, and $A:O\to\LL(E,H)$ be a $C^1$-smooth map such that, for any $y\in O$, the closure of the image of the linear application~$A(y)$ contains~$G$. Let finite-dimensional subspaces $F_n\subset E$ be as in~\hyperlink{DLP'}{\rm(DLP$'$)}.  Then, for any $\e>0$ and any compact set $Y\subset O$, there is an integer $n\ge1$ and a $C^1$-smooth map\footnote{Let us emphasise that the map~$B_\e$ is defined on~$O$, but it depends on the choice of~$Y$.} $B_\e:O\to\LL(G,F_n)$ such that 
\begin{equation}\label{ARI}
	\|A(y)B_\e(y)f-f\|_H\le \e\,\|f\|_H
	\quad\mbox{for any $f\in G$, $y\in Y$}. 
\end{equation}
\end{lemma}

By Hypothesis~\hyperlink{ALC}{(ALC)}, for any $(v,\xi)\in O\subset H\times E$, the closure of the image of the operator $D_\xi S(v,\xi):E\to H$ contains the finite-dimensional subspace~$G$. Applying Lemma~\ref{l-right-inverse} to $A(v,\xi)=D_\xi S(v,\xi)$, for any $\e>0$ we can construct $C^1$-smooth mapping $B_\e:O\to\LL(G,F_n)$ such that 
\begin{equation}\label{right-inverse}
\|D_\xi S(v,\xi)B_\e(v,\xi)f-f\|_H\le\e\,\|f\|_H\quad\mbox{for $(v,\xi)\in X\times \KK$, $f\in G$}.
\end{equation}
We now set 
\begin{equation}\label{Phi-definition}
\tilde\varPhi(v,v',\xi)=B_\e(v,\xi)\,{\mathsf P}_G(D_vS)(v,\xi)(v-v'). 	
\end{equation}
The validity of~\eqref{tPhi-bound} with some constant $C'=C'(\e)$ follows from the compactness of $X\times \KK$ and the continuity of the functions entering~\eqref{Phi-definition}. Furthermore, substituting~$\tilde\varPhi$ into~\eqref{taylor} and using~\eqref{determining}, \eqref{right-inverse}, and~\eqref{estimate-r}, we derive
$$
\|\Delta(v,v',\xi)\|_H\le (\varkappa+C_2\e)\|v-v'\|_H+C_3\|v-v'\|_H^2,
$$
where we set 
$$
C_2=\sup\bigl\{\bigl\|{\mathsf P}_G(D_vS)(v,\xi)\bigr\|_{\LL(H,G)}:(v,\xi)\in X\times \KK\bigr\},\quad C_3=C_1(1+C').
$$ 
Up to this point, the number~$\e>0$ was arbitrary. We now choose $\e>0$ and $\delta>0$ so that $q':=\varkappa+C_2\e+C_3\delta<1$. In this case, inequality~\eqref{tPhi-squeezing} holds, and this completes the proof of the theorem.\qed

\subsection{Concluding remarks}
\label{s-remarks}

	The proof given above implies two additional properties: the stationarity of the limiting measure~$\mu$ and convergence in the space of  trajectories. Namely, we have the two propositions below, provided that the conditions of  Theorem~\ref{t-mixing-dissipation} are fulfilled.   

\begin{proposition}[Stationarity of~$\mu$]\label{p-stationarity}
Equation~\eqref{stationary-RDS} possesses a weak solution $\{(\hat u_k,\hat\xi_k),k\in\Z_+\}$ such that $\{(\hat u_k,\hat\xi_k)\}$ is a stationary process in~$X\times\EE$. Moreover, $\DD(\hat u_k)\equiv\mu$.
\end{proposition}

\begin{proof}
	Let us consider again the extended system~\eqref{RDS-example}, which generates a Markov $U$-process in the space $\XXXX=X\times\EE$. The latter has a unique stationary measure $\mmu\in\PP(\XXXX)$ whose projection to~$X$ is the limiting measure $\mu\in\PP(X)$ for~\eqref{stationary-RDS}; see~\eqref{proj-mu}. Let~$\{V_k=(v_k,\xxi_k)\}$ be a stationary trajectory of~\eqref{RDS-example}, so that $\DD(V_k)\equiv\mmu$. It suffices to take $\hat u_k=v_k$ for $k\ge0$. Indeed, by~\eqref{grr}, $\DD(\xxi_k)=\sigma$. So, by Lemma~\ref{l-reduction}, $\{v_k,k\ge0\}$ is a weak solution of~\eqref{stationary-RDS}, and $\DD(v_k)=(\Pi_X)_*\mmu=\mu$ for each~$k\ge0$. The construction readily implies that $\{\hat u_k,k\ge0\}$ is an $X$-valued stationary process whose law is equal to~$\mu$ at any time $k\ge0$.
\end{proof} 

Relation~\eqref{stationary-RDS} implies that the process $\{\hat u_k\}_{k\in\Z_+}$ is feebly mixing in the sense of Section~\ref{s-SR-mixing}; cf.\ the proof of Proposition~\ref{p-support}.

\begin{proposition}[Convergence of trajectories]\label{p-CT}
Let~$u$ be an $X$-valued random variable independent of~$\{\eta_k\}$ and let $\{u_k,k\in\Z_+\}$ be the trajectory defined by~\eqref{stationary-RDS}, \eqref{IC-stationary}. Then, for any integer $s\ge1$, the process $[u_k,\dots,u_{k+s}]$ converges in law in the dual-Lipschitz metric to $[\hat u_0,\dots,\hat u_s]$, where~$\hat u_k$ is the stationary trajectory described above. Moreover, the convergence is uniform with respect to all initial conditions. 	
\end{proposition}

\begin{proof}
Let us consider again the extended system~\eqref{RDS-example} and its stationary trajectory $\{V_k,k\in\Z_+\}$ as in Theorem~\ref{t-mixing-kantorovich}. Let~$U_k$ be the trajectory of~\eqref{RDS-example} issued from an initial state whose law equals $\lambda\otimes\sigma\in\PP(\XXXX)$, where $\lambda=\DD(u)\in\PP(X)$. Projecting estimate~\eqref{convergence-V-s} to the first component, we obtain the required convergence.
\end{proof}

Finally, let us notice that, by Theorem~\ref{t-mixing-dissipation}, Theorem~\ref{t-mixing-kantorovich} applies to the $U$-process in the space $\XXXX=X\times\EE$ and, hence, the latter is exponentially mixing. So, for any $V=(v,\xxi)\in X\times\EE$, the corresponding trajectory~$U_k$ defined by~\eqref{RDS-example} converges to~$\mmu$ in the dual-Lipschitz metric. The following proposition, whose proof may be obtained by literal repetition of the arguments in Section~\ref{s-proofMT}, shows that the initial point~$V$ can be taken from the larger set~$X\times\bKK$, provided that the strong recurrence to zero holds for~$\xxi\in\bKK$; cf.\ Theorem~\ref{p-stationary-convergence} below.

\begin{proposition}\label{p-generalisation}
Let $\{Q(\xxi;\cdot),\xxi\in\bKK\}\subset\PP(\KK)$ be a transition probability on~$\KK$ that coincides with the conditional probability  introduced in Subsection~\ref{ss-RDS} for $\xxi\in\EE$. Suppose that the hypotheses of Theorem~\ref{t-mixing-dissipation} are fulfilled, with  inequality~\eqref{to-zero} replaced by the stronger condition 
	\begin{equation}\label{to-zero-K}
		\inf_{\xxi\in\bKK}Q_s^n\bigl(\xxi;\OO_\delta(\mathbf{0}_n)\bigr)>0,
	\end{equation}
	where $Q_s^n(\xxi;\cdot)$ is defined in~\eqref{Qkm}. Then there are positive numbers~$\gamma$ and~$C$ such that 
	\begin{equation*}
		\sup_{V\in X\times\bKK}\|\DD(U_k)-\mmu\|_L^*\le Ce^{-\gamma k}, \quad k\ge0,
	\end{equation*}
	where $\mmu\in\PP(\XXXX)$ is the stationary measure of the $U$-process on~$X\times\EE$ constructed in the proof of Theorem~\ref{t-mixing-dissipation}, and $\|\cdot\|_L^*$ stands for the dual-Lipschitz metric over~$X\times\bKK$. 
\end{proposition}

\subsection{Proof of Lemma~\ref{l-right-inverse}}
\label{s-proof-lemma}

The result would follow immediately from Theorem~2.8 in~\cite{KNS-gafa2020} if~$E$ was a Hilbert space. Since this is not the case, we need to proceed differently. 

Let us fix a compact set $Y\subset O$ and a number $\e>0$, and denote by $\dd_H(f,H')$ the distance of~$f$ from a subspace~$H'\subset H$. Since the union $\cup_n F_n$ is dense in~$E$, then the continuous functions $(y,f)\mapsto \dd_H\bigl(f,A(y)F_n)\bigr)$ pointwise monotonically converge to zero as~$n\to\infty$.  By the Dini theorem, this convergence is uniform on compact sets.  So there exists an integer $n\ge1$ such that
\begin{equation}\label{approx-G}
	\sup_{y\in Y}\sup_{f\in \overline{B_G(1)}}\dd_H\bigl(f,A(y)F_n)\bigr)\le \e.
\end{equation}
Let us endow~$F_n$ with an inner product and denote by~$A_n(y)$ the restriction of~$A(y)$ to~$F_n$. Then we can consider the adjoint operator $A_n(y)^*:H\to F_n$ and, given $\delta>0$, define the family of operators
$$
B(y,\delta)=A_n(y)^*\bigl(A_n(y)A_n(y)^*+\delta I\bigr)^{-1}:H\to F_n,
$$
which form a $C^1$ function of $(y,\delta)\in O\times (0,+\infty)$ with range in~$\LL(H,F_n)$. Suppose we have proved that, for any $y\in Y$ and $f\in \overline{B_G(1)}$, 
\begin{equation}\label{limsup-e}
	\limsup_{\delta\to0^+}\|D(y,\delta)f-f\|_H\le 2\e,
\end{equation}
where we set $D(y,\delta)=A(y)B(y,\delta)$. Noting that $\|D(y,\delta)\|_{\LL(H)}\le 1$ for any~$(y,\delta)$ and denoting by $f_1,\dots,f_N\in G$ an $\e$-net in~$\overline{B_G(1)}$, we can write 
$$
\|D(y,\delta)f-f\|_H\le 2\e+\|D(y,\delta)f_j-f_j\|_H,
$$
where $j\in\{1,\dots,N\}$ is such that $\|f-f_j\|_H\le\e$. Combining this with~\eqref{limsup-e}, we see that, for any $y\in Y$ and a sufficiently small $\delta_y>0$, we have 
$$
\sup_{f\in \overline{B_G(1)}}\|D(y,\delta)f-f\|_H\le 3\e
\quad\mbox{for $0<\delta\le\delta_y$}. 
$$
Since the function $\delta\mapsto \|D(y,\delta)f-f\|_H$ decreases with $\delta>0$, using the compactness of~$Y$ and the continuity of the function $y\to D(y,\delta)$, we can find $\delta_\e>0$ such that 
$$
\sup_{y\in Y}\sup_{f\in \overline{B_G(1)}}\|D(y,\delta_\e)f-f\|_H\le 4\e. 
$$
By homogeneity, this implies the required inequality~\eqref{ARI} with $B_\e(y)=B(y,\delta_\e)$, in which~$\e$ on the right-hand side is replaced with~$4\e$. Thus, it remains to establish~\eqref{limsup-e} for fixed~$y$ and~$f$. 

In view of~\eqref{approx-G}, there is $f_\e\in A_n(y)F_n$ such that $\|f-f_\e\|_H\le \e$. Since the norm of~$D(y,\delta)$ does not exceed~$1$, we obtain
\begin{equation}\label{estimate-dye}
\|D(y,\delta)f-f\|_H\le 2\e+\|D(y,\delta)f_\e-f_\e \|_H.	
\end{equation}
Using the spectral theorem for self-adjoint operators, it is straightforward to check that $\|D(y,\delta)f_\e-f_\e \|_H\to0$ as $\delta\to0^+$; see the proof of Lemma~2.7 in~\cite{KNS-gafa2020}. Combining this with~\eqref{estimate-dye}, we arrive at~\eqref{limsup-e}. This completes the proof of the lemma.

\section{Stationary processes satisfying the mixing hypothesis} 
\label{s-mixingnoise}
In this section, we discuss Hypotheses~\hyperlink{Fel}{(SF)}, \hyperlink{DLP'}{(DLP$'$)}, and~\hyperlink{SRZ}{(SRZ)} and construct stationary processes for which they are satisfied. This type of questions were intensively studied in the middle of last century; see the papers~\cite{dobrushin-1968,dobrushin-1970,dobrushin-1974} and the references therein. Here  we prove that the existence and uniqueness (in law) of a stationary process with  prescribed conditional laws follows from Theorem~\ref{t-mixing-kantorovich}. We also discuss the relation of properties~\hyperlink{DLP'}{(DLP$'$)} and~\hyperlink{SRZ}{(SRZ)} with various concepts of mixing for stationary processes. 

\subsection{Existence and uniqueness}
Let $E$ be a separable Banach space, $\KK\subset E$ be a compact subset, and $\bKK=\KK^{\Z_-}$ be the direct product of countably many copies of~$\KK$ with the distance~$\dd$ as in~\eqref{distance-EE}. We begin with the following result due to Dobrushin~\cite{dobrushin-1970}.

\begin{proposition}\label{p-dobrushin}
	Let $\{Q(\xxi; \cdot),\xxi\in\bKK\}\subset\PP(E)$ be a transition probability from~$\bKK$ to~$E$ for which  $\supp Q(\xxi; \cdot)\subset\KK$ for any $\xxi\in\bKK$ and Hypotheses~\hyperlink{Fel}{\rm(SF)} and~\hyperlink{DLP'}{\rm(DLP$'$)} hold with $\EE=\bKK$,  and let ${\boldsymbol Q}_k(\xxi;\cdot)\in\PP(E^k)$ be defined by~\eqref{conditional-law} for each $k\ge1$. Then there is an $E$-valued stationary random process $\{\eta_l\}_{l\in\Z}$ such that, for every $k\ge1$, the conditional law of~$(\eta_1,\dots,\eta_k)$ given the past $\xxi\in \bKK$ is equal to~${\boldsymbol Q}_k(\xxi;\cdot)$ for $\xxi\in\supp\sigma$, where~$\sigma$ is defined by~\eqref{sigma-etak}.
\end{proposition}

This proposition follows from Theorem~1 in~\cite[Section~2]{dobrushin-1970} (applied to Case~2 discussed there). We do not give further details, referring the reader to the original paper and noting that the law of the constructed process may not be unique, and that Hypothesis~\hyperlink{SRZ}{(SRZ)} is not necessarily satisfied for it. The following result (whose proof does not use Proposition~\ref{p-dobrushin}) shows that the uniqueness holds, provided that the conditional measures, in addition to~\hyperlink{Fel}{\rm(SF)} and~\hyperlink{DLP'}{\rm(DLP$'$)}, satisfy a {\it strong recurrence\/} property, which is slightly more general than~\hyperlink{SRZ}{(SRZ)}; cf.\ Proposition~\ref{p-generalisation}.  Moreover, the result and its proof remain true in the non-stationary case, as is mentioned in Remark~\ref{r_intro}~(2). 

\begin{theorem}\label{p-stationary-convergence}
In addition to the conditions of Proposition~\ref{p-dobrushin}, let us assume that $\{Q(\xxi;\cdot),\xxi\in\bKK\}$ possesses the following property:
	\begin{description}
		\item [\hypertarget{SR}{(SR)} Strong recurrence.] \sl For any integer $n\ge1$ and any $\delta>0$, there is a vector $\vec{\xi}_n=(\bar\xi_1,\dots,\bar\xi_n)\in E^n$ and an integer~$s\ge0$ such that 
\begin{equation}\label{xi-s-transition}
	\inf_{\xxi\in\bKK}Q_s^n\bigl(\xxi;\OO_\delta(\vec{\xi}_n)\bigr)>0,
\end{equation}
where the measures $Q_s^n\bigl(\xxi;\cdot)$ are constructed from the transition probability $Q\bigl(\xxi;\cdot)$ using~\eqref{conditional-law} and~\eqref{Qkm}.
	\end{description}
Then there is an $E$-valued stationary random process $\{\eta_k\}_{k\in\Z}$ such that the conditional measure of~$\eta_1$ given the past~$\xxi\in\bKK$ is equal to~$Q(\xxi;\cdot)$ for any $\xxi\in \EE:=\supp\sigma$, where $\sigma=\DD(\{\eta_k,k\in\Z_-\})$. Moreover, the law of the process~$\{\eta_k\}$ is uniquely defined, and there are positive numbers~$C$ and~$\gamma$ such that, for any $m,k\ge0$,
	\begin{equation}\label{convergence-condlaw}
	\sup_{\xxi\in\bKK}\,\bigl\|Q_k^m(\xxi;\cdot)-\ssigma_{\!m}\bigr\|_L^*\le C\,e^{-\gamma k},	
	\end{equation}
	where $\ssigma_{\!m}\in\PP(E^m)$ is the law of $(\eta_1,\dots,\eta_m)$. 
\end{theorem}

\begin{proof}
%
{\it Step~1: Proof of~\eqref{convergence-condlaw}\/}. Let us set $\XXXX=\bKK$ and denote by $\{\zeta^\xxi,\xxi\in\XXXX\}$ a random field on~$\KK$ such that $\DD(\zeta^\xxi)=Q(\xxi;\cdot)$ for any $\xxi\in\XXXX$. Consider the following RDS on~$\XXXX$, defined on some suitable probability space~$(\widehat\Omega,\widehat\FF,\widehat\IP)$:
	\begin{equation}\label{RDS-noise}
		\xxi_k=\bigl(\xxi_{k-1},\zeta_k^{\xxi_{k-1}}\bigr), \quad k\ge1.
	\end{equation}
	Here $\zeta_k^\xxi(\widehat\omega)=\zeta^\xxi(\omega_k)$ for any $\xxi\in\XXXX$ and $\widehat\omega\in\widehat\Omega$; cf.~\eqref{RDS-example}. We note that~\eqref{RDS-noise} can be written in the form~\eqref{RDS}, where the mapping $\SSS:\XXXX\times \KK\to \XXXX$ is given by $\SSS(\xxi,\eta)=(\xxi,\eta)$. By Proposition~\ref{p-markov-chain}, it defines a Markov process in~$\XXXX$. It follows from~\eqref{RDS-noise} that 
\begin{equation}\label{one-step}
	P_1(\xxi;\cdot)=\delta_{\xxi}\otimes Q(\xxi;\cdot)\in \PP(\bKK)\otimes\PP(\KK)\simeq\PP(\XXXX),
\end{equation}
so that 
\begin{equation}\label{transition-k-noise}
P_l(\xxi;\cdot)=\delta_\xxi\otimes{\boldsymbol Q}_l(\xxi;\cdot)\in \PP(\bKK)\otimes\PP(\KK^l)\simeq\PP(\XXXX)
\end{equation}
for $l\ge1$. In particular, if we write $\xxi_k$ as $(\xi_k^i,i\le k)$, then taking $l=k+m$ in~\eqref{transition-k-noise} and projecting to the last~$m$ components, we see that 
\begin{equation}\label{xi-k-projection}
	\DD\bigl(\{\xi_k^i, k+1\le i\le k+m\}\bigr)=Q_k^m(\xxi;\cdot).
\end{equation}

We claim that the RDS~\eqref{RDS-noise} satisfies all the conditions of Theorem~\ref{t-mixing-kantorovich}. Indeed, Hypothesis~\hyperlink{GCP}{(GCP)} follows from~\hyperlink{SR}{(SR)} and an argument used in Step~2 of Section~\ref{s-proofMT}, with~${\mathbf0}_n$ replaced by $\vec\xi_n=(\bar\xi_1,\dots,\bar\xi_n)$. Furthermore, Hypothesis~\hyperlink{LAC}{(LAC)} with $F=\{0\}$ is trivial since we can take $\varPhi\equiv0$,  and~\hyperlink{DLP}{(DLP)} follows from~\hyperlink{DLP'}{(DLP$'$)}. Thus, the conclusion of Theorem~\ref{t-mixing-kantorovich} holds for~\eqref{RDS-noise}, so that~\eqref{RDS-noise} has a unique stationary measure $\sigma\in\PP(\XXXX)$, and there are positive numbers~$C$ and~$\gamma$ such that
\begin{equation}\label{xxi-k-ssigma}
	\sup_{\xxi\in\XXXX}\|\DD(\xxi_k)-\sigma\|_L^*\le Ce^{-\gamma k}, \quad k\ge0.
\end{equation}
Denoting by~$\ssigma_m\in\PP(\KK^m)$ the projection of~$\sigma$ to the last $m$ components and using~\eqref{xi-k-projection}, we obtain~\eqref{convergence-condlaw} as a direct consequence of~\eqref{xxi-k-ssigma}. 

\smallskip
{\it Step~2: Stationarity and construction of the process\/}. 
We now prove that $\sigma\in\PP(\XXXX)$ is shift invariant. To this end, it suffices to show that
\begin{equation}\label{shift-invariance}
	\ssigma_{m+1}(\Gamma\times\KK)=\ssigma_m(\Gamma)
	\quad\mbox{for any $m\ge1$ and $\Gamma\in\PP(\KK^m)$}.  
\end{equation}
The definition of~$Q_k^m$ (see~\eqref{Qkm}) implies that 
$$
Q_k^{m+1}(\xxi;\Gamma\times\KK)=Q_k^m(\xxi;\Gamma)\quad
\mbox{for any $\Gamma\in\BB(\KK^m)$}. 
$$
Passing to the (weak) limit as $k\to\infty$ in this relation (regarded as equality of measures) and using~\eqref{convergence-condlaw}, we arrive at~\eqref{shift-invariance}. 

Thus, $\sigma$ is shift invariant, so that, by the Kolmogorov theorem, it can be extended to a shift invariant measure on~$\KK^\Z$, which is denoted by~$\ssigma$.  We now define a stationary sequence~$\{\eta_k,k\in\Z\}$ as the canonical process corresponding to~$\ssigma$. In the next step, we show that the conditional measures of~$\{\eta_k\}$ are given by~$Q(\xxi;\cdot)$, establishing thus the existence claim of the theorem.\footnote{By Proposition~\ref{p-dobrushin}, the existence is valid under weaker hypotheses. Our proof does not use Dobrushin's result and is given to make the presentation self-contained.}

\smallskip
{\it Step~3: Conditional measures\/}. Let us consider a trajectory of~\eqref{RDS-noise} issued from a random initial condition~$\xxi_0$ that depends only on~$\omega_0$ and is distributed as~$\sigma\in\PP(\XXXX)$. In view of the stationarity of~$\sigma$, the Markov property, and relation~\eqref{one-step}, for any bounded continuous function $F:\XXXX\times\KK\to\R$, we have 
\begin{align*}
\int_\XXXX F(\xxi)\sigma(\dd\xxi)&=\E F(\xxi_1)
=\E\bigl( \E\{F(\xxi_0,\zeta_1^{\xxi_0})|\,\FF_0\}\bigr)\\
&=\E\int_\KK F(\xxi_0,\xi)P_1(\xxi_0;\dd\xi)
=\int_\XXXX \sigma(\dd\xxi)\int_\KK F(\xxi,\xi)Q(\xxi;\dd\xi). 
\end{align*}
We see that $Q(\xxi;\cdot)$ is indeed the conditional measure of~$\eta_1$, given the past~$\eeta=\xxi$.

\smallskip
{\it Step~4: The uniqueness\/}.  Suppose there is another $\KK$-valued stationary process $\{\eta_k'\}_{k\in\Z}$ whose conditional law is given by~$Q(\xxi;\cdot)$ for $\xxi\in\EE'$, where $\EE'\subset\bKK$ is the support of the law of $\{\eta_k',k\in\Z_-\}$, and denote by~$\ssigma'\in\PP(\KK^\Z)$ its law. We wish to prove that $\ssigma'=\ssigma$. 
	
	Given an integer $m\ge1$, we denote by $\ssigma_{\!m}'\in\PP(\KK^m)$ the law of the random vector $(\eta_1',\dots,\eta_m')$. For any integers $k,m\ge1$ and any function $f\in C_b(E^m)$, we can write 
	$$
	\int_{\KK^m}f(\yyy)\ssigma_{\!m}'(\dd\yyy)
	=\int_\bKK \int_{\KK^m}f(\yyy)Q_k^m(\xxi;\dd\yyy)\sigma'(\dd\xxi),
	$$
	where $\sigma'\in\PP(\bKK)$ stands for the projection of~$\ssigma'$ to~$\bKK$. Passing to the limit as $k\to\infty$ and using~\eqref{convergence-condlaw}, we obtain $\langle f,\ssigma_{\!m}'\rangle=\langle f,\ssigma_{\!m}\rangle$. Since this is true for any $f\in C_b(E^m)$, we see that $\ssigma_{\!m}'=\ssigma_{\!m}$ for all $m\ge1$, whence we conclude that $\ssigma=\ssigma'$. The proof of the theorem is complete. 
	\end{proof}

\begin{corollary}\label{c-process}
	Let $Q(\xxi;\cdot)$ be a transition probability from~$\bKK$ to~$\KK$ that satisfies Hypotheses~\hyperlink{Fel}{\rm(SF)}, \hyperlink{DLP'}{\rm(DLP$'$)}, and~\hyperlink{SRZ}{\rm(SRZ)}. Then there is a unique in distribution stationary random process~$\{\eta_k\}_{k\in\Z}$ whose conditional measure coincides with~$Q(\xxi;\cdot)$ for $\xxi\in\supp\sigma$, where $\sigma$ is given by~\eqref{sigma-etak}. 
\end{corollary}

Since it is easy to construct numerous examples of transition probabilities $Q(\xxi;\cdot)$ from~$\bKK$ to~$\KK$ which meet assumptions~\hyperlink{Fel}{(SF)}, \hyperlink{DLP'}{(DLP$'$)} and~\hyperlink{SRZ}{\rm(SRZ)}, then there are plenty of stationary processes~$\{\eta_k\}$ with compact support that satisfy the hypotheses imposed on the random input in Theorem~\ref{t-mixing-dissipation}.

\subsection{Strong recurrence, strong irreducibility, and mixing}
\label{s-SR-mixing}
The aim of this section is to prove implications~\eqref{mixing-recurrence}.  We shall say that an $E$-valued stationary process $\eeta:=\{\eta_k\}_{k\in\Z}$ (or $\eeta:=\{\eta_k\}_{k\in\Z_+}$) is {\it feebly mixing\/} if for any integers $n,m\ge1$ and any functions $f\in C_b(E^n)$ and $g\in C_b(E^m)$  we have 
\begin{equation}\label{mixing-fg-old}
	\lim_{k\to\infty}\E\bigl(f(\eeta)g(\theta_k\eeta)\bigr)=\E\,f(\eeta)\, \E\,g(\eeta),
\end{equation}
where $\theta_k\eeta=\{\eta_{j+k},j\in\Z\}$ is the shifted process, and for a function $h:E^l\to\R$ we set $h(\eeta)=h(\eta_1,\dots,\eta_l)$. We denote by~$\ssigma_{\!m}$ and~$\sigma$ the  laws of~$(\eta_1,\dots,\eta_m)$ and $\{\eta_k,k\in\Z_-\}$, respectively. 

\begin{proposition}\label{p-support}
	Let $\eeta=\{\eta_k\}_{k\in\Z}$  be an $E$-valued stationary process whose conditional probabilities satisfy Hypotheses~\hyperlink{Fel}{\rm(SF)}, \hyperlink{DLP'}{\rm(DLP$'$)}, and~\hyperlink{SR}{\rm(SR)}. Then $\eeta$ is feebly mixing. 
\end{proposition}

\begin{proof}
	The main idea of the proof is well known, so that we skip the details. Let $f\in C_b(E^n)$ and $g\in C_b(E^m)$. For any $k>n$ and $\xxi\in\EE$, the conditional expectation given the past can be written as
	$$
	\E\bigl\{g(\theta_k\eeta)\,\big|\,\eeta^n=\xxi\bigr\}=\int_{E^m}Q_{k-n}^m(\xxi;\dd\vec{y}_m)g(\vec{y}_m),
	$$
	where $\vec{y}_m=(y_1,\dots,y_m)$ and $\eeta^n=\{\eta_j,j\le n\}$. Denoting by~$\{\FF_k\}_{k\in\Z}$ the natural filtration associated with~$\eeta$ and taking the conditional expectation with respect to~$\FF_n$, we obtain 
	\begin{align}
	\E\bigl(f(\eeta)g(\theta_k\eeta)\bigr)
	&=\E\bigl(f(\eeta)\,\E\bigl\{g(\theta_k\eeta)\,\big|\,\FF_n\bigr\}\bigr)\notag\\
	&=\E\biggl\{f(\eeta)\int_{E^m}Q_{k-n}^m(\eeta^n,\dd\vec{y}_m)g(\vec{y}_m)\biggr\}.\label{MV-mixing} 
	\end{align}
	In view of Theorem~\ref{p-stationary-convergence}, the integral over~$E^m$ converges (uniformly in~$\eeta^n\in\EE$) as $k\to\infty$ to~$\langle g,\ssigma_{\!m}\rangle=\E\,g(\eeta)$. Passing to the limit in~\eqref{MV-mixing} as $k\to\infty$, we arrive at~\eqref{mixing-fg-old}. 
\end{proof}

\begin{remark}
In fact, 
our argument shows  that under the assumptions of Proposition~\ref{p-support}, the rate of convergence in~\eqref{mixing-fg-old} is exponential. 
\end{remark}

Now let  $\eeta$ be a stationary process whose conditional probabilities satisfy Hypothesis~\hyperlink{Fel}{(SF)}. We shall say that $\eeta$ is {\it strongly feebly mixing\/} if, for each $m\ge1$,
\begin{equation}\label{uniform-mixing-g}
	\lim_{k\to\infty}\sup_{\xxi,g}\,\biggl|\int_{E^m}Q_k^m(\xxi;\dd \vec{y}_m)g(\vec{y}_m)-\E\,g(\eeta)\biggr|=0,
\end{equation}
where the supremum is taken over all points $\xxi\in\EE=\supp\sigma$ and all the functions $g\in C_b(E^m)$ with $\|g\|_\infty\le1$. In this case, it follows from~\eqref{MV-mixing} that, for fixed~$m$ and~$n$, the limit in~\eqref{mixing-fg-old} holds uniformly in continuous functions~$f$ and~$g$ whose absolute values are bounded  by~$1$. Thus, when property~\hyperlink{Fel}{(SF)} is satisfied, the concept of {\it strong feeble mixing\/} defined above is close to the usual concept of {\it strong mixing\/}, for which the validity of~\eqref{uniform-mixing-g} is required uniformly in~$m\ge1$; see~\cite[Definition~17.2.2]{IL1971}. Finally, let us note the argument used in~\cite{KS-TV-2025} to prove convergence in the total variation metric applies also in this case and allows one to establish~\eqref{uniform-mixing-g} for bounded measurable functions~$g$.

In view of Proposition~\ref{p-CT}, a stationary weak solution of~\eqref{stationary-RDS} is strongly feebly mixing. The following result prove an irreducibility property for such processes. 

\begin{proposition}\label{p-mixing-recurrence}
Any strongly feebly mixing $E$-valued stationary process~$\eeta$ satisfying Hypothesis~\hyperlink{DLP'}{\rm(DLP$'$)} is strongly irreducible in the following sense: 
\begin{description}
	\item[\hypertarget{SI}{(SI)}]\sl 
For any $n\in\N$ and~$\delta>0$ there is an integer $s\ge1$ such that
\begin{equation*}\label{prob-zero-supp}
	\inf_{\xxi\in\supp\sigma}Q_s^n\bigl(\xxi;\OO_\delta(\vec{\xi}_n)\bigr)>0
\end{equation*}
	for any $\vec{\xi}_n=(\xi_1,\dots,\xi_n)\in\supp\ssigma_{\!n}$, where $\sigma\in\PP(\EE)$ is defined in~\eqref{sigma-etak}. 
\end{description}
\end{proposition}

\begin{proof}
Let us fix an integer $n\ge1$, a point $\vec{\xi}_n\in\supp\ssigma_{\!n}$, and any number~$\delta>0$. We take any  function $g\in C_b(E^n)$ supported by $\OO_{\delta}(\vec{\xi}_n)$ such that $0\le g\le 1$ and $g(\vec{y}_n)=1$ for $\vec{y}_n\in \OO_{\delta/2}(\vec{\xi}_n)$. In this case, we have 
$$
p:=\E\, g(\eeta)>0, \quad 
Q_s^n\bigl(\xxi;\OO_{\delta}(\vec{\xi}_n)\bigr)\ge \int_{E^n}Q_s^n(\xxi;\dd \vec{y}_n)g(\vec{y}_n).
$$
Combining this with~\eqref{uniform-mixing-g}, we see that 
$$
\inf_{\xxi\in\supp\sigma}Q_s^n\bigl(\xxi;\OO_{\delta}(\vec{\xi}_n)\bigr)\ge p/2,
$$
provided that $s\ge1$ is sufficiently large. This completes the proof. 
\end{proof}

\subsection{Examples of processes satisfying the hypotheses of Section~\ref{s-formulation-MR}}\label{s-examplesprocesses}
Corollary~\ref{c-process} implies the existence of many stationary processes satisfying~\hyperlink{Fel}{(SF)}, \hyperlink{DLP'}{(DLP$'$)}, and~\hyperlink{SRZ}{\rm(SRZ)}. In Subsections~\ref{ss-DTMN} and~\ref{ss-DTMA} below, we describe two classes of stationary processes possessing those properties. 

\subsubsection{Discrete-time Markovian noises}\label{ss-DTMN}
	Let $E$ be a finite-dimensional space and let $\eeta=
	\{\eta_k\}_{k\ge0}$ be a stationary Markov process in~$E$ corresponding to transition probabilites $\{P_k(\xi;\cdot),k\ge0\}$ that possess the following properties:
	\begin{itemize}\sl 
	\item the law of~$\eta_1$ has a compact support, and there is a Lipschitz continuous function $\rho:E\times E\to\R_+$ such that $P_1(\xi;\dd y)=\rho(\xi,y)\ell(\dd y)$ for any $\xi\in E$;
	\item there is a ball $B\subset E$, a point $a\in E$, an integer $m\ge1$, and a number $p>0$ such that $\rho(\xi,a)>0$ for any $\xi\in B$, and $P_m(\xi;B)\ge p$ for any $\xi\in E$. 
	\end{itemize}
In this case, Hypotheses~\hyperlink{Fel}{(SF)} and~\hyperlink{DLP'}{(DLP$'$)} are satisfied for~$\eeta$. Moreover, using Theorem~\ref{p-stationary-convergence}, it is not difficult to prove that~\eqref{uniform-mixing-g} holds with $\xxi\in\EE$, so that~$\eeta$ is strongly feebly mixing. Then, by Proposition~\ref{p-mixing-recurrence}, the process~$\eeta$ satisfies the property of strong irreducibility~\hyperlink{SI}{(SI)}. Moreover, if the second condition mentioned above is fulfilled for the point $a=0$ and some ball $B=B_E(\delta)$ with $\delta>0$, then a simple calculation based on the Kolmogorov--Chapman relation shows that Hypothesis~\hyperlink{SRZ}{(SRZ)} is also satisfied. 

\subsubsection{Discrete-time moving averages}\label{ss-DTMA}
	Let $\zzeta=\{\zeta_k\}_{k\in\Z}$ be a sequence of i.i.d.\ random variables in finite-dimensional space~$E$ whose law possesses a Lipschitz continuous density $\rho(y)$. We assume that the support of~$\rho$ is compact. Let us fix any sequence $\{a_l\}_{l\ge1}$ exponentially going to zero and define 
\begin{equation}\label{etak-zetak}
	\eta_k=\zeta_k+\sum_{l=1}^\infty a_l\zeta_{k-l}, \quad k\in\Z.
\end{equation}
Obviously $\eeta=\{\eta_k\}$ is a stationary process in~$E$. To simplify the presentation, we further assume that $\sum_l|a_l|<1$. In this case, denoting by~$\ell_\infty(E)$ the space of bounded $E$-valued sequences indexed by~$\Z_-$, we see that the linear operator $A:\ell_\infty(E)\to\ell_\infty(E)$ taking $\{\zeta_j,j\in\Z_-\}$ to the sequence~$\{\eta_k, k\in\Z_-\}$ defined by~\eqref{etak-zetak} is invertible, and the inverse operator~$B$ has the form 
	$$
	(B\xxi)_k=\xi_k+\sum_{l=1}^\infty b_l\xi_{k-l}, \quad k\in\Z_-,
	$$
	where $\{b_l,l\in\N\}$ is another  exponentially decaying sequence of real numbers. It follows that the conditional law of~$\eeta$ can be written as
	$$
	Q(\xxi;\dd y)=\rho(y-h(\xxi)), \quad 
	h(\xxi)=\sum_{l=1}^\infty a_l(B\xxi)_{1-l}=-\sum_{l=1}^\infty b_l\xi_{1-l},
	$$
	where $\xxi=(\xi_l,l\in\Z_-)\in\EE$. Since~$\rho$ is Lipschitz continuous, it follows that $Q(\xxi;\dd y)$ has a Lipschitz continuous density, and so~$\{\eta_k\}$ satisfies~\hyperlink{DLP'}{(DLP$'$)} and~\hyperlink{Fel}{(SF)}. Furthermore, since an i.i.d.\ random process satisfies~\hyperlink{SI}{(SI)}, using the fact that the processes~$\eeta$ and~$\zzeta$ can be reconstructed from each other with the help of the continuous linear operators~$A$ and~$B$, it is not difficult to conclude that~\hyperlink{SI}{(SI)} holds also for~$\eeta$. If, in addition, the support of the law of~$\zeta_k$ contains $0\in E$, then~\eqref{etak-zetak} readily implies that Hypothesis~\hyperlink{SRZ}{(SRZ)} is satisfied. 

\subsubsection{Continuous-time noise}
\label{ss-CTN}
In this section, we denote by~$J$ the interval $[0,1)\subset\R$ and by~$E$ the Lebesgue space $L^2(J)$. Let~$\{\varphi_l\}_{l\ge1}$  be an orthonormal basis in~$E$ and $\{\zeta_l\}_{l\ge1}$ be a sequence of independent scalar random variables whose laws possess Lipschitz continuous densities~$d_l:\R\to\R_+$ with respect to Lebesgue measure on~$\R$. We assume that $|\zeta_l|\le a_l$ almost surely for any $l\ge1$, where the numbers $a_l>0$ are such that $\sum_la_l^2<\infty$. Thus, the function~$d_l$ vanishes outside the interval~$[-a_l,a_l]$, and  the random series 
\begin{equation*}\label{zeta-t-phi}
\zeta(t):=\sum_{l=1}^\infty\zeta_l\varphi_l(t), \quad t\in J	
\end{equation*}
converges almost surely in~$E$. We denote by $\nu\in\PP(E)$ the law of the random variable~$\zeta$. Let us write~$\KK$ for the support of~$\nu$. This is a compact subset of~$E$. Next, we fix any number $\iota>1$ and define a distance~$\dd(\xxi,\xxi')$ on $\bKK=\KK^{\Z_-}$ by relation~\eqref{distance-EE}. 

\begin{proposition}\label{p-continuous-time}
Let the above hypotheses be fulfilled, let $g:\bKK\times E\to\R_+$ be a Lipschitz continuous function such that $g(\xxi,y)\ge c>0$ for all $\xxi\in\bKK$ and $y\in E$, and let $m(\xxi)=\int_Eg(\xxi,y)\nu(\dd y)$. Then the transition probabilities 
	\begin{equation}\label{measure-Qxi}
		Q(\xxi;\dd y)=\rho(\xxi,y)\nu(\dd y), \quad \rho(\xxi,y)=m(\xxi)^{-1}g(\xxi,y),
	\end{equation}
	satisfy Hypotheses~\hyperlink{Fel}{\rm(SF)} and~\hyperlink{DLP'}{\rm(DLP$'$)}, in which $F_n$ is the vector span of $\varphi_1,\dots,\varphi_n$, and~$F_n^\dagger$ is~$F_n^\bot$ -- the orthogonal complement to~$F_n$ in~$E$. Moreover, the densities of the  conditional measures  $Q_n(\xxi,y_n^\dagger;\cdot)$ entering~\eqref{decomposition-Q} are given by the formula
	\begin{equation}\label{conditional-density}
		\rho_n(\xxi,y_n^\dagger,y_n)=\frac{g(\xxi,y_n,y_n^\dagger)D_n(y_n)}{\int_{F_n}g(\xxi,z,y_n^\dagger)\nu_n(\dd z)},
	\end{equation}
	where $\nu_n$ is the projection of~$\nu$ to~$F_n$, and $D_n$ stands for the product of the functions $d_1,\dots,d_n$.
\end{proposition}

\begin{proof}
 The validity of~\hyperlink{Fel}{(SF)} follows immediately from~\eqref{measure-Qxi}. 
In view of~\eqref{measure-Qxi}, for any integer $n\ge1$ and arbitrary bounded continuous functions $f:F_n\to\R$ and $g:F_n^\bot \to\R$, we can write 
\begin{multline} \label{integral-fg}
	\int_Ef(y_n)g(y_n^\dagger)Q(\xxi;\dd y)\\
	=\int_{F_n^\bot}g(y_n^\dagger)\biggl\{\int_{F_n}f(y_n)\rho(\xxi,y_n,y_n^\dagger)D_n(y_n)\ell_n(\dd y_n)\biggr\}\nu_n^\bot(\dd y_n^\dagger),	
\end{multline}
where $\nu_n^\bot$ is the projection of~$\nu$ to~$F_n^\bot$, and $\ell_n$ stands for  the Lebesgue measure on~$F_n$. On the other hand, the projection of~$Q(\xxi;\cdot)$ to~$F_n^\bot$ is given by 
$$
Q_n^\dagger(\xxi;\dd y_n^\dagger)=\biggl\{\int_{F_n}\rho(\xxi,z,y_n^\dagger)\nu_n(\dd z)\biggr\}\nu_n^\bot(\dd y_n^\dagger).
$$
Combining this with~\eqref{integral-fg}, we arrive at relation~\eqref{conditional-density} for the density of the measure~$Q(\xxi;y_n^\dagger,\cdot)$. The Lipschitz continuity of~$\rho_n$ follows from similar property of the function~$g$ and the inequality $g\ge c$. Finally, the density of the union~$\cup_nF_n$ is a consequence of the construction. 
\end{proof}

\begin{corollary}\label{c-CT-example}
	In addition to the conditions of Proposition~\ref{p-continuous-time}, let us assume that $0\in \supp d_l$ for any $l\ge1$. Then Hypothesis~\hyperlink{SRZ}{\rm(SRZ)} holds for the family $\{Q(\xxi;\cdot),\xxi\in\bKK\}$. Moreover, for any $m\ge1$ and $\delta>0$, inequality~\eqref{to-zero-K} holds with $s=0$.
\end{corollary}

\begin{proof}
Since the support of~$d_l$ contains zero, we see that $Q(\xxi;B_E(0,\delta))>0$ for any $\xxi\in \bKK$. Combining this with~\eqref{conditional-law}, for any $\xxi\in\bKK$ and $m\ge1$, we derive
$$
{\boldsymbol Q}_m\bigl(\xxi;\OO_\delta(\mathbf{0}_m)\bigr)>0. 
$$
By the portmanteau theorem, the function taking~$\xxi\in \bKK$ to the left-hand side of this relation is lower semicontinuous and, hence, separated from zero on any compact set. This implies the required inequality~\eqref{xi-s-transition}. 
\end{proof}

We now construct a continuous-time process whose restrictions to the intervals $J_k:=[k-1,k)$, $k\in\Z$ satisfy Hypotheses~\hyperlink{Fel}{(SF)}, \hyperlink{DLP'}{(DLP$'$)}, and~\hyperlink{SRZ}{(SRZ)}. To this end, we first use Corollary~\ref{c-process} to construct an $E$-valued stationary process~$\{\eta_k\}_{k\in\Z}$ such that the conditional law of~$\eta_1$ given the past $\{\eta_k=\xi_k,k\in\Z_-\}$ is equal to~$Q(\xxi;\cdot)$ on the support of the law of $\{\eta_k\}_{k\in\Z_-}$. Then we define a process $\eta=\{\eta(t)\}_{t\in\R}$ by the relation 
\begin{equation*}\label{definition-eta}
	\eta(t)=\eta_k(t-k)\quad\mbox{for $k\le t<k+1$, $k\in\Z$}. 
\end{equation*}
The following result follows immediately from the construction. 

\begin{proposition}\label{p-continuous-final}
Almost every trajectory of~$\eta$ is a bounded  function of time, and the following two properties hold:
	\begin{description}
		\item [\sl Periodicity.] The law of~$\eta$ regarded as a probability measure on the separable Fr\'echet space $L_{\mathrm{loc}}^2(\R)$ is invariant under the time-$1$ shift. 
		\item [\sl Regularity and recurrence.] The restrictions of~$\eta$ to the intervals~$\{J_k\}_{k\in\Z}$ form a stationary process in~$E$ that satisfy Hypotheses~\hyperlink{Fel}{\rm(SF)}, \hyperlink{DLP'}{\rm(DLP$'$)}, and~\hyperlink{SRZ}{\rm(SRZ)}, where $F_n=\lspan\{\varphi_1,\dots,\varphi_n\}$. 
	\end{description}
\end{proposition}

Finally, we construct a continuous-time process that takes values in an arbitrary separable Hilbert space~$H$ and satisfies Hypotheses~\hyperlink{Fel}{\rm(SF)}, \hyperlink{DLP'}{\rm(DLP$'$)}, and~\hyperlink{SRZ}{\rm(SRZ)}, and  the support~$\KK$ of the law of its restriction to $J_1$ has a vector span dense in $E=L^2(J_1, H)$.  To this end, we denote by~$\{e_j\}$ an orthonormal basis in~$H$, choose independent identically distributed processes~$\{\eta^j(t),t\in\R\}$ satisfying the conclusions of Proposition~\ref{p-continuous-final}, and consider the process
\begin{equation}\label{eta-CT}
	\eta(t)=\sum_{j=1}^\infty b_j\eta^j(t)e_j,\quad t\in\R,
\end{equation}
where $\{b_j\}$ are non-zero numbers satisfying the inequality $\sum_jb_j^2<\infty$. The following result is a consequence of Proposition~\ref{p-continuous-final} and the observation that a conditional measure of a tensor product is the tensor product of conditional measures. 

\begin{proposition}\label{p-CT-process}
	Almost every trajectory of the process~$\eta(t)$ constructed above is a bounded function of time with range in~$H$. Moreover, the two properties formulated in Proposition~\ref{p-continuous-final} remain valid if we replace~$L_{\mathrm{loc}}^2(\R)$ with $L_{\mathrm{loc}}^2(\R;H)$, $E=L^2(J)$ with $L^2(J;H)$,	and choose $F_n=\lspan\{\varphi_l \otimes e_m, l+m \le n\}$. 
\end{proposition}

\section{Application to randomly forced ODE and PDE}
\label{s-applications}
Theorem~\ref{t-mixing-dissipation} directly applies to systems~\eqref{01}. In this section, we give three examples which illustrate its applications to systems~\eqref{02}.

\subsection{Chain of anharmonic oscillators coupled to heat reservoirs}
\label{s-ode}
Our first example deals with an application of the theorem to a chain of anharmonic oscillators coupled to heat reservoirs. Following~\cite{JP-1998,EPR-1999}, we fix a smooth Hamiltonian  of the form~\eqref{hamiltonian} for a ``small system'' and, after excluding the variables describing the reservoirs, write the equations of motion for the coupled system in the form\footnote{We confine ourselves to a ``linear'' chain coupled to reservoirs at the endpoints. It is not difficult to extend our analysis to more complicated geometries, similar to those studied in~\cite{CEHR-2018,raquepas-2019}.}
\begin{equation}\label{coupled-system}
	\dot q=p, \quad \dot p =-\nabla V(q)+(-\gamma_1p_1+\zeta_1(t))e_1+(-\gamma_np_n+\zeta_n(t))e_n,
\end{equation}
where $\{e_j, 1\le j\le n\}$ is the standard basis in~$\R^n$, $\gamma_1, \gamma_n$ are positive numbers, and $\zeta_1,\zeta_n$ are stochastic processes with locally square-integrable trajectories. Thus, the dissipation~$\gamma_ip_i$ and stochastic force~$\zeta_i$ enter only the equations for~$p_1$ and~$p_n$. We shall prove that the question of mixing for~\eqref{coupled-system} can be reduced to a pure control problem for a family of linear ODEs with variable coefficients. The latter can be studied with the help of well-known methods of control theory and will be investigated in a subsequent publication (see also Example~\ref{e-hamiltonian}).  

In what follows, we assume that~$V$ is twice continuously differentiable and  denote by~$v(p,q)$ the vector field entering~\eqref{coupled-system}:
\begin{equation*}\label{vector-field}
v(p,q)=(p,-\nabla_q V(q)-\gamma_1p_1e_1-\gamma_np_ne_n).
\end{equation*}
A simple calculation shows that if $(p(t), q(t))$ is a trajectory for~\eqref{coupled-system}, then 
\begin{equation}\label{dHpq}
\frac{\dd}{\dd t}H\bigl(p, q\bigr)
=-\gamma_1p_1^2-\gamma_np_n^2+\bigl(\zeta_1p_1+\zeta_np_n\bigr)
\le C\bigl(\zeta_1^2+\zeta_n^2\bigr),	
\end{equation}
where $H(p,q)=\sum_jp_j^2/2+V(q)$. Since~$\zeta_1$ and~$\zeta_n$ are locally square-integrable, we see that $H(p(t), q(t))$ does not explode in finite time. Therefore, assuming that 
\begin{equation*}\label{coercive}
	V(q)\to+\infty\quad\mbox{as}\quad  |q|\to+\infty,
\end{equation*}
we can conclude that all the solutions are defined for all $t\in\R$. We shall denote by $\varPhi^t(\cdot;\zeta):\R^{2n}\to\R^{2n}$ the map taking an initial condition $(p^0,q^0)$ to the value of the corresponding solution of~\eqref{coupled-system} at time~$t\ge0$, and by~$\overline{B}(r)$ the closed ball in~$\R^{2n}$ of radius~$r$ centred at zero. Let us impose, in addition, the following two hypotheses on the system which we examine:

\begin{description}\sl
\item [\hypertarget{BAS}{(BAS) Bounded absorbing set.}] For any $M>0$, there is $\rho>0$ such that, for every $r>0$, we can find $T_r>0$ satisfying the following property: if a measurable function $\zeta=(\zeta_1,\zeta_n):\R_+\to\R^2$ is such that 
	\begin{equation}\label{L2bound}
	\sup_{t\ge0}\int_t^{t+1}\bigl(|\zeta_1(s)|^2+|\zeta_n(s)|^2\bigr)\,\dd s\le M,
	\end{equation}
then for each $(p^0,q^0)\in \overline{B}(r)$ we have $\varPhi^t(p^0,q^0;\zeta)\in \overline{B}(\rho)$ if~$t\ge T_r$. 
\end{description}
This property is certainly satisfied if we impose the stronger condition~\eqref{as-boundedness} on the functions~$\zeta_1,\zeta_n$ and require the vector field~$v(p,q)$ to possesses a Lyapunov function for which~\eqref{lyapunov-qualified} holds. As will be shown in Proposition~\ref{p-boundedness} below, Hypothesis~\hyperlink{BAS}{(BAS)} implies the existence of a compact set~$X\subset\R^{2n}$ that absorbs the trajectories of~\eqref{coupled-system} in finite time and is invariant under the flow with an arbitrary~$\zeta$ satisfying~\eqref{L2bound}. 

\begin{description} 
\item [\hypertarget{SE}{(SE) Stable equilibrium point.}] There is $c\in\R^n$ such that 
\begin{equation}\label{converge0}
	\sup_{(p^0,q^0)\in X}\bigl|\varPhi^t(p^0,q^0;0)-(0,c)\bigr|\to0\quad\mbox{as} \quad t\to\infty.
\end{equation}
\end{description}
Since all trajectories of system~\eqref{coupled-system} with $\zeta\equiv0$  are absorbed by~$X$ in finite time, this condition implies that~$(0,c)$ is a unique globally stable equilibrium for the flow~$\varPhi^t(\cdot,0)$. A sufficient condition under which~\hyperlink{BAS}{(BAS)} and~\hyperlink{SE}{(SE)} are fulfilled will be discussed in Example~\ref{e-hamiltonian}. The following result describes some further properties of trajectories for~\eqref{coupled-system} with a deterministic driving force~$\zeta$. 

\begin{proposition}\label{p-boundedness}
	Let us assume that Hypotheses~\hyperlink{BAS}{\rm(BAS)} and~\hyperlink{SE}{\rm(SE)} are fulfilled, and deterministic measurable functions~$\zeta_1,\zeta_n$ in~\eqref{coupled-system} satisfy~\eqref{L2bound}. Then the following properties hold.
\begin{description}
\item [\sl Invariant set.] There is a compact set $X\subset \R^{2n}$  depending only on~$M$ such that, for any $(p^0,q^0)\in X$, the trajectory of~\eqref{coupled-system} issued from~$(p^0,q^0)$ satisfies the inclusion 
\begin{equation*}\label{belongX}
	\bigl(p(t),q(t)\bigr)\in X\quad\mbox{for any $t\ge0$}.
\end{equation*}
\item [\sl Absorption.]
For any $r>0$, there is an integer $T_r>0$ such that the trajectory of~\eqref{coupled-system} issued from an initial condition $(p^0,q^0)\in \overline{B}(r)$ belongs to~$X$ for $t\ge T_r$.
\end{description}
\end{proposition}

\begin{proof}
	This result is established using a well-known argument, and we only outline it. Let us recall that~$\rho>0$ is the radius of the absorbing ball defined in~\hyperlink{BAS}{(BAS)}. We denote $T\ge1$ the minimal time such that $\varPhi^t(p^0,q^0;\zeta)\in \overline{B}(\rho)$ for any $t\ge T$, $(p^0,q^0)\in \overline{B}(\rho)$, and any pair $\zeta=(\zeta_1,\zeta_n)$ satisfying~\eqref{L2bound}. Let~$E_M$ be the set of functions $\zeta\in L^2(0,T;\R^2)$ for which inequality~\eqref{L2bound} holds with the supremum taken over $t\in[0,T-1]$. We endow~$E_M$ with the weak topology and consider the map $(t,p^0,q^0,\zeta)\mapsto \varPhi^t(p^0,q^0;\zeta)$ acting from the compact set $[0,T]\times \overline{B}(\rho)\times E_M$ to the space~$\R^{2n}$. Standard theory of ODEs implies that this is a continuous map, so that the set
	$$
	X:=\bigcup_{t=0}^{T}\bigcup_{\zeta\in E_M}\varPhi^t\bigl(\overline{B}(\rho);\zeta\bigr)
	$$
	is also compact. Moreover, it contains the absorbing ball $\overline{B}(\rho)$, and the very definition of~$X$ easily implies that~$X$ is invariant under the flow~$\varPhi^t(\cdot,\zeta)$ for any choice of~$\zeta$ satisfying~\eqref{L2bound}.
\end{proof}

To formulate the main result of this subsection, we define $E=L^2(J,\R^2)$, where $J=[0,1)$, and impose the following two hypotheses on the random process~$\zeta=(\zeta_1,\zeta_n)$ (whose trajectories are assumed to be elements of~$L_{\rm{loc}}^2(\R,\R^2)$ almost surely). 

\begin{description}\sl 
	\item [\hypertarget{CP}{(CP) Compactness and periodicity}.] The law of~$\zeta$ is invariant under the time-$1$ shift, and the support of its projection to~$J$ is a compact subset $\KK\subset E$. 
	\item [\hypertarget{DR}{(DR) Decomposability and recurrence}.] The stationary process $\{\eta_k\}$ in~$E$ defined by $\eta_k(t)=\zeta(t+k-1)$ for $t\in J$ satisfies Hypothesis~\hyperlink{SRZ}{\rm(SRZ)}, in which~$\EE\subset\bKK:=\KK^{\Z_-}$ is the support of the law of~$\{\eta_l,l\in\Z_-\}$. Moreover, the conditional law $Q(\xxi;\cdot)$ of~$\eta_1$, given the past $\{\eta_l=\xi_l, l\in\Z_-\}$, is such that~\hyperlink{Fel}{\rm(SF)} and~\hyperlink{DLP'}{\rm(DLP$'$)} hold.
\end{description}
We also need a condition on the linearisation of Eqs.~\eqref{coupled-system} about its solutions~$(p(t),q(t))$, 
\begin{equation}\label{coupled-linearisation}
\dot x=y, \quad \dot y =-(D^2V)(q(t))x+(-\gamma_1y_1+\xi_1(t))e_1+(-\gamma_ny_n+\xi_n(t))e_n,
\end{equation}
supplemented with the zero initial condition:
\begin{equation}\label{IC-coupled-linearised}
	x(0)=0, \quad y(0)=0. 
\end{equation}
Here $(p(t),q(t))$ is a solution issued from the set~$X$, which was constructed in Proposition~\ref{p-boundedness}. We thus obtain a family of linear ODEs, parametrised by the initial condition $(p^0,q^0)$ and the (deterministic) right-hand side~$\zeta$ entering~\eqref{coupled-system}. Writing $\xi=(\xi_1,\xi_n)$, we impose the following hypothesis. 

\begin{description}\sl 
	\item [\hypertarget{LC}{\bf(LC) Linearised controllability}.] 
There is an open set $O\subset \R^{2n}\times E$ containing $X\times \KK$ such that, for any $((p^0,q^0),\zeta)\in O$, the associated linear system~\eqref{coupled-linearisation}, \eqref{IC-coupled-linearised} in~$\R^{2n}$ is controllable at time $t=1$ with control $\xi\in E$. 
\end{description}
Before formulating the main result of this subsection, let us mention that Hypotheses~\hyperlink{CP}{(CP)} and~\hyperlink{DR}{(DR)} are satisfied, for example, for random forces described in Subsection~\ref{ss-CTN} (see Corollary~\ref{c-CT-example} and Proposition~\ref{p-continuous-final}), whereas the controllability property~\hyperlink{LC}{(LC)} requires a separate study. A simple system for which the latter holds is given in Example~\ref{e-hamiltonian}. 

\begin{theorem}\label{t-CO}
	Let Hypotheses~\hyperlink{BAS}{\rm(BAS)}, \hyperlink{SE}{\rm(SE)}, \hyperlink{CP}{\rm(CP)}, \hyperlink{DR}{\rm(DR)}, and~\hyperlink{LC}{\rm(LC)} be fulfilled. Then the random flow of~\eqref{coupled-system} restricted to integer times is exponentially mixing in the dual-Lipschitz metric. More precisely, there is a measure $\mu\in\PP(\R^{2n})$ with compact support and a number~$\gamma>0$ such that, for any $r>0$ and a sufficiently large $C_r>0$, we have 
	\begin{equation}\label{oscillator-convergence}
	\sup_{(p^0,q^0)\in \overline{B}(r)}	\|\mu_t(p^0,q^0)-\mu\|_L^*\le C_re^{-\gamma t}, \quad t\in\Z_+,
	\end{equation}
	where we denote by~$\mu_t(p^0,q^0)$ the law at time~$t$ of the trajectory of~\eqref{coupled-system}, issued from $(p^0,q^0)\in\R^{2n}$.
\end{theorem}

\begin{proof}
We claim that Theorem~\ref{t-mixing-dissipation} is applicable to the random flow of~\eqref{coupled-system}, restricted to integer times and the invariant subset~$X$. Indeed, \hyperlink{GD}{(GD)} is a consequence of~\eqref{converge0}, and~\hyperlink{ALC}{(ALC)} and~\hyperlink{SRZ}{(SRZ)} are postulated in~\hyperlink{LC}{(LC)} and~\hyperlink{DR}{(DR)}, respectively. The latter condition also implies that the properties of Hypothesis~\hyperlink{DLP'}{(DLP$'$)} are satisfied for $F_n=E$. Thus, all the conditions of Theorem~\ref{t-mixing-dissipation} holds, so that we have inequality~\eqref{oscillator-convergence}, in which~$\overline{B}(r)$ is replaced with~$X$. Finally, the validity of~\eqref{oscillator-convergence} for an arbitrary~$r>0$ follows from the fact that~$X$ absorbs all the trajectories issued from~$\overline{B}(r)$ at some positive time~$T_r$; see the second point of Proposition~\ref{p-boundedness}. This completes the proof of the theorem.
\end{proof}

Theorem~\ref{t-CO} ensures the convergence of the restrictions of trajectories to integer times. It is not difficult to prove that, for any $s\in(0,1)$, the restrictions of trajectories to the times $s+\Z_+$ also converge to a limiting measure~$\mu_s\in\PP(\R^{2n})$. Since the corresponding argument is well known and is exactly the same as in the case of the Navier--Stokes system, we discuss it briefly in the next subsection. We conclude with an example for which Hypotheses~\hyperlink{BAS}{(BAS)} \hyperlink{CL}{(CL)}, and~\hyperlink{SE}{(SE)} are satisfied. 

\begin{example}\label{e-hamiltonian}
Let us assume that the potential~$V$ entering~\eqref{hamiltonian} is representable in the form 
\begin{equation}\label{V-representation}
		V(q)=\sum_{i=1}^n\bigl(a_iq_i^2+F_i(q_i)\bigr)+\sum_{i=1}^{n-1}b_i(q_i-q_{i+1})^2,\quad q\in\R^n,
\end{equation}
where $a_i$ and~$b_i$ are positive numbers, and $F_i\in C^2(\R)$ are some functions with bounded derivatives. We claim that~\hyperlink{BAS}{(BAS)} holds. To see this, we denote by~$V_0(q)$ the sum of the quadratic terms in~\eqref{V-representation} (that is, the expression in the right-hand side of~\eqref{V-representation} with $F_i\equiv0$) and write~$H_0(p,q)$ for the corresponding Hamiltonian. The damped Hamiltonian system associated  with~$H_0$ has the form 
\begin{equation}\label{linear-damped}
\dot q=p, \quad \dot p=-Dq-\gamma_1p_1e_1-\gamma_np_n e_n,	
\end{equation}
where $D>0$ is the symmetric matrix corresponding to the quadratic form~$V_0$. Using  the inequalities $\gamma_1,\gamma_n>0$, it is easy to check that $(p,q)=(0,0)$ is a stable equilibrium for~\eqref{linear-damped}; see the argument below applied in the verification of Hypothesis~\hyperlink{LC}{(LC)}. Therefore there is a positive-definite quadratic form $L_0(p,q)$ such that 
\begin{equation*}\label{nabla-L0}
\langle \nabla L_0(p,q),v_0(p,q)\rangle\le -\delta(|p|^2+|q|^2), 
\quad (p,q)\in\R^{2n},	
\end{equation*}
where $\delta>0$ is a number, and~$v_0$ is the linear vector field in~\eqref{linear-damped}. Since the difference $V-V_0$ has bounded first-order partial derivatives, it follows that 
\begin{equation*}\label{nablaL0vbound}
\langle \nabla L_0(p,q),v(p,q)\rangle\le -\frac{\delta}{2}(|p|^2+|q|^2)	
\end{equation*}
outside a sufficiently large ball. This readily implies the existence of an absorbing ball and of an absorbing invariant set~$X\subset\R^{2n}$.

To ensure the validity of~\hyperlink{(SE}{(SE)} and~\hyperlink{LC}{(LC)}, we further assume that 
\begin{equation}\label{lower-bound-Fi}
F_i'(0)=0,\quad	\inf_{s\in\R}F_i''(s)>-2a_i\quad\mbox{for any $1\le i\le n$}. 
\end{equation}
To check~\hyperlink{SE}{(SE)}, note that if $(p,q)$ is a trajectory of~\eqref{coupled-system} with $\zeta_1\equiv\zeta_n\equiv0$, 
\begin{equation}\label{homogeneous-chain}
	\dot q=p,\quad\dot p=-\nabla V(q)-\gamma_1p_1e_1-\gamma_np_ne_n, 
\end{equation}
then $H(p(t),q(t))$ is a non-decreasing function of time; see~\eqref{dHpq}. It follows that the Hamiltonian~$H$ must be constant on the $\omega$-limit set of a solution. In particular, for any trajectory~$(p,q)$ of~\eqref{homogeneous-chain} lying on such an $\omega$-limit set, we have $p_1\equiv p_n\equiv0$. We claim that $p\equiv q\equiv0$ is the only such trajectory. This will imply the validity of~\hyperlink{SE}{(SE)} by a simple compactness argument. 

The relation $\dot q_1=p_1$ (which is the first equation in~\eqref{homogeneous-chain}) implies that~$q_1$ is constant in time. The equation 
$$
0\equiv\dot p_1=-\p_{q_1}V(q)=2b_1q_2-2(a_1+b_1)q_1-F_1'(q_1)
$$ 
implies that~$q_2$ is also constant in time, so that $p_2\equiv0$. Arguing by induction, we see that all~$p_i$ vanish identically and all~$q_i$ are constant. It remains to prove that $q_i=0$. The second equation in~\eqref{homogeneous-chain} imply that 
$$
2a_iq_i+2b_{i-1}(q_i-q_{i-1})+2b_i(q_i-q_{i+1})+F_i'(q_i)=0, \quad 1\le i\le n,
$$
where $b_0=q_0=b_n=q_{n+1}=0$. A simple analysis based on~\eqref{lower-bound-Fi} shows that $q_1=\cdots=q_n=0$ is the only solution of this system. 

Finally, the verification of the linearised controllability~\hyperlink{LC}{(LC)} is based on well-known techniques in the control theory of ODEs, using the propagation of perturbation via the derivative of the potential~$V$. Since the corresponding argument is technically more involved, we shall give it in a subsequent publication devoted to a study of more general chains of oscillators and rotators.  

Thus, Theorem~\ref{t-CO} applies to the chain of anharmonic oscillators~\eqref{coupled-system} with a potential of the form~\eqref{V-representation}, provided that the random forces~$\zeta_1,\zeta_n$ are independent and belong to the class of processes described in Subsection~\ref{ss-CTN}.
\end{example}

\subsection{Two-dimensional Navier--Stokes system}
\label{s-nse}

Let us consider the Navier--Stokes system~\eqref{NS} in a bounded domain~$D\subset\R^2$ with a smooth boundary~$\p D$. Let us impose the no-slip boundary condition~\eqref{non-slip}, 
define the usual functional space 
\begin{equation}\label{space-H}
H=\{u\in L^2(D,\R^2):\mbox{$\diver u=0$ in $D$, $\langle u,\nnn\rangle=0$ on~$\p D$}\},	
\end{equation}
and denote by~$\Pi$ the orthogonal projection in $L^2(D,\R^2)$ onto~$H$. Excluding the pressure from~\eqref{NS} with the help of the standard argument, we write the problem in question as a nonlocal PDE in~$H$,
\begin{equation}\label{NS-reduced}
	\p_tu+\nu Lu+B(u)=\eta(t), \quad u(t)\in H,
\end{equation}
supplemented with the initial condition~\eqref{NS-IC}. Here $L=-\Pi\Delta$ is the Stokes operator, $B(u)=\Pi(\langle u,\nabla\rangle u)$ is a quadratic map, $u_0\in H$ is an initial function, and~$\eta$ is a bounded stationary stochastic process in~$H$ of the form~\eqref{eta-CT}, in which $\{e_j\}$ is an orthonormal basis in~$H$ composed of the eigenfunctions of~$L$,  $b_j\ne0$ are some numbers such that $\sum_jb_j^2<\infty$, and $\eta^j=\{\eta^j(t), t\in\R\}_{j\ge1}$ are i.i.d.\ stochastic processes that are almost surely bounded by~$1$ and possess the properties stated in Proposition~\ref{p-continuous-final}. 


Let us denote $J=[0,1)$ and introduce random functions~$\eta_k^j:J\to\R$ by the formula $\eta_k^j(t)=\eta^j(t+k-1)$. Writing
$$
\eeta_k=\sum_{j=1}^\infty b_j\eta_k^je_j, 
$$
we obtain a random variable in $E:=L^2(J,H)$, and the construction implies that $\{\eeta_k\}_{k\ge1}$ is a stationary sequence in~$E$. Denoting by~$\KK$ the Hilbert cube $\{v\in E:|(v,e_j)|\le b_j\mbox{ for $j\ge1$}\}$, we see that $\eeta_k\in\KK$ with probability~$1$. 

For any initial condition $u_0\in H$, Equation~\eqref{NS-reduced} has a unique solution $u(t;u_0)$ that belongs to the space $L_{\mathrm{loc}}^2(\R_+,V)\cap C(\R_+,H)$ almost surely and satisfies the initial condition~\eqref{NS-IC}. Let us denote by $\{u_k(v),k\in\Z_+\}$ the restrictions of~$u(\cdot;v)$ to integer times. The following theorem describes the asymptotic behaviour of~$u_k$ as $k\to\infty$.

\begin{theorem}\label{t-NS}
	Under the above hypotheses, there is a measure $\mu\in\PP(H)$  with compact support and a number $\gamma>0$ such that, for any $R>0$, we have
	\begin{equation}\label{EC-NS}
		\sup_{v\in B_H(R)}\bigl\|\DD\bigl(u_k(v)\bigr)-\mu\bigr\|_L^*\le C_Re^{-\gamma k}, \quad k\ge0,
	\end{equation}
	where $C_R>0$ depends only on~$R$. 
\end{theorem}

\begin{proof}
	Since the norms of trajectories of~$\eta(t)$ are bounded by a universal number for all $t\in \R$, a standard argument shows that there is a compact absorbing set $X\subset H$ that is invariant under the random flow; cf.\ \cite[Section~1.6]{BV1992}. Thus, it suffices to prove inequality~\eqref{EC-NS} in which the ball~$B_H(R)$ is replaced with~$X$. To this end, we shall use Theorem~\ref{t-mixing-dissipation}. Namely, denoting by~$S:H\times E\to H$ the time-$1$ shift along trajectories of Equation~\eqref{NS-reduced}, we see that 
	\begin{equation}\label{NS-discretised}
		u_k=S(u_{k-1},\eeta_k), \quad k\ge1. 
	\end{equation}
	Hence, it suffices to check the validity of Hypotheses~\hyperlink{GD}{(GD)}, \hyperlink{ALC}{(ALC)}, \hyperlink{DLP'}{(DLP$'$)}, and~\hyperlink{SRZ}{(SRZ)}. The first of them follows from a well-known dissipation property of the Navier--Stokes system and is valid with $k=1$. The validity of~\hyperlink{DLP'}{(DLP$'$)} and~\hyperlink{SRZ}{(SRZ)} is the content of Proposition~\ref{p-CT-process}. Finally, to check~\hyperlink{ALC}{(ALC)}, it suffices to prove that the linearised Navier--Stokes system 
	\begin{equation}
		\p_tv+\nu Lv+B(u,v)+B(v,u)=\zzeta(t), \quad v(0)=0
	\end{equation}
	is approximately controllable at time~$1$. Here $u\in \HH:=L^2(J,V)\cap C(J,H)$	is a reference trajectory of~\eqref{NS-reduced} associated with an initial condition $u_0\in X$ and a right-hand side~$\eta\in\bKK$. The required property follows from the observations that the control~$\zzeta$ has a full range in~$H$ and that the reference trajectories~$u$ are taken from a bounded ball in~$\HH$.  This completes the proof of the theorem. 
	\end{proof}

Theorem~\ref{t-NS} proves the convergence of solutions for the Navier--Stokes system as the times goes to~$+\infty$ along integer values. Let us show that, for any $s\in(0,1)$, the measures~$\DD(u(s+k;v))$ exponentially converge to a limit~$\mu_s$ in the dual-Lipschitz metric. Indeed, let $\{(v_k,\xxi_k),k\in\Z_+\}$ be a stationary weak solution of~\eqref{NS-discretised} (see Section~\ref{ss-RDS}). Then $\DD(v_k)\equiv\mu$, and the process $\{\xxi_k,k\in\Z_+\}$ is distributed as $\{\eeta_k,k\in\Z_+\}$. Let $\xi(t)$ be the corresponding $H$-valued process (so that $\xi(t)=\xxi_k(t-k+1)$ for $t\in[k-1,k)$) and let~$w(t)$, $t\ge0$ be a solution of~\eqref{NS-reduced} with~$\eta$ replaced by~$\xi$ such that $w(0)=v_0$. Then $w(k)=v_k$ for $k\in\Z_+$. For $v\in H$, let $u^\xi(t;v)$ be a solution of~\eqref{NS-reduced} with $\eta\equiv\xi$ such that $u^\xi(0;v)=v$. Then $\DD(u^\xi(t;v))=\DD(u(t;v))$, where, as before, $u(t;v)$ solves~\eqref{NS-reduced} and equals~$v$ at $t=0$. Relation~\eqref{EC-NS} combined with a simple result on  preservation of the weak convergence of measures under Lipschitz maps (e.g., see \cite[Lemma~3.2]{shirikyan-jems2021})  implies that
\begin{equation*}
		\sup_{v\in B_H(R)}\bigl\|\DD(u(t;v))-\DD(w(t))\bigr\|_L^*\le C_R'e^{-\gamma t}, \quad t\ge0,
\end{equation*}
where $C_R'>0$ depends only on~$R$. Obviously, for $t=s+k$ with $k\in\Z_+$ and $0<s<1$, the measures $\DD(w(s+k))$ do not depend on~$k$. Denoting that measure by~$\mu_s\in\PP(H)$, we obtain the exponential convergence of $\DD(u(s+k;v))$ to it as $k\to\infty$. 

In conclusion, we note that Theorem~\ref{t-NS} remains valid for a kick force of the form~\eqref{kick-force}, provided that the stationary random process~$\{\eta_k\}_{k\in\Z}$ satisfies appropriate non-degeneracy and decomposability conditions. The corresponding result is similar to that discussed in the next subsection for the Ginzburg--Landau equation. 

\subsection{Complex Ginzburg--Landau equation}
\label{s-cgl}
To simplify the presentation, we shall confine ourselves to the case of a kick force. However, the result remains valid for continuous-time forces of the form~\eqref{eta-CT} under the same conditions as in Section~\ref{s-nse}.

Let $D\subset \R^d$, $1\le d\le 4$ be a bounded domain with a smooth boundary. We consider the PDE
\begin{equation}\label{cgl}
	\p_tu-(\nu+i)\Delta u+ic|u|^{2s}u=\eta(t,x), \quad x\in D,
\end{equation}
supplemented with the boundeaty condition~\eqref{non-slip}. Here~$\nu$ and~$c$ are positive numbers, and~$s\ge1$ is an integer such that $s\le \frac{2}{d-2}$ for $d=3\mbox{ or }4$. In this case, problem~\eqref{cgl}, \eqref{non-slip} with $\eta\equiv0$ is well posed in the Sobolev space~${\mathscr H}:=H_0^1(D,\C)$, regarded as a real Hilbert space with the inner product 
$$
(u,v)=\Re\int_D\langle \nabla u,\nabla\bar v\rangle\dd x.
$$
The corresponding phase flow in~${\mathscr H}$ possesses a Lyapunov function,
\begin{equation}\label{integrals}
\HH(u)=\int_D\Bigl(\frac12|\nabla u|^2+\frac{c}{2s+2}|u|^{2s+2}\Bigr)\dd x,
\end{equation}
which satisfies the inequality
\begin{align}
\frac{\dd}{\dd t}\HH(u(t))&\le -\nu\|\Delta u\|^2-\nu c\bigl(|u|^{2s},|\nabla u|^2\bigr),\label{lyapunov-cgl}
\end{align}
where $(\cdot,\cdot)$ and  $\|\cdot\|$ stand for the $L^2$ real inner product and the corresponding norm in~$L^2(D;\C)$. Concerning the random force~$\eta$, we assume that it has the form~\eqref{kick-force}, in which~$\{e_j\}$ is the complete set of $L^2$-normalised eigenfunctions of the Dirichlet Laplacian $-\Delta$ corresponding to the eigenvalues $\alpha_j$ indexed in the increasing order, $\eta_k^j$ are real-valued random variables such that $|\eta_k^j|\le1$ with probability~$1$, and~$\{b_j\}$ are positive numbers such that 
\begin{equation*}\label{finite-B1}
	\sum_{j=1}^\infty \alpha_jb_j^2<\infty. 
\end{equation*}
These conditions ensure that~$\{\eta_k\}$ are ${\mathscr H}$-valued random variables whose law is supported by a compact subset $\KK\subset {\mathscr H}$. Thus, denoting by $S:{\mathscr H}\to {\mathscr H}$ the time-$1$ shift along trajectories of~\eqref{cgl} with~$\eta\equiv0$ and writing $u_k=u(k)$, we obtain the relation
\begin{equation}\label{rds-cgl}
	u_k=S(u_{k-1})+\eta_k, \quad k\ge1. 
\end{equation}
This is a particular case of the RDS~\eqref{stationary-RDS}, in which $H=E={\mathscr H}$. The following theorem provides sufficient conditions under which the results of Section~\ref{s-MR} can be applied to~\eqref{rds-cgl}.

\begin{theorem}\label{t-cgl}
In addition to the above conditions, assume that~$\{\eta_k\}_{k\in\Z}$ is a stationary random process that satisfies  Hypotheses~\hyperlink{Fel}{\rm(SF)},  \hyperlink{SRZ}{\rm(SRZ)}, and~\hyperlink{DLP'}{\rm(DLP$'$)} with $F_n=\lspan\{e_j,1\le j\le n\}$ and $F_n^\dagger=F_n^\bot$, where the orthogonal complement is taken in~${\mathscr H}$. Then there is a measure $\mu\in \PP({\mathscr H})$ with compact support and a number $\gamma>0$ such that, for any $R>0$, inequality~\eqref{EC-NS} holds with $H={\mathscr H}$ and a constant~$C_R$ depending only on~$R$. 
\end{theorem}

\begin{proof}
	We apply Theorem~\ref{t-mixing-dissipation}. The validity of Hypotheses~\hyperlink{DLP'}{(DLP$'$)} and~\hyperlink{SRZ}{(SRZ)} are postulated in the statement of the theorem, and~\hyperlink{GD}{(GD)} follows immediately from~\eqref{lyapunov-cgl} and the fact that~$\HH(u)$ can be estimated from above and from below by a power of the  $H^1$-norm. To check~\hyperlink{ALC}{(ALC)}, note that the derivative of the operator $S(u,\eta)=S(u)+\eta$ with respect to~$\eta\in {\mathscr H}$ is the identity operator in~${\mathscr H}$, so that its image is equal to~${\mathscr H}$. This completes the proof of the theorem.
\end{proof}

\section{Appendix: the image of transition probabilities under Lipschitz maps}
\label{s-transformation}
Let~$E$ be a separable Banach space represented as the direct sum of its closed subspaces~$F$ and~$F^\dagger$ such that $\dim F<\infty$, and let~$\KK\subset E$ be a compact subset. We consider a map $\varPsi:\KK\to E$ of the form $\varPsi(\eta)=\eta+\varPhi(\eta)$, where $\varPhi:\KK\to E$ is a continuous map whose image is included in~$F$, and for any $\eta,\eta'\in\KK$,  we have
\begin{align}
	\|\varPhi(\eta)\|_E\le\varkappa, \quad 
	\|\varPhi(\eta)-\varPhi(\eta')\|_E\le\varkappa\,\|\eta-\eta'\|_E,\label{lipschitz-property}
\end{align}
where $\varkappa>0$ is a number not depending on~$\eta$ and~$\eta'$. Let $\PPP(U;\cdot)$ be a transition probability with an underlying compact metric space~$\XXXX$. 

\begin{proposition}\label{p-image-measure}
	Suppose that $\PPP(U;\cdot)$ satisfies Hypothesis~\hyperlink{(DLP)}{\rm(DLP)}. Then there is~$C>0$ not depending on~$U$ and~$\varkappa$ such that 
	\begin{equation}\label{estimate-image}
		\bigl\|\PPP(U;\cdot)-\varPsi_*\bigl(\PPP(U;\cdot)\bigr)\bigr\|_{\mathrm{var}}\le C\varkappa\quad\mbox{for any $U\in\XXXX$}. 
	\end{equation}
\end{proposition}

\begin{proof}
	This result is essentially a simpler version of Theorem~2.4 in~\cite{KNS-gafa2020}, and we confine ourselves to outlining the main steps. 
	
	Since the total variation distance is bounded by~$1$, there is no loss of generality in assuming that $\varkappa\le\frac{1}{2}$. We need to prove that if~$f$ is the indicator function of a Borel set in~$E$, then the expression $|\langle f\circ\varPsi,\PPP(U;\cdot)\rangle-\langle f,\PPP(U;\cdot)\rangle|$ is bounded above by the right-hand side of~\eqref{estimate-image} with a number~$C$ not depending on~$U$ and~$f$. 
	
	To this end, we use~\eqref{decomposition-E-E} to write
\begin{align*}
	\langle f\circ\varPsi,\PPP(U;\cdot)\rangle &=
	\int_{F^\dagger}\PPP_{F^\dagger}(U,\dd \eta)\int_{F}f\bigl(\xi+\eta+\varPhi(\xi+\eta)\bigr)p_F(U,\eta;\xi)\,\ell_F(\dd\xi),\\
	\langle f,\PPP(U;\cdot)\rangle &=
	\int_{F^\dagger}\PPP_{F^\dagger}(U,\dd \eta)\int_{F}f(\xi+\eta)p_F(U,\eta;\xi)\,\ell_F(\dd\xi).
\end{align*}
Note that the integrations are carried out over the projections of~$\KK$ to~$F^\dagger$ and~$F$ (which are compact subsets), since the supports of the measures $\PPP(U;\cdot)$ are contained in~$\KK$. The required estimate will be established if we prove that 
\begin{equation}\label{int<kappa}
	\biggl|\int_{F}f(\xi+\eta+\varPhi(\xi+\eta))p_F(U,\eta;\xi)\,\ell_F(\dd\xi)-\int_{F}f(\xi+\eta)p_F(U,\eta;\xi)\,\ell_F(\dd\xi)\biggr|\le C\varkappa. 
\end{equation}
To this end, we wish to make the change of variable $\xi'=\xi+\varPhi(\xi+\eta)$ in the first integral. However, the map $\varPhi$ is defined only on~$\KK$. To overcome this difficulty, we first use the Kirzsbraun--McShane theorem (see~\cite[Theorem~6.1.1]{dudley2002}) to extend~$\varPhi$ to the whole space~$E$ in such a way that the extended map (for which we keep the same notation) is also $\varkappa$-Lipschitz. In this case, for any $\eta\in F^\dagger$, the map $\xi\mapsto \xi+\varPhi(\xi+\eta)$ is a bi-Lipschitz homeomorphism of the space~$F$ onto itself, and we denote by $\Theta_\eta(\cdot)$ its inverse. The latter and its differential (which exists almost everywhere in view of the Rademacher theorem; see~\cite[Section~3.1.6]{federer1969})  can be written in the form 
\begin{equation}\label{inverse-diff}
\Theta_\eta(\xi')=\xi'-\varPhi(\Theta_\eta(\xi')+\eta), \quad D\Theta_\eta(\xi')=\bigl(\Id{F}+D\varPhi(\Theta_\eta(\xi')+\eta)\bigr)^{-1}.
\end{equation}
Using Theorem~3.2.5 in~\cite{federer1969} to perform the change of variable $\xi=\Theta_\eta(\xi')$, we can write the first integral in~\eqref{int<kappa} as 
$$
\int_{F}f(\xi+\eta+\varPhi(\xi+\eta))p_F(U,\eta;\xi)
=\int_{F}f(\xi'+\eta)\frac{p_F(U,\eta;\Theta_\eta(\xi'))}{\det\bigl(\Id{F}+D\varPhi(\Theta_\eta(\xi')+\eta)\bigr)}, 
$$ 
where we dropped the Lebesgue measure $\ell_F$ to shorten the formula. Thus, the left-hand side~$\delta(U)$ of~\eqref{int<kappa} satisfies the inequality
\begin{equation}\label{delta-U}
	\delta(U)
\le \int_{F}\biggl|\frac{p_F(U,\eta;\Theta_\eta(\xi'))}{\det\bigl(\Id{F}+D\varPhi(\Theta_\eta(\xi')+\eta)\bigr)}-p_F(U,\eta;\xi)\biggr|\,\ell_F(\dd\xi),
\end{equation} 
where we used the inequality $0\le f\le1$. It follows from~\eqref{inverse-diff} and~\eqref{lipschitz-property} that~$\Theta_\eta$ is a $2$-Lipschitz function such that $|\Theta_\eta(\xi')-\xi'|\le\varkappa$. Combining this with the Lipschitz property of~$p_F$ (see~\eqref{Lip}), we can easily show that 
\begin{gather*}
	p_F(U,\eta;\Theta_\eta(\xi))
	=p_F(U,\eta;\xi)+q(U,\xi,\eta),\\
	\det\bigl(\Id{F}+D\varPhi(\Theta_\eta(\xi)+\eta)\bigr)^{-1}=1+r(\xi,\eta),
\end{gather*}
where $q$ and $r$ are some functions satisfying the inequality
$$
|q(U,\xi,\eta)|+|r(\xi,\eta)|\le C_1\varkappa
$$
with a constant $C_1>0$ not depending on $U$, $\xi$, $\eta$, and~$\varkappa$. Combining this with~\eqref{delta-U}, we arrive at the required inequality~\eqref{int<kappa}. 
\end{proof}

\addcontentsline{toc}{section}{Bibliography}
\newcommand{\etalchar}[1]{$^{#1}$}
\def\cprime{$'$} \def\cprime{$'$}
  \def\polhk#1{\setbox0=\hbox{#1}{\ooalign{\hidewidth
  \lower1.5ex\hbox{`}\hidewidth\crcr\unhbox0}}}
  \def\polhk#1{\setbox0=\hbox{#1}{\ooalign{\hidewidth
  \lower1.5ex\hbox{`}\hidewidth\crcr\unhbox0}}}
  \def\polhk#1{\setbox0=\hbox{#1}{\ooalign{\hidewidth
  \lower1.5ex\hbox{`}\hidewidth\crcr\unhbox0}}} \def\cprime{$'$}
  \def\polhk#1{\setbox0=\hbox{#1}{\ooalign{\hidewidth
  \lower1.5ex\hbox{`}\hidewidth\crcr\unhbox0}}} \def\cprime{$'$}
  \def\cprime{$'$} \def\cprime{$'$} \def\cprime{$'$}
\providecommand{\bysame}{\leavevmode\hbox to3em{\hrulefill}\thinspace}
\providecommand{\MR}{\relax\ifhmode\unskip\space\fi MR }
\providecommand{\MRhref}[2]{%
  \href{http://www.ams.org/mathscinet-getitem?mr=#1}{#2}
}
\providecommand{\href}[2]{#2}

\end{document}